\documentclass[12pt]{scrartcl}

\usepackage{graphicx}
\usepackage{enumerate}
\usepackage{enumitem}

\usepackage{amssymb}
\usepackage{amsmath}
\usepackage{amsthm}
\usepackage{dsfont}
\usepackage{bm}

\usepackage{subfig}

\usepackage[round,comma]{natbib}

\usepackage{hyperref}
\hypersetup{
	colorlinks   = true,
	urlcolor     = LMUpurple,
	linkcolor    = LMUgreen,
	citecolor   = LMUgreen
}
\usepackage[capitalize,nameinlink]{cleveref}

\usepackage{xcolor}

\definecolor{LMUblue}{RGB}{0, 17, 88}
\definecolor{LMUlightblue}{RGB}{92,177,235}
\definecolor{LMUgreen}{RGB}{0,136,58}
\definecolor{LMUlightgreen}{RGB}{170, 173, 0}
\definecolor{LMUred}{RGB}{190,25,8}
\definecolor{LMUpurple}{RGB}{176, 32, 121}
\definecolor{LMUorange}{RGB}{241, 135, 0}
\definecolor{bggray}{RGB}{245, 245, 245}
\definecolor{LMUgray}{RGB}{0.2,0.2,0.2}%

\usepackage{shadethm}
\usepackage{etoolbox}

\newtheoremstyle{new}
  {12pt}      %
  {12pt}      %
  {\itshape}  %
  {}          %
  {\bfseries\color{black}} %
  {.}         %
  { }         %
  {}          %
\theoremstyle{new}
\newtheorem{theorem}{Theorem}[section]
\newtheorem{corollary}[theorem]{Corollary}
\newtheorem{proposition}[theorem]{Proposition}
\newtheorem{lemma}[theorem]{Lemma}
\newtheorem{definition}[theorem]{Definition}

\newtheorem{example}[theorem]{Example}
\newtheorem{remark}[theorem]{Remark}

\newtheorem*{theorem*}{Theorem}
\newtheorem*{example*}{Example}
\newtheorem*{definition*}{Definition}
\newtheorem*{lemma*}{Lemma}
\newtheorem*{proposition*}{Proposition}
\newtheorem*{corollary*}{Corollary}

\AfterEndEnvironment{theorem}{\noindent\ignorespaces}
\AfterEndEnvironment{corollary}{\noindent\ignorespaces}
\AfterEndEnvironment{proposition}{\noindent\ignorespaces}
\AfterEndEnvironment{lemma}{\noindent\ignorespaces}
\AfterEndEnvironment{definition}{\noindent\ignorespaces}
\AfterEndEnvironment{assumption}{\noindent\ignorespaces}
\AfterEndEnvironment{example}{\noindent\ignorespaces}
\AfterEndEnvironment{remark}{\noindent\ignorespaces}
\AfterEndEnvironment{claim}{\noindent\ignorespaces}

\Crefname{theorem}{Theorem}{Theorems}
\Crefname{corollary}{Corollary}{Corollaries}
\Crefname{proposition}{Proposition}{Propositions}
\Crefname{lemma}{Lemma}{Lemmas}
\Crefname{definition}{Definition}{Definitions}
\Crefname{example}{Example}{Examples}
\Crefname{remark}{Remark}{Remarks}
\Crefname{claim}{Claim}{Claims}

\usepackage[breakable, theorems, skins]{tcolorbox}
\tcbset{enhanced}

\definecolor{shadethmcolor}{cmyk}{0,0,0,0.075}    
\definecolor{shaderulecolor}{rgb}{1,1,1}   %

\newtheoremstyle{shad}
  {12pt}      %
  {12pt}      %
  {\itshape }  %
  {}          %
  {\bfseries\color{black}} %
  {.}         %
  { }         %
  {}          %
\theoremstyle{shad} 

\newshadetheorem{thmbox}[theorem]{Theorem}
\newshadetheorem{corbox}[theorem]{Korollar}
\newshadetheorem{propbox}[theorem]{Proposition}
\newshadetheorem{lembox}[theorem]{Lemma}
\newshadetheorem{exbox}[theorem]{Beispiel}
\newshadetheorem{defbox}[theorem]{Definition}
\newshadetheorem{rembox}[theorem]{Anmerkung}

\newtheoremstyle{shad*}
  {12pt}      %
  {12pt}      %
  {\itshape }  %
  {}          %
  {\bfseries\color{black}} %
  {.}         %
  { }         %
  {}          %
\theoremstyle{shad*} 
\newshadetheorem{anobox}{}

\Crefname{theorem}{Theorem}{Theoreme}
\Crefname{corbox}{Korollar}{Korollare}
\Crefname{propbox}{Proposition}{Propositionen}
\Crefname{lembox}{Lemma}{Lemmas}
\Crefname{exbox}{Beispiel}{Beispiele}
\Crefname{defbox}{Definition}{Definitionen}
\Crefname{rembox}{Anmerkung}{Anmerkungen}

\usepackage{bm}

\newcommand{\ind}{\mathds{1}}
\newcommand{\R}{\mathds{R}}
\newcommand{\N}{\mathds{N}}

\providecommand{\P}{}
\renewcommand{\P}{\mathds{P}}
\newcommand{\Z}{\mathds{Z}}
\newcommand{\E}{\mathds E}

\newcommand{\G}{\mathbb{G}}
\newcommand{\D}{\mathds{D}}

\newcommand{\Ccal}{\mathcal{C}}

\newcommand{\Fcal}{\mathcal{F}}
\newcommand{\Gcal}{\mathcal{G}}

\newcommand{\Scal}{\mathcal{S}}

\newcommand{\Vcal}{\mathcal{V}}
\newcommand{\Wcal}{\mathcal{W}}
\newcommand{\Xcal}{\mathcal{X}}

\newcommand{\eps}{\varepsilon}

\newcommand{\BL}{\operatorname{BL}}

\renewcommand{\bar}{\overline}

\providecommand{\Pr}{}
\renewcommand{\Pr}{\mathbb{P}}
\newcommand{\var}{{\mathds{V}\mathrm{ar}}}

\newcommand{\cov}{{\mathds{C}\mathrm{ov}}}

\newcommand{\wh}[1]{\widehat{#1}}

\newcommand{\dc}{\to_d}
\newcommand{\rd}{\leftrightarrow_d}

\newcommand{\sumin}{\sum_{i = 1}^n}

\newcommand{\fig}[1][1]{
  \includegraphics[width = #1\textwidth]
}

\makeatletter
\def\blfootnote{\gdef\@thefnmark{}\@footnotetext}
\makeatother

\setenumerate[1]{label={(\roman*)}} %

\title{Uniform central limit theorems for non-stationary processes via relative weak convergence}
\author{Nicolai Palm and Thomas Nagler}

\begin{document}

\maketitle

\begin{abstract}
	Statistical inference for non-stationary data is hindered by the failure of classical central limit theorems (CLTs), 
	not least because there is no fixed Gaussian limit to converge to.	
	To resolve this, we introduce \emph{relative weak convergence}, 
	an extension of weak convergence that compares a statistic or process to a sequence of evolving processes.
	Relative weak convergence retains the essential consequences of classical weak convergence and coincides with it under stationarity.
	Crucially, it applies in general non-stationary settings where classical weak convergence fails.
	We establish concrete relative CLTs for random vectors and empirical processes, 
	along with sequential, weighted, and bootstrap variants, that parallel the state-of-the-art in stationary settings. 	
	Our framework and results offer simple, plug-in replacements for classical CLTs whenever stationarity is untenable, 
	as illustrated by applications in nonparametric trend estimation and hypothesis testing.

\end{abstract}
\tableofcontents
\newpage
\section{Introduction}
This paper studies uniform central limit theorems (CLTs) for empirical processes 
arising from non-stationary triangular arrays. 
Let $X_{n,1},\ldots, X_{n,k_n}$ be stochastic processes with bounded sample paths, indexed by a set $T$,
and define the empirical process 
$$\G_n\in \ell^\infty(T), \quad \G_n(t)=\frac{1}{\sqrt{k_n}}\sum_{i=1}^{k_n}\left(X_{n,i}(t)-\E[X_{n,i}(t)]\right)$$
where $\ell^\infty(T)$ denotes the space of bounded functions $T\to \R$.
We allow the laws of $X_{n,i}$ to vary in both $n$ and $i$, thereby 
covering genuinely non-stationary regimes in which even the covariance structure
of $\G_n$ need not converge. 
In this setting, a classical weak limit of $\G_n$ need not exist.
We introduce a framework that provides a systematic asymptotic theory
for $\G_n$ in such regimes, together with practical tools for statistical inference.

Originating with Donsker's theorem, classical results establish weak convergence of 
$\G_n$ to a tight Gaussian limit 
$\G\in \ell^\infty(T)$, 
and form the backbone of modern inference in high- and infinite-dimensional problems \citep{van2023weak}. 
For stationary data, a comprehensive theory now provides such uniform CLTs under high-level conditions \citep{dehling2002empirical, kosorok2008introduction, van2023weak}. 
Our interest is in genuinely non-stationary settings, where the law of the observations may vary with both $n$ and $i$. 
To understand the resulting obstacle, it is useful to view weak convergence through its two defining components:
asymptotic tightness and marginal convergence.
Although establishing asymptotic tightness under non-stationarity is technically demanding, 
it can still be ensured under high-level conditions; see \Cref{thm:asy-tightness-sup-br}. 
The difficulty lies instead in the marginals of $\G_n$. 
When the data-generating mechanism evolves, the finite-dimensional distributions of 
$\G_n$ typically drift rather than approach a fixed limit.

A simple illustration is provided by 
the empirical distribution process of a non-stationary sequence $X_{n}\in \R$. 
Let 
\begin{align*}
    \G_n\in \ell^\infty(\R), \quad \G_n(x)=\frac{1}{\sqrt{n}}\sum_{i=1}^{n}\left(\ind\{X_i\leq x\}-\E[\ind\{X_i\leq x\}]\right).
\end{align*}
The process $\G_n$ measures the fluctuation of the empirical distribution around its 
mean and gives rise to familiar statistics such as 
Kolmogorov-Smirnov-type suprema and Cramer-von Mises-type integrated quadratic 
functionals. 
The underlying function class of left-tail indicators is sufficiently small that, under 
general assumptions, $\G_n$ remains asymptotically tight, 
even in non-stationary settings.
The covariances of $\G_n$, however, may not converge and $\G_n$ need not 
converge weakly to any element of $\ell^\infty(\R)$ if the underlying covariance structure drifts.
As a consequence, the usual route of approximating $\G_n$ by its Gaussian weak limit and constructing critical values 
based on such approximation is unavailable. 
This illustrates the main difficulty addressed in the present paper: 
in non-stationary problems, $\G_n$ may remain tight yet without being governed by a single limiting distribution.

\paragraph*{Existing approaches} 
A common ad-hoc fix to this conundrum is to de-trend, difference, 
or otherwise transform the data to remove the most glaring effects of 
non-stationarity \citep[e.g.,][]{shumway2000time}. The pre-processed data is assumed 
to be stationary, and classical CLTs are applied. This approach is sometimes successful in practice, but it has its pitfalls. The pre-processing is unlikely to remove all non-stationarity issues, and potential uncertainty and data dependence in the pre-processing are not accounted for in the subsequent analysis. 
A likely reason for the popularity of these heuristics is the apparent 
scarcity of limit theorems for non-stationary data that could support 
the development of inferential methods.

In certain cases, non-stationarity issues can be bypassed through standardization.
For example \citet{merlevede2020functional} 
establish a univariate, uniform-in-time CLT by standardizing the 
sample average by its standard deviation, 
which ensures a fixed standard Gaussian limit. 
Extending this approach to multivariate settings is straightforward 
by additionally de-correlating the coordinates. 
However, standardization is doomed to fail in general infinite-dimensional settings, 
where a decorrelated Gaussian limit process has unbounded sample paths (see \Cref{sec:standardization} for more details). 

Another recent line of work \citep{zhang2017gaussian, karmakar2020optimal, mies2023sequential, 
gpapproximation2024} couples the sample average to an explicit sequence of 
Gaussian variables on an enriched probability space. 
Such results are stronger than necessary for most statistical applications and 
apply only to multivariate sample averages with extensions 
to more general empirical processes being an open problem.

An alternative approach relies on the \emph{local stationarity} assumption, 
as summarized in \cite{dahlhaus2012locally} and recently extended to empirical 
processes by \cite{10.3150/21-BEJ1351}. Here, asymptotic results become 
possible by considering a hypothetical sequence of data-generating processes 
providing increasingly many, increasingly stationary observations in a given time window. 
This framework is well suited to localized procedures, but it addresses a different problem from the one studied here. 
It does not apply, for example, to a simple (non-localized) sample average 
over non-stationary random variables.

\paragraph*{Contribution 1: Relative weak convergence}
This paper takes a different approach. 
Fundamentally, 
a CLT compares the distribution of a sample quantity 
to a fixed Gaussian law. 
In case of non-stationary data, where marginals don't stabilize, there is no fixed distribution to converge to. 
The idea is to simply compare to an appropriate sequence of Gaussian processes, whose parameters can vary with the sample size. 
The resulting notion of weak convergence of two sequences of processes retains the essential properties of 
classical weak convergence — such as the continuous mapping theorem, 
the functional delta method, and quantile convergence — while being flexible enough 
to handle non-stationary data. 
Conceptually, this is similar to the idea of Gaussian approximation \citep{chernozhukov2016empirical, chang2024central}, however, within a genuine weak convergence framework for stochastic processes. 
Despite its generality, we will demonstrate that
the required conditions are comparable to those used in stationary CLTs, 
and that the framework can substitute uniform CLTs in statistical applications.

To formalize the idea, we introduce relative weak convergence.
\begin{definition*}
Let $X_n,Y_n\in \ell^\infty(T)$ be two sequences of stochastic processes indexed by some set $T$.
 We say $X_n,Y_n$ are \emph{relatively weakly convergent}
 if $$\bigl|\E^*[f(X_n)]-\E^*[f(Y_n)]\bigr|\to 0$$
 for all $f:\ell^\infty(T)\to \R$ bounded and continuous.
\end{definition*}
Note that if $Y_n=Y$ is constant, this is simply the definition of weak convergence. If $T$ is finite and the distributions of $Y_n$ are suitably continuous, relative weak convergence is equivalent to convergence of the difference of distribution functions.
Our notion of asymptotic normality looks as follows.
\begin{definition*}
A sequence $Y_n\in \ell^\infty(T)$ of stochastic processes indexed by some set $T$ satisfies a
 \emph{relative central limit theorem} 
 if there exists a relatively compact sequence of centered, tight, and Borel measurable Gaussian processes $N_{n}\in \ell^\infty(T)$ with 
 $$\cov[Y_n(s),Y_n(t)]=\cov[N_n(s),N_n(t)]$$ such that $Y_n$ and $N_{n}$ are relatively weakly convergent.
\end{definition*}
Let us explain the definition.
The covariance structure of the approximating Gaussian process $N_{n}$ is the natural Gaussian process to compare $Y_n$ to.
Relatively compact sequences are the appropriate analog of tight limits in the relative context and defined formally in \Cref{sec:rwc}.
In stationary settings, the sequence of approximating Gaussian processes stabilizes 
and relative and classical CLTs are essentially equivalent.

\paragraph*{Contribution 2: Relative CLTs for sample averages and empirical processes}

Both finite- and infinite-dimensional relative CLTs hold under high-level assumptions similar to those in classical stationary CLTs.
One of our main results looks as follows (\Cref{thm:multiplier-rel-clt}).
\begin{theorem*}
    Let $X_{n,1},\ldots, X_{n,k_n}$ be a triangular array of random variables with values in a Polish space $\Xcal$.
    Let $\Fcal$ be a set of measurable functions $\Xcal\to \R$ and $w_{n,i}:S \to \R$ be a family of uniformly bounded weights.
    Under some moment, mixing, and bracketing entropy conditions
    the weighted empirical process
    $\mathbb{G}_n\in \ell^\infty(S\times \Fcal)$ defined by
    $$\G_n(s,f)=\frac{1}{\sqrt{k_n}}\sum_{i=1}^{k_n}w_{n,i}(s)\bigl(f(X_{n,i})-\E[f(X_{n,i})]\bigr)$$
    satisfies a relative CLT.
\end{theorem*}
The weights $w_{n,i}$ can be specified to cover a variety of different CLT flavors. For example, $w_{n,i} \equiv 1$ yields a CLT for the unweighted empirical process;
taking $w_{n,i}(s)=\ind_{i\leq \lfloor sn\rfloor}$ yields a relative sequential CLT (\Cref{thm:rel-sequ-clt}); 
taking $ w_{n,i}(s) = K((i/n - s) / b_n)$ for some kernel $K$ and bandwidth $b_n$, we obtain localized CLTs similar to those established under local stationarity by \citet{10.3150/21-BEJ1351}. 

To prove the theorem, we rely on a characterization of relative CLTs
in terms of relative compactness and marginal relative CLTs of $\G_n$,
similar to classical empirical process theory. 
Regarding the marginals, we provide a relative version of Lyapunov's multivariate CLT (\cref{thm:multi-rel-clt}).
Relative compactness is established under bracketing conditions tailored
to non-stationary $\beta$-mixing sequences (\cref{thm:asy-tightness-sup-br}).
Generally, such entropy conditions also ensure the existence of the sequence of approximating Gaussians.
The proof relies on a chaining argument with adaptive coupling and truncation.
An important intermediate step is a new maximal inequality (\Cref{thm:chaining-in-text}) that may be of independent interest.

\paragraph*{Contribution 3: Bootstrap inference under non-stationarity}
Our third contribution addresses the key practical obstacle: 
the covariance structure of the approximating Gaussian process is unknown and generally difficult to estimate.
The bootstrap offers a natural solution, but existing procedures 
are mostly tailored to specific types of non-stationarity \citep{buhlmann1998sieve,synowiecki2007consistency}.
Extending \cite{bucher2019note}, we characterize bootstrap consistency 
through relative weak convergence (\Cref{thm:bstrap-rwc}) 
and establish new consistency results for generic multiplier bootstrap procedures 
for non-stationary time-series solely under moment, mixing, and entropy conditions (\Cref{thm:bstrap-univ}).
The framework also makes transparent a fundamental limitation of genuinely non-stationary settings: 
the mean function $t\mapsto \E[X_{n,i}(t)]$ generally cannot be estimated consistently. 
We show that the bootstrap nonetheless can provide valid (though possibly conservative) inference, 
and that it achieves full consistency when a consistent mean estimator is available, 
e.g., under common null hypothesis.

\paragraph*{Summary and outline}
In summary, our main contributions are as follows:
\begin{itemize}
 \item We introduce an extension of weak convergence, which is general enough to explain the asymptotics of non-stationary processes and convenient to use in statistical applications.
 \item We derive asymptotic normality of empirical processes for non-stationary time series under the same high-level assumptions as for classical stationary CLTs. 
 \item We facilitate the practical use of our theory by establishing consistency of a multiplier bootstrap for non-stationary time series.
\end{itemize}
This provides a new perspective on the asymptotic theory of non-stationary processes and a comprehensive framework of concepts and tools for statistical inference with general non-stationary data.  
Our results are applicable to a wide range of statistical methods and provide effective drop-in replacements for classical CLTs in non-stationary settings.

The remainder of this paper is structured as follows. 
Relative weak convergence and CLTs are developed in \Cref{sec:rwc}.
We provide tools for proving relative CLTs and foster their intuition by connecting relative with classical CLTs.
\Cref{sec:ts-rel-clts} develops 
relative CLTs for non-stationary $\beta$-mixing sequences.
\Cref{sec:bstrap} establishes multiplier bootstrap consistency 
for non-stationary time series, with applications to non-parametric trend estimation and hypothesis testing discussed in \Cref{sec:applications}.

\section{Relative Weak Convergence and CLTs}
\label{sec:rwc}

This section introduces definitions, characterizations, and basic properties of relative weak convergence and CLTs.

\subsection{Background and notation}
Throughout this paper, we make use of Hoﬀmann-Jørgensen’s theory of weak convergence. We first recall some definitions and central statements about weak convergence in general
metric spaces and in the space of bounded functions taken from Chapter 1 of \cite{van2023weak}, abbreviated as VdV below.

\paragraph*{Weak convergence in general metric spaces}
In what follows we denote by $X_n\colon\Omega_n \to \D$ sequences of (not necessarily measurable) maps with $\Omega_n$ probability spaces
and $\mathbb{D}$ some metric space. We write $\E^*$ and $\E_*$ for outer and inner expectation, respectively, and $\Pr^*$ and $\Pr_*$ for outer and inner probability, respectively (see Chapter 1.2 of VdV).
We will assume 
$\Omega_n=\Omega$ without loss of generality (see discussion above Theorem 1.3.4 in VdV).
To avoid clutter, we omit the domain $\Omega$ and write $X_n\in \D$ whenever it is clear from context.

\begin{definition}%
    The sequence $X_n$ converges weakly to some Borel measurable map $X\in \D$
    if
    $$\E^*[f(X_n)]\to \E[f(X)]$$
    for all bounded and continuous functions $f\colon \mathbb{D}\to \R$.
    We write
    $$X_n\dc X.$$
\end{definition}
Weak convergence of (measurable) random vectors agrees with the usual notation
in terms of distribution functions.

\begin{definition}
    A sequence $X_n$ is called \emph{asymptotically measurable}
    if $$\E^*[f(X_n)]-\E_*[f(X_n)]\to 0$$
    for all $f:\D\to \R$ bounded and continuous. %
    The sequence is called \emph{asymptotically tight} if,
    for all $\eps>0$, a compact set $K\subseteq \D$ exists such that
    $$\liminf_{n \to \infty} \Pr_*(X_n\in K^{\delta})\geq 1-\eps,$$
    for all $\delta>0$ with $K^{\delta}=\{y\in \D\colon \exists x\in K \text{ s.t. }d(x,y)<\delta\}$. 
\end{definition}

Weak convergence to some tight Borel measure implies
asymptotic measurability and asymptotic tightness.
Conversely, Prohorov's theorem 
ensures weak convergence along subsequences
whenever the sequence is asymptotically tight and measurable.

\paragraph*{Weak convergence of stochastic processes}
A stochastic process indexed by a set $T$ is a collection
$\{Y(t)\colon t\in T\}$ of random variables $Y(t)\colon \Omega\to \R$ defined on the same probability space.
A \emph{Gaussian process (GP)} is a stochastic process $\{N(t)\colon t\in T\}$
such that $(N(t_1),\ldots,N(t_k))$ is multivariate Gaussian for all $t_1,\ldots,t_k\in T$.

Weak convergence of processes is typically considered in (a subspace of) the space
of bounded functions 
$$\ell^\infty(T)=\left\{f:T\to \R\colon \|f\|_T=\sup_{t\in T}|f(t)|<\infty \right\},$$
equipped with the uniform metric $d(f,g)=\sup_{t\in T}|f(t)-g(t)|=\|f-g\|_T$.
If the sample paths $t\mapsto Y(t)(\omega)$ of a stochastic process $Y$ are bounded for all $\omega\in \Omega$, it induces a map 
$$Y\colon \Omega\to \ell^\infty(T),\quad \omega \mapsto (t\mapsto Y(t)(\omega)).$$
To abbreviate notation, we say $Y\in \ell^\infty(T)$ is a stochastic process (with bounded sample paths)
indicating that $Y(t)\colon \Omega\to \R$ are random variables with common domain.

A sequence $Y_n\in \ell^\infty(T)$ converges weakly to some tight and
Borel measurable $Y\in \ell^\infty(T)$ if and only if $Y_n$ is asymptotically tight, asymptotically measurable 
and all marginals converge weakly, i.e., 
$$(Y_n(t_1),\ldots,Y_n(t_k))\dc (Y(t_1),\ldots,Y(t_k))$$ 
in $\R^k$ for all $t_1,\ldots,t_k\in T$. 

\subsection{Non-stationary CLTs via weak convergence along subsequences}\label{sec:standardization}
Let us first discuss some difficulties with non-stationary CLTs.
Consider a sequence of stochastic processes 
$X_1, X_2, \ldots\in \ell^{\infty}(T)$.
CLTs establish weak convergence of the empirical process $\G_n\in \ell^\infty(T)$ defined by
$$\G_n(t)=\frac{1}{\sqrt{n}}\sum_{i=1}^n X_i(t)-\E[X_i(t)]$$
to some tight and measurable Gaussian process $\G\in \ell^\infty(T)$.
The convergence is characterized in terms of asymptotic tightness and marginal CLTs. 
The former is independent of stationarity (e.g., Theorems 2.11.1 and 2.11.9 of \cite{van2023weak}).
The latter involves weak convergence of the marginals. In particular,
$$\G_n(t)\dc \G(t) \quad \mbox{for all } t \in T.$$
This also implies convergence of the variances under mild conditions.
Here a certain degree of stationarity of the samples is required.
For non-stationary observations,
the convergence $$\var[\G_n(t)]\to \var[\G(t)]$$ fails in general.
As an example consider the empirical distribution process 
\begin{align}\label{ex:emp-distr}
    \G_n\in \ell^\infty(\R), \quad \G_n(x)=\frac{1}{\sqrt{n}}\sum_{i=1}^{n}\left(\ind\{X_i\leq x\}-\E[\ind\{X_i\leq x\}]\right).
\end{align}
of a non-stationary sequence $X_{n}\in \R$. 
Even under the common null hypothesis of equal distribution functions
\begin{align}\label{ex:h0-emp-distr}
    H_0 \colon \; \Pr(X_i\leq x) = \Pr(X_1\leq x) \quad \forall i\in \N, x\in \R.
\end{align}
the variances 
$$\var[\G_n(x)]=\Pr(X_1\leq x)(1-\Pr(X_1\leq x))+\frac{2}{n}\sum_{1\leq i< j\leq n}\cov[\ind\{X_i\leq x\},\ind\{X_j\leq x\}].$$
need not converge if the underlying covariance structure of $X_n$ drifts.

For this reason, some non-stationary CLTs 
rely on standardization such that the covariance matrix equals the identity.
Such standardization can only work in finite dimensions, however.
The intuitive reason is a decorrelated Gaussian limit process $\G$
cannot have bounded sample paths (nor be tight) unless $T$ is finite.
\begin{lemma}\label{lem:bounded-sample-paths}
    If a Gaussian process $\G\in \ell^\infty(T)$ satisfies
    $\cov[\G(s),\G(t)]=0$ for $s\neq t$ and $\var[\G(s)]=1$, then, $T$ is finite.
\end{lemma}
\begin{proof}
    \Cref{proof:bounded-sample-paths}.
\end{proof}
In summary, the problem with uniform CLTs is that marginal CLTs generally require a 
certain degree of stationarity, 
which cannot be bypassed through (naive) standardization.

The basic result we want to put into a larger context is a version of Lyapunov's CLT.
Write
$$S_n=\frac{1}{\sqrt{k_n}}\sum_{i=1}^{k_n}X_{n,i}-\E[X_{n,i}].$$
for $X_{n,1}\ldots X_{n,k_n}\in \R$ a triangular array of independent random variables.
Assume the general moment condition $\sup_{n,i}\E[|X_{n,i}|^{2+\delta}]<\infty$ for some $\delta>0$.
Lyapunov's CLT asserts asymptotic normality whenever variances converge.
\begin{proposition}
    Suppose $\sup_{n,i}\E[|X_{n,i}|^{2+\delta}]<\infty$ for some $\delta>0$. Then
    $\var[S_n]\to \sigma_\infty^2$ if and only if $S_n\dc \mathcal{N}(0,\sigma_\infty^2)$.
\end{proposition}
\begin{proof}
    The sufficiency is a special case of Proposition 2.27 of \cite{van2000asymptotic}
    and the necessity is Example 1.11.4 of \cite{van2023weak}.
\end{proof}
When $\var[S_n]$ doesn't converge, the moment condition still
implies that $\var[S_n]$ is uniformly bounded.
Thus, every subsequence of $\var[S_n]$ contains a converging subsequence $\var[S_{n_k}]$
along which Lyapunov's CLT implies asymptotic normality of $S_{n_k}$.
The weak limit of $S_{n_k}$ however depends on the limit of $\var[S_{n_k}]$.
In other words, $S_n$ converges weakly to some Gaussian along subsequences, yet not globally.
To obtain a global Gaussian to compare $S_n$ to, the complementary observation is the following fact:
\begin{align*}
    \var[S_{n_k}]\to \sigma_\infty^2 \qquad \text{if and only if} \qquad \mathcal{N}(0,\var[S_{n_k}])\dc \mathcal{N}(0,\sigma_\infty^2).
\end{align*}
Thus, along subsequences, $S_{n}$ and $\mathcal{N}(0,\var[S_{n}])$ have the same weak limits. 
Because any subsequence of $S_n$ contains a weakly convergent subsequence, i.e., $S_n$ is relatively compact,
some thought reveals their difference of distribution functions converges to zero.
\begin{proposition}[Relative Lyapunov's CLT]
    Denote by $F_n$ resp. $\Phi_n$ the distribution function of 
    $S_n$ resp. $\mathcal{N}(0,\var[S_n])$.
    Then,
    $|F_n(t)-\Phi_n(t)|\to 0$
    for all $t\in \R$.
\end{proposition}

\begin{proof}
    This is a special case of \Cref{thm:rel-lind-feller}.
\end{proof}

In other words, even in non-stationary settings, the sample average remains approximately Gaussian,
but with the notion of a
limiting distribution replaced by a sequence of Gaussians. 
This mode of convergence and the resulting type of asymptotic normality extend naturally to 
stochastic processes.

\subsection{Relative weak convergence}
All proofs of the remaining section are found in \Cref{ap:rwc-rclts}.
Let $X_n, Y_n$ be sequences of arbitrary maps from probability spaces $\Omega_n,\Omega^\prime_n$
into a metric space $\mathbb{D}$.
Throughout the rest of this paper it is tacitly understood that all such maps have a common probability space as their domain.

\begin{definition}[Relative weak convergence]
    We say $X_n$ and $Y_n$ are \emph{relatively weakly convergent}
    if 
    $$\left|\E^*[f(X_n)]-\E^*[f(Y_n)]\right|\to 0$$
    for all $f:\mathbb{D}\to \R$ bounded and continuous.
    We write $$X_n\rd Y_n.$$
\end{definition}

\begin{remark}
    For $X$ a fixed Borel law, we have $X_n\rd X$ if and only if 
    $X_n\dc X$. 
\end{remark}

Contrary to weak convergence, relative weak convergence 
implies neither measurability nor tightness.
Any sequence is relatively weakly convergent to itself.
For the purpose of this paper, we restrict to relative weak convergence of 
relatively compact sequences.
Those turn out to be the natural analog 
of tight, measurable limits in classical weak convergence.  

\begin{definition}[Relative asymptotic tightness and compactness]
    We call $X_n$ \emph{relatively asymptotically tight}
    if every subsequence contains a further subsequence which is
    asymptotically tight.
    We call $X_n$ \emph{relatively compact}
    if every subsequence contains a further subsequence which converges weakly to a tight Borel law.
\end{definition}

We see from the definition that, if $X_n\rd Y_n$ and $Y_n\dc Y$,
then $X_n\dc Y$. 
Thus, if $Y_n$ is relatively compact, $X_n\rd Y_n$ 
essentially states that $X_n$ and $Y_n$ have the same weak limits 
along subsequences.

\begin{proposition}\label{prop:chara-rc-rwc}
    Assume that $X_n$ is relatively compact. The following are equivalent:
    \begin{enumerate}
        \item \label{A1:chara-rc-rwc}$X_n\rd Y_n$.
        \item \label{A2:chara-rc-rwc}For all subsequences $n_{k}$ such that $X_{n_{k}}\dc X$ with $X$ a tight Borel law it follows $Y_{n_{k}}\dc X$.
        \item \label{A3:chara-rc-rwc}For all subsequences $n_{k}$ there exists a further subsequence $n_{k_i}$ such that both $X_{n_{k_i}}$ and $Y_{n_{k_i}}$ converge weakly to the same tight Borel law.
    \end{enumerate}
    In such case, $Y_n$ is relatively compact as well.
\end{proposition}

The characterization of relative weak convergence via weak convergence on subsequences is very convenient. 
As a result, many useful properties of weak convergence
can be transferred to relative weak convergence.
The following results are particularly relevant for 
statistical applications and show that relative 
weak convergence supports conclusions comparable 
to those obtained under weak convergence.

In statistical inference, we usually make statements about event probabilities. Relative weak convergence implies convergence of probabilities for certain sequences of events.
Let $S^\delta = \{x \in \mathds D\colon d(x, S) < \delta\}$ denote the $\delta$-enlargement and $\partial S$ the boundary (closure minus interior) of a set $S$.

\begin{proposition} \label{prop:portmanteau}
    Let  $X_n \rd Y_n$ with $Y_n$ relatively compact, and $S_n$ be Borel sets such that 
    every subsequence of $S_n$ contains a further subsequence $S_{n_k}$ such that $\lim_{k \to \infty} S_{n_k} = S$ for some $S$ with 
    \begin{align} \label{eq:continuity-set}
        \liminf_{\delta \to 0}\limsup_{k \to \infty} \Pr^*(Y_{n_k} \in (\partial S)^\delta) = 0.
    \end{align}
    Then
    \begin{align*}
        \lim_{n \to \infty} |\Pr^*( X_n \in S_n) - \Pr^*( Y_n \in S_n)| = 0.
    \end{align*}
\end{proposition}
Condition \eqref{eq:continuity-set} ensures that $S_n$ are continuity sets of $Y_n$ asymptotically in a strong form. This prevents the laws from accumulating too much mass near the boundary of $S_n$. In statistical applications, $Y_n$ is typically (the supremum of) a non-degenerate Gaussian process, for which appropriate anti-concentration properties can be guaranteed for any set $S$ \citep[e.g.,][]{giessing2023anticoncentrationsupremagaussianprocesses}. We conclude that relative weak convergence is indeed sufficient for most statistical applications, particularly for assessing validity of confidence sets and hypothesis tests.

Next, we discuss tools for establishing relative weak convergence in the space of bounded functions $\ell^\infty(T)$.
A core tool for establishing classical weak convergence in $\ell^\infty(T)$ is its characterization in terms of marginal weak convergence and asymptotic tightness \citep[Theorem 1.5.7]{van2023weak}. 
A similar result holds for relative weak convergence.

\begin{theorem}\label{thm:rwc-marginals}
    Let $X_n,Y_n\in \ell^\infty(T)$ with $Y_n$ relatively compact.
    The following are equivalent:
    \begin{enumerate}
        \item $X_n\rd Y_n$.
        \item $X_n$ is relatively asymptotically tight and
                all marginals satisfy $$\bigl(X_n(t_1),\ldots,X_n(t_d)\bigr)\rd \bigl(Y_n(t_1),\ldots,Y_n(t_d)\bigr)$$
                for all finite subsets $\{t_1,\ldots,t_d\}\subset T$. 
    \end{enumerate}
\end{theorem}

The remaining results are useful for transferring relative weak convergence of given sequences to others.

\begin{proposition}[Relative continuous mapping]\label{prop:rcm}
    If $X_n \rd Y_n$ and $g\colon \D\to \E$ is continuous then $g(X_n) \rd g(Y_n)$.
\end{proposition}

\begin{proposition}[Extended relative continuous mapping]\label{prop:ext-cont-map}
    Let $g_n\colon \D\to \E$ be a sequence of functions and $Y_n$ be relatively compact.
    Assume that for all subsequences of $n$ there exists another subsequence $n_k$ and some
    $g\colon \D\to \E$ such that $g_{n_k}(x_k)\to g(x)$ for all $x_k\to x$ in $\D$.
    Then,
    \begin{enumerate}
        \item $g_n(Y_n)$ is relatively compact.
        \item if $X_n \rd Y_n$ then $g_n(X_n)\rd g_n(Y_n)$.
    \end{enumerate}
\end{proposition}

\begin{proposition}[Relative delta-method]\label{prop:rel-delta-meth}
    Let $\D,\E$ be metrizable topological vector spaces 
    and $\theta_n\in \D$ be relatively compact. 
    Let $\phi\colon \D\to \E$ be continuously Hadamard-differentiable (see \cref{def:hadamard-diff}) in an open subset $\D_0\subset \D$ with $\theta_n\in \D_0$ for all $n$.
    Assume $$r_n\left(X_n-\theta_n\right)\rd Y_n$$
    for some sequence of constants $r_n\to \infty$
    with $Y_n$ relatively compact.
    Then,
    \begin{align*}
        r_n\bigl(\phi(X_n)-\phi(\theta_n)\bigr)\rd \phi^\prime_{\theta_n}(Y_n).
    \end{align*}
\end{proposition}

\subsection{Relative central limit theorems}

Fix a sequence of stochastic processes $Y_n \in \ell^\infty(T)$ with 
$\sup_{t \in T}\E[Y_n(t)^2]<\infty$. 
CLTs assert weak convergence $Y_n\dc N$ with $N$ some tight and 
measurable GP. 
Naturally, we define (relative) asymptotic normality 
as $Y_n\rd N_n$ with $N_n$ some sequences of GPs. 
Contrary to CLTs with a fixed limit, there is no unique sequence of `limiting' GPs, however. 
It shall be convenient to specify the covariance structure of $N_n$
to mirror that of $Y_n$.
This allows for a tractable Gaussian approximation. 

\begin{definition}[Corresponding GP]
    Let $Y\in \ell^\infty(T)$ be a stochastic process with finite second moments. 
    A \emph{Gaussian process (GP) corresponding to $Y$} is a map $N_Y$ with values in $\ell^\infty(T)$
    such that
    $$\{N_Y(t)\colon t\in T\}$$ is a centered GP with covariance function
    given by $(s,t)\mapsto \cov[Y(s),Y(t)]$.
\end{definition}
Similar to classical CLTs and in view of \Cref{prop:chara-rc-rwc}, we further restrict to 
relatively compact sequences of tight and Borel measurable GPs 
and define a relative CLT as follows.

\begin{definition}[Relative CLT]
    We say that the sequence $Y_n$ satisfies a
    \emph{relative central limit theorem}
    if
         a relatively compact sequence of tight and Borel measurable GPs $N_{Y_n}$ corresponding to $Y_n$ with
        $Y_n\rd N_{Y_n}$ exists.
\end{definition}

\Cref{thm:rwc-marginals} characterizes relative CLTs
in terms of marginal relative CLTs and tightness.
\begin{corollary}\label{lem:rclt-marginal-rclt}
    The sequence $Y_n$ satisfies a 
    relative CLT if and only if 
    \begin{enumerate}
        \item \label{A1:rclt-marginal-rclt}there exist tight and Borel measurable GPs $N_{Y_n}$ corresponding to $Y_n$,
        \item \label{A2:rclt-marginal-rclt}$Y_n$ and $N_{Y_n}$ are relatively asymptotically tight and
        \item \label{A3:rclt-marginal-rclt}all marginals $\left(Y_n(t_1),\ldots,Y_n(t_k)\right)\in \R^k$ satisfy a relative CLT.
    \end{enumerate}
\end{corollary}

The restriction to relatively compact sequences of tight and Borel measurable GPs
enables inference just as in 
classical weak convergence theory (see previous section and \Cref{sec:applications}). 
Such sequences exist under mild assumptions.
In finite dimensions, corresponding tight
Gaussians always exist.
Any such sequence converges iff its covariances 
converge.
Hence, a sequence of Gaussians is relatively compact 
iff covariances converge along subsequences.
Equivalently, 
the sequences of variances are
uniformly bounded (\Cref{cor:gaussian-dominated-entropy}).
As a result, multivariate relative CLTs can be characterized as follows.

\begin{proposition}\label{prop:rel-clt-chara}
    Let $Y_n$ be a sequence of $\R^d$-valued random variables.
    Denote by $\Sigma_n$ the covariance matrix of $Y_n$.
    Then, the following are equivalent:
    \begin{enumerate}
        \item \label{A1:rel-clt-chara}$Y_n$ satisfies a relative CLT.
        \item \label{A2:rel-clt-chara}for all subsequences $n_k$ with $\Sigma_{n_{k}}\to \Sigma$ it holds
              $$Y_{n_{k}}\dc\mathcal{N}(0,\Sigma)$$ and $\sup_{n\in \N,i\leq d}\var[Y_{n}^{(i)}]<\infty$ where $Y_{n}^{(i)}$ denotes the $i$-th component of $Y_n$.
        \item \label{A3:rel-clt-chara}all subsequences $n_k$ contain a subsequence $n_{k_i}$ such that $\Sigma_{n_{k_i}}\to \Sigma$
        and
        $$Y_{n_{k_i}}\dc\mathcal{N}(0,\Sigma).$$
    \end{enumerate}
\end{proposition}
In other words, multivariate relative CLTs are essentially equivalent to 
CLTs along subsequences where covariances converge. 
From this it is straightforward to generalize classical multivariate CLTs to relative multivariate CLTs 
(e.g., Lindeberg's CLT, \Cref{thm:rel-lind-feller}).

For infinite dimensional index sets $T$, 
Kolmogorov's extension theorem implies existence of GPs $\{N_n(t)\colon t\in T\}$
with
$$\cov[N_n(s),N_n(t)]=\cov[Y_n(s),Y_n(t)],$$
but potentially unbounded sample paths. 
Bounded sample paths, tightness and asymptotic tightness can be established under entropy conditions, as shown in the following.

Define the $\epsilon$-covering number $N(\epsilon,T,d)$ of a semi-metric space $(T,d)$
as the minimal number of $\epsilon$-balls needed to cover $T$.
Denote by $$\rho_{n}(s,t)=\var[Y_n(s)-Y_n(t)]^{1/2}, \quad s,t \in T,$$
the standard deviation semi-metric on $T$ induced by $Y_n$.

\begin{proposition}\label{prop:ex-tight-GPs} 
    If for all $n$ it holds $$\int_{0}^{\infty}\sqrt{\ln N(\epsilon,T,\rho_{n})}d\epsilon<\infty,$$ 
    then a sequence of tight and Borel measurable GPs $N_{Y_n}$ corresponding to $Y_n$ exists.
\end{proposition}

\begin{proposition}\label{prop:asy-tight-GPs}
    Let $N_{Y_n}$ be a sequence of Borel measurable GPs corresponding to $Y_n$.
    Assume that there exists a semi-metric $d$ on $T$ such that 
    \begin{enumerate}
        \item \label{A1:asy-tight-GPs}$(T,d)$ is totally bounded.
        \item \label{A2:asy-tight-GPs}$\lim_{n\to \infty}\int_{0}^{\delta_n}\sqrt{\ln N(\epsilon,T,\rho_{n})}d\epsilon= 0$ for all $\delta_n\downarrow 0.$
        \item \label{A3:asy-tight-GPs}$\lim_{n\to \infty}\sup_{d(s,t)<\delta_n}\rho_{n}(s,t)=0$ for every $\delta_n\downarrow 0.$
    \end{enumerate}
    If further $\sup_n \sup_{t \in T}\var[Y_n(t)]<\infty$, the sequence $N_{Y_n}$ is asymptotically tight.
\end{proposition}
Note that condition \ref{A3:asy-tight-GPs} requires the existence of a global 
semi-metric with respect to which the sequence of standard 
deviation semi-metrics are asymptotically uniformly continuous. 
In many practical settings, one can simply take $d(t, s)=\sup_n\rho_n(t, s)$. 

\section{Relative CLTs for non-stationary time-series}\label{sec:ts-rel-clts}

With a suitable notion of relative asymptotic normality on hand, 
this section provides specific instances of relative CLTs for non-stationary time-series.
The assumptions of these CLTs are relatively weak, making them applicable in a wide range of statistical problems. 
As in classical theory, however, we cannot expect a relative CLT to hold under arbitrary dependence structures.
Several measures exist to constrain the dependence between observations, such as $\alpha$-mixing or $\phi$-mixing coefficients \citep{bradley2005basicproperties}, or the functional dependence measure of \citet{wu2005nonlinear}. 
In the following, we will focus on $\beta$-mixing, because it is widely applicable and allows for sharp coupling inequalities.

\begin{definition}
    Let $(\Omega,\mathcal{A},P)$ be a probability space and $\mathcal{A}_1,\mathcal{A}_2\subset \mathcal{A}$
    sub-$\sigma$-algebras.
    The \emph{$\beta$-mixing coefficient} is defined as
    $$\beta(\mathcal{A}_1,\mathcal{A}_2)=\frac{1}{2}\sup\sum_{(i,j)\in I\times J}|\Pr(A_i\cap B_j)-\Pr(A_i)\Pr(B_j)|,$$
    where the supremum is taken over all finite partitions $\cup_{i\in I}A_i=\cup_{j\in J}B_j=\Omega$ with $A_i\in \mathcal{A}_1,B_j\in \mathcal{A}_2$.
    For a triangular array $X_{n,i}$ of random variables with common (co)domain, $p,n\in \N$ and $p<k_n$ define $$\beta_n(p)=\sup_{k\leq k_n-p}\beta\left(\sigma(X_{n,1},\ldots,X_{n,k}),\sigma(X_{n,k+p},\ldots,X_{n,k_n})\right).$$
\end{definition}
The $\beta$-mixing coefficients quantify how 
independent events become when they are temporally separated.
If the $\beta$-mixing coefficients become zero as
$n$ and $p$ approach infinity, the events become 
close to independent. Note that the mixing coefficients themselves are indexed by $n$, reflecting the triangular array setup.

\subsection{Multivariate relative CLT}

We start with a multivariate relative CLT for triangular arrays of random variables.
\Cref{prop:rel-clt-chara} extends classical to relative multivariate CLTs.
The following result builds on Lyapunov's CLT in combination with 
a coupling argument. 
Let $X_{n,1},\ldots,X_{n,k_n}$ be a triangular array of $\R^d$-valued random variables.

\begin{theorem}[Multivariate relative CLT]\label{thm:multi-rel-clt}
    For some $\gamma>2$ and $\alpha<(\gamma-2)/2(\gamma-1)$ 
    assume
    \begin{enumerate}
        \item \label{A1:multi-rel-clt} $k_n^{-1}\sum_{i,j=1}^{k_n}|\cov[X_{n,i}^{(l_1)},X_{n,j}^{(l_2)}]|\leq K$ for all $n$ and $l_1,l_2=1,\ldots d$.
        \item \label{A2:multi-rel-clt} $\sup_{n,i}\E\left[|X_{n,i}^{(l)}|^\gamma\right]<\infty$ for all $l = 1, \dots, d$.
        \item \label{A3:multi-rel-clt}$k_n\beta_{n}(k_n^{\alpha})^\frac{\gamma-2}{\gamma}\to 0$.
    \end{enumerate}
    Then, the scaled sample average $k_n^{-1/2}\sum_{i=1}^{k_n}(X_{n,i}-\E[X_{n,i}])$ satisfies a relative CLT. 
\end{theorem}
\begin{proof}
    \Cref{sec:multi-rel-clt}
\end{proof}
The summability condition \ref{A1:multi-rel-clt} on the covariances can be seen as a minimal requirement for any 
general CLT under mixing conditions \citep{bradley1999growth}.
Condition \ref{A1:multi-rel-clt} and \ref{A3:multi-rel-clt} restrict the 
dependence where \ref{A1:multi-rel-clt} essentially bounds the variances of the scaled sample average.
Condition \ref{A2:multi-rel-clt} excludes heavy tails of $X_{n,i}$.
Note that \ref{A2:multi-rel-clt} and \ref{A3:multi-rel-clt} exhibit a trade-off. 
Uniform bounds on higher moments weaken the conditions on the mixing coefficients' decay rate.

\begin{example}
    Assuming the moment condition \ref{A2:multi-rel-clt} with $\gamma=4$, 
    the $\beta$-mixing condition reads $k_n^2\beta_n(k_n^{\alpha})\to 0$ for some 
    $\alpha<1/3$. This holds, for example, if $\beta_n(p)\leq Cp^{-\rho}$ for some $C < \infty, \rho>6$ and all $n, p \in \N$.
\end{example}

\subsection{Asymptotic tightness under bracketing entropy conditions}\label{sec:asy-tightness}

In order to extend the multivariate relative CLT to an empirical process CLT, we need a way to ensure relative compactness.
Consider the following general setup:
\begin{itemize}
	\item $X_{n,1},\ldots,X_{n,k_n}$ is a triangular array of random variables in some Polish space $\mathcal{X}$,
	\item $\mathcal{F}_n=\{f_{n,t}\colon t\in T\}$ is a set of measurable functions from $\mathcal{X}$ to $\R$ for all $n$, 
	\item $\Fcal = \bigcup_{n\in \N}\mathcal{F}_n$ admits a finite envelope function $F\colon \mathcal{X}\to \R$, i.e., $\sup_{n\in \N,f\in \Fcal_n}|f(x)|\leq F(x)$ for all $x\in \Xcal$.
\end{itemize}
Define the empirical process $\G_n$ on $T$
by $$\mathbb{G}_n(t)=\frac{1}{\sqrt{k_n}}\sum_{i=1}^{k_n}f_{n,t}(X_{n,i})-\E[f_{n,t}(X_{n,i})].$$
The envelope guarantees that $\G_n$ has bounded sample paths and we obtain a map
$\mathbb{G}_n$ with values in $\ell^\infty(T)$.
This empirical process can be seen as a triangular version of the (classical) empirical process 
indexed by a single set of functions. 

To ensure asymptotic tightness, we rely on bracketing entropy conditions, with norms tailored to the nonstationary time-series setting.
Given a semi-norm $\|\cdot \|$ on (an extension of) $\Fcal$, define the bracketing number 
$N_{[]}(\epsilon,\mathcal{F}, \| \cdot \|)$ as the minimal number 
of brackets $$[l_i,u_i]=\{f\in \mathcal{F}\colon l_i\leq f\leq u_i\}$$
such that $\mathcal{F}=\cup_{i=1}^{N_{\epsilon}}[l_i,u_i]$ with $l_i,u_i\colon \Xcal\to \R$ measurable
and $\|l_i-u_i\|\le\epsilon$.
For stationary observations $X_{n,i}\sim P$, bracketing entropy is usually measured with respect to 
an $L_p(P)$-norm, $p \ge 2$.
In case of non-stationary observations the brackets need to be measured with respect 
to all underlying laws of the samples. 
It turns out that a scaled average of $L_p(P_{X_{n,i}})$-norms is sufficient.

\begin{definition}
    Let $\gamma\geq 1$.
    Define the semi-norms $\|\cdot\|_{\gamma,n}$ resp. $\|\cdot\|_{\gamma,\infty}$
    on $\Fcal$ by 
    \begin{align*}
        \|h\|_{\gamma,n}=\left(\frac{1}{k_n}\sum_{i=1}^{k_n}\E\left[|h(X_{n,i})|^\gamma\right]\right)^{1/\gamma},
        \qquad
        \|h\|_{\gamma,\infty}=\sup_{n\in \N}\|h\|_{\gamma,n}.
    \end{align*}
\end{definition}
Note that $\|\cdot\|_{\gamma,n}$ is a composition of semi-norms, hence,
a semi-norm itself (\Cref{lem:properties-norm}).
The following result shows that $\|\cdot\|_{\gamma,n}$-bracketing entropy conditions imply asymptotic tightness 
of $\G_n$ under mixing assumptions.
Because $\|\cdot\|_{\gamma,n}$-bracketing entropy bounds covering entropy (\Cref{rem:br-bound-cn}) the existence 
of an approximating sequence of GPs is also guaranteed.

\begin{theorem}\label{thm:asy-tightness-sup-br}
    Assume that for some $\gamma>2$
    \begin{enumerate}
        \item \label{A1:asy-tightness-sup-br}$\|F\|_{\gamma,\infty}<\infty$,
        \item \label{A2n:asy-tightness-sup-br} $\sup_{n\in \N} \max_{m\leq k_n} m^{\rho}\beta_n(m)<\infty$ for some $\rho>\gamma/(\gamma-2)$,
        \item \label{A3n:asy-tightness-sup-br}$\int_{0}^{\delta_n}\sqrt{\ln N_{[]}(\epsilon,\mathcal{F}_n,\|\cdot\|_{\gamma,n})}d\epsilon \to 0$ for all $\delta_n\downarrow 0$ and the integrals are finite for all $n$. 
    \end{enumerate}
    Write $$d_n(s,t)=\|f_{n,s}-f_{n,t}\|_{\gamma,n}$$
    for $s,t\in T$.
    Assume that there exists a semi-metric $d$ on $T$ such that 
    $$\lim_{n\to \infty}\sup_{d(s,t)<\delta_n}d_{n}(s,t)=0$$ for all $\delta_n\downarrow 0$ and 
    $(T,d)$ is totally bounded.
    Then, 
    \begin{itemize}
        \item $\mathbb{G}_n$ is asymptotically tight.
        \item there exists an asymptotically tight sequence of tight Borel measurable GPs $N_n$
        corresponding to $\mathbb{G}_n$.
    \end{itemize}
\end{theorem}
The proof, given in \Cref{sec:proof-bracketing}, relies on a chaining argument with adaptive coupling and truncation.
The techniques are related to those in \cite{doukhan1995invariance} and 
\citet[Chapter 8]{rio2017asymptotic}, 
who consider a stationary $\beta$-mixing setting. 
However, since we work in a non-stationary framework, 
their results and arguments cannot be applied directly. 
In particular, the non-stationarity necessitates the use of a 
different norm to quantify the size of the brackets.

An important intermediate step is a new maximal inequality (see \Cref{thm:bracketing-coupled}) that may be of independent interest:
\begin{theorem}\label{thm:chaining-in-text}
	Let $\Fcal$ be a class of functions $f\colon \Xcal \to \R$ with envelope $F$, and
	\begin{align*}
		\| f \|_{\gamma,n} \le \delta, \qquad \frac{1}{n} \sum_{i,j= 1}^{n} |\cov[h(X_i),h(X_j)]|\leq K_1 \|h\|_{\gamma,n}^2,
	\end{align*}
	for some $\gamma > 2$, all $f\in \Fcal$ and $h:\Xcal\to \R$ bounded and measurable.
	Suppose that $\sup_n \beta_n(m) \leq K_2 m^{-\rho}$ for some $\rho \ge \gamma /(\gamma -2 )$.
	Then, for any $n \ge 5$ and $\delta \in (0,1)$,
	\begin{align*}
		\E\|\G_n \|_{\Fcal}   \lesssim  \int_0^\delta  \sqrt{\ln_+ N_{\left[\right]}(\epsilon) } d \epsilon + \frac{  \|F\|_{\gamma, n} [\ln N_{\left[\right]}(\delta)]^{ [1 - 1/(\rho + 1)](1 - 1/\gamma)}}{n^{-1/2 + [1 - 1/(\rho + 1)](1 - 1/\gamma)}} + \sqrt{n} N_{[]}^{-1}(e^n),
	\end{align*}
    where $N_{[]}(\eps) = N_{[]}(\epsilon, \Fcal, \|\cdot\|_{\gamma,n})$ and $\ln_+ x= \ln(x+1)$.
\end{theorem}
\citet[Theorem 4]{scholze2024weakconvergencefunctionindexedsequential} concurrently proved a similar result under a weaker $\alpha$-mixing assumption, but much
stricter entropy conditions. 

In many applications, $\|\cdot\|_{\gamma,n}$-bracketing numbers
can be replaced by $L_\gamma(Q)$-bracketing numbers whenever all $P_{X_{n,i}}$ are 
simultaneously dominated by some measure $Q$ (\Cref{sec:comp-brackets}), or simply the $L_\infty$-bracketing numbers (which coincide with $L_\infty$-covering numbers) whenever the function class is uniformly bounded.
Many bounds on $L_\gamma(Q)$-bracketing and $L_\infty$-covering numbers are well known \citep[Section 2.7]{van2023weak}.

\subsection{Weighted uniform relative CLT}\label{sec:weighted-unif-rclt}
With a multivariate relative CLT and conditions for asymptotic tightness, we have all we need to establish relative empirical process CLTs. Our main result below should cover many statistical applications.

Let
$\Fcal$ be a set of measurable functions with finite envelope $F$.
Let $w_{n,i}\colon S\to \R$, $n\in \N,1 \le i\leq k_n,$ be a family of weights
satisfying
$\sup_{n,i,x}|w_{n,i}(x)|<\infty$ and define the weighted empirical process as
$\mathbb{G}_n\in \ell^\infty(S\times \Fcal)$ 
$$\G_n(s,f)=\frac{1}{\sqrt{k_n}}\sum_{i=1}^{k_n}w_{n,i}(s)\bigl(f(X_{n,i})-\E[f(X_{n,i})]\bigr).$$

A relative CLT for $\G_n$ can be established in terms of bracketing entropy for the function class
$$\Wcal_n = \{g_{n,s}: \{1, \dots, k_n \}\to \R, i\mapsto w_{n,i}(s)\colon s\in S\}.$$
For $s,t\in S$ define the semi-metric 
$$d^w_n(s,t)=\|g_{n,s}-g_{n,t}\|_{\gamma,n}=\left(\frac{1}{k_n}\sum_{i=1}^{k_n}|w_{n,i}(s)-w_{n,i}(t)|^{\gamma}\right)^{1/\gamma},$$
 and assume the following entropy conditions on the weights: 
\begin{enumerate}[label=(W\arabic*)]
    \item \label{A1:weights} $\int_{0}^{\delta_n}\sqrt{\ln N_{\left[\right]}\left( \epsilon,\Wcal_n,\|\cdot\|_{\gamma,n}\right)}d\epsilon \to 0$ for all $\delta_n\downarrow 0$ and finite for all $n$.
    \item \label{A2:weights} there exists a semi-metric $d^w$ on $S$ such that for all $\delta_n\downarrow 0$
    $$\lim_{n\to \infty}\sup_{d^w(s,t)<\delta_n}d^w_{n}(s,t)=0.$$
    \item \label{A3:weights}$(S,d^w)$ is totally bounded.
\end{enumerate}
These assumptions cover the constant case $w_{n,i}(s)=1$ as well as
more sophisticated weights; see the next sections.

\begin{theorem}[Weighted relative CLT]\label{thm:multiplier-rel-clt}
    For some $\gamma>2$ assume that \ref{A1:weights}--\ref{A3:weights} and the following hold:
    \begin{enumerate}
        \item \label{uniformcltA1}$\sup_{i,n}\|F(X_{n,i})\|_{\gamma}<\infty$,
        \item \label{uniformcltA2n} $\sup_{n\in \N} \max_{m\leq k_n} m^{\rho}\beta_n(m)<\infty$ for some $\rho>2\gamma(\gamma-1)/(\gamma-2)^2$,
        \item \label{uniformcltA4}$\int_{0}^{\infty}\sqrt{\ln N_{[]}(\epsilon,\mathcal{F},\|\cdot\|_{\gamma,\infty})}d\epsilon<\infty.$
    \end{enumerate}
    Then, $\G_n$ satisfies a relative CLT in $\ell^\infty(S\times \Fcal)$.
\end{theorem}

\begin{proof}
    \Cref{sec:multiplier-rel-clt}
\end{proof}

The moment condition \ref{uniformcltA1} ensures that all $\gamma$-moments 
$\E[|f(X_{n,i})|^\gamma]$ are uniformly bounded. 
Again, \ref{uniformcltA1} and \ref{uniformcltA2n} entail a trade-off. 
Higher moments allow for a slower decay of the mixing coefficients. 

\begin{example}\label{ex:gamma-rho}
    Assuming the moment condition \ref{uniformcltA1} with $\gamma=4$, 
    the $\beta$-mixing coefficients must decay as
    $\sup_n\beta_n(m)\leq Cm^{-\rho}$ for some $\rho>6$ and 
    all $m$. 
\end{example}

The weighted relative CLT covers the empirical distribution process \eqref{ex:emp-distr} as a special case.

\begin{example}
    Let $X_n\in \R$ be a sequence of random variables and set 
    $\Fcal=\{\ind_{(-\infty,x]}\colon x\in \R\}$ as the class of left-tail indicators.
    Under the null of equal distribution function \eqref{ex:h0-emp-distr}, 
    the entropy condition \ref{uniformcltA4} of \Cref{thm:multiplier-rel-clt} is satisfied.
    Clearly, the moment condition is also satisfied and under the mixing assumption \ref{uniformcltA2n}
    $\G_n$ satisfies a relative CLT in $\ell^\infty(\Fcal)$.
\end{example}

Thus, even when the covariances of $X_i$ drift
the empirical distribution process is still asymptotically equivalent, in the sense of relative weak convergence, to a sequence of centered Gaussian processes with matching covariance structure. 
In particular, this example shows that the theory applies to one of the most classical empirical processes precisely in regimes where ordinary weak convergence may fail because there is no fixed Gaussian limit.
We return to a more sophisticated version of this example in \Cref{sec:applications}.

\subsection{Sequential relative CLT}

The famous invariance principle of Donsker asserts weak convergence of the partial sum process over \emph{iid} data.
A generalization to sequential empirical processes $\Z_n\in \ell^\infty([0,1]\times \Fcal)$
indexed by functions, defined as
    $$\Z_n(s,f)=\frac{1}{\sqrt{k_n}}\sum_{i=1}^{\lfloor sk_n\rfloor}f(X_{n,i})-\E[f(X_{n,i})],$$
can be found in Theorem 2.12.1 of \citet{van2023weak}. 
Beyond the \emph{iid} case, asymptotic normality of $\Z_n$
is hard to prove and requires additional technical assumptions even for finite function classes \citep{dahlhaus2019towards, 10.3150/18-BEJ1088}. 
Specifying $w_{n,i}(s)=\ind\{i\leq \lfloor sk_n\rfloor\}$, relative sequential CLTs are simple corollaries of weighted relative CLTs.

\begin{corollary}[Sequential relative CLT]\label{thm:rel-sequ-clt}
    Under conditions \ref{uniformcltA1}--\ref{uniformcltA4} of \Cref{thm:multiplier-rel-clt},
    the sequential empirical process $\mathbb{Z}_n\in \ell^\infty([0,1]\times \Fcal)$
    satisfies a relative CLT.
\end{corollary}

\begin{proof}
    \Cref{prf:rel-sequ-clt}
\end{proof}

\section{Bootstrap inference}\label{sec:bstrap}

To make practical use of relative CLTs, we need a way to approximate the distribution of limiting GPs. Their covariance operators are a moving target, however, and generally difficult to estimate. The bootstrap is a convenient way to approximate the distribution of limiting GPs, and easy to implement in practice. This section provides some general results on the consistency of multiplier bootstrap schemes for non-stationary time-series.

\subsection{Bootstrap consistency and relative weak convergence}
To define bootstrap consistency in the context of empirical processes, we follow the setup and notation of \cite{bucher2019note}.
We shall see that the usual definition of bootstrap consistency can be equivalently expressed in terms of relative weak convergence.
Let $\mathds{X}_n$ be some sequence of random variables with values in $\Xcal_n$ and
$\mathds{V}_n$ an additional sequence of random variables, independent of $\mathds{X}_n$, with values in $\Vcal_n$ with $\mathds{V}_n^{(j)}$ denoting independent copies 
of $\mathds{V}_n$.
Denote by $\G_n=\G_n(\mathds{X}_n)$ resp. $\G_n^{(j)}=\G_n(\mathds{X}_n,\mathds{V}_n^{(j)})$ a sequence of maps constructed from $\mathds{X}_n$ resp. $\mathds{X}_n,\mathds{V}_n^{(j)}$ with values in $\ell^\infty(T)$
such that each $\G_n(t),\G_n^{(j)}(t)$ is measurable. 
All proofs of the remaining section are found in \Cref{ap:bootstrap} and we omit the asterisk for better readability.

\begin{proposition}\label{thm:bstrap-rwc}
    Assuming that $\G_n$ is relatively compact, the following are equivalent:
    \begin{enumerate}
        \item \label{A1:bstrap-rwc}for $n\to \infty$
        \begin{align*}
            \sup_{h\in \BL_1(\ell^\infty(\Fcal))}\left|\E\left[h(\G_n^{(1)})| \mathds{X}_n \right]-\E\left[h(\G_n)\right]\right|\overset{\Pr^*}{\to} 0,
        \end{align*}
        and $\G_n^{(1)}$ is asymptotically measurable,
        \item \label{A2:bstrap-rwc} it holds $$\left(\G_n,\G_n^{(1)},\G_n^{(2)}\right)\rd\G_n^{\otimes 3}.$$
    \end{enumerate}
    where $\BL_1(\ell^\infty(\Fcal))$ denotes the set of bounded and $1$-Lipschitz
    continuous functions from $\ell^\infty(\Fcal)$ to $\R$.
    Call $\G_n^{(j)}$ a \emph{consistent bootstrap scheme} in any such case.
\end{proposition}
Classically, $\G_n$ is some (transformation of an) empirical process and
consistency of the bootstrap is derived from CLTs for $\G_n$ and $\G_n^{(j)}$.
In view of \Cref{thm:bstrap-rwc}, this approach generalizes to relative CLTs (\Cref{cor:rclt-bstrap}).

\subsection{Multiplier bootstrap}
Now fix some triangular array $\mathds{X}_n=(X_{n,1},\ldots,X_{n,k_n})\in \Xcal^{k_n}$ of random variables with values in a Polish space $\mathcal{X}$
and some family of uniformly bounded functions $w_{n,i}\colon S\to \R$.
Let $\Fcal$ be a set of measurable functions from $\mathcal{X}$ to $\R$ with finite envelope $F$.
Denote by $\mathds{V}_n=(V_{n,1},\ldots, V_{n,k_n})\in \R^{k_n}$ a triangular array of random variables
and by $\mathds{V}_n^{(i)}=(V_{n,1}^{(i)},\ldots,V_{n,k_n}^{(i)})$ independent copies of $\mathds{V}_n$. 
Define $\G_n,\G_n^{(j)}\in \ell^\infty(S\times \Fcal)$ by
\begin{align*}
    \G_n(s,f)       & =\frac{1}{\sqrt{k_n}}\sum_{i=1}^{k_n}w_{n,i}(s)\bigl(f(X_{n,i})-\E[f(X_{n,i})]\bigr), \\
    \G_n^{(j)}(s,f) & =\frac{1}{\sqrt{k_n}}\sum_{i=1}^{k_n}V_{n,i}^{(j)}w_{n,i}(s)\bigl(f(X_{n,i})-\E[f(X_{n,i})]\bigr).
\end{align*}

\begin{proposition}\label{thm:bstrap-univ}
    Let $X_{n,i}$ satisfy the conditions of
    \Cref{thm:multiplier-rel-clt} for some $\gamma>2$ and $\rho$.
    For every $\epsilon>0$, let $\nu_n(\epsilon)$ be such that
     $$\max_{|i-j|\leq \nu_n(\epsilon)}\left|\cov[V_{n,i},V_{n,j}]-1\right|\le\epsilon.$$
    Assume that
    \begin{enumerate}
        \item \label{A1:bstrap-univ}$V_{n,1},\ldots, V_{n,k_n}$ are identically distributed and independent of $(X_{n,i})_{i \in N}$,
        \item \label{A2:bstrap-univ}$\E[V_{n,i}]=0$, $\var[V_{n,i}]=1$, and $\sup_{n}\E[|V_{n,i}|^\gamma]<\infty$,
        \item \label{A3:bstrap-univ}$k_n\beta_n^{X}\bigl(\nu_n(\epsilon)\bigr)^{\frac{\gamma-2}{\gamma}},k_n\beta_n^{V}\bigl(k_n^{\alpha}\bigr)^{\frac{\gamma-2}{\gamma}}\to 0$ for every $\epsilon>0$ and some $\alpha<(\gamma-2)/2(\gamma-1)$.
    \end{enumerate}
    Then, $\G_n^{(j)}$ is a consistent bootstrap scheme.
\end{proposition}

\begin{example}[Block bootstrap with exponential weights]\label{cor:block-bstrap}
    Let $\xi_{i}\sim \mathrm{Exp}(1)$ be \emph{iid}
    for $i\in \Z$  and define $$V_{n,i}=\frac{1}{\sqrt{m_n}}\sum_{j=i+1}^{i+m_n}(\xi_{j}-1).$$ 
    Then $V_{n, i}$ are $m_n$-dependent and it holds $|\cov[V_{n,i}, V_{n, j}] - 1|\le |i-j|/m_n$ for all $|i-j|\leq m_n$. 
    Choosing $\nu_n(\eps) = \lfloor \eps m_n \rfloor$, 
    we see that if 
    \begin{enumerate}
        \item $m_n < k_n^\alpha$ for some $\alpha<(\gamma-2)/2(\gamma-1)$,
        \item $k_n\beta_n^X(\epsilon m_n)^{\frac{\gamma-2}{\gamma}} \to 0$ for every $\epsilon>0$,
    \end{enumerate}
    conditions \ref{A1:bstrap-univ}--\ref{A3:bstrap-univ} of 
    \Cref{thm:bstrap-univ} are satisfied. 
    As in \Cref{ex:gamma-rho} with $\gamma=5$ and $\sup_n\beta_n(m)\leq Cm^{-\rho}$ for some 
    $\rho>6$, we can pick $m_n=\mathcal{O}(k_n^{1/3})$. 
\end{example}

\subsection{Practical inference} \label{sec:bstrap-practical-inference}
The bootstrap process $\G_n^{(j)}$ in the previous section depends on the unknown quantity $\mu_n(i, f) =\E[f(X_{n,i})]$. 
In many testing applications, we have $\E[f(X_{n,i})] = 0$ at least under the null hypothesis; see \cref{sec:applications}. If this is not the case, estimating $\mu_n(i, f)$ consistently may still be possible in simple problems (e.g., fixed-degree polynomial trend), or under triangular array asymptotics where $\mu_n(i, f)$ approaches a simple function (e.g., local stationarity).
For a general, observed non-stationary process $(X_i)_{i \in \N}$, it is impossible to distinguish a random series $(X_i)_{i \in \N}$ with $\E[f(X_i)] = 0$ from a deterministic one with $X_i = \E[f(X_i)] \neq 0$ for all $i \in \N$ a.s. As a consequence, it is generally impossible to quantify the uncertainty in $\G_n$ consistently.
This is a fundamental problem in non-stationary time series analysis, which the relative CLT framework makes transparent. 
A modified bootstrap can still provide valid, but possibly conservative, inference. 

Let $\wh \mu_n(i, f)$ be a potentially non-consistent estimator of $\mu_n(i, f)$,  $\bar \mu_n(i, f)=\E[\wh \mu_n(i, f)]$ its expectation, and define the processes
\begin{align*}
    \wh \G_n^*(s, f)  & = \frac{1}{\sqrt{n}} \sumin V_{n,i} w_{n,i}(s) (f(X_i) - \wh \mu_n(i, f)),  \\
    \bar \G_n^*(s, f) & = \frac{1}{\sqrt{n}} \sumin V_{n,i} w_{n,i}(s) (f(X_i) - \bar \mu_n(i, f)).
\end{align*}
If bias and variance of the mean estimator vanish at an appropriate rate, the approximated bootstrap process $\wh \G_n^*$ is, in fact, consistent.
\begin{proposition}\label{prop:bstr-mean-estim-MSE-rate}
    Suppose the conditions of \cref{thm:bstrap-univ} are satisfied, $\wh \G_n^*$ is relatively compact, and for every $\epsilon > 0$,
    \begin{align*}
        \max_{1 \le i \le n}\E\left[(\wh \mu_n(i, f) - \mu_n(i,f ))^2\right] = o(\nu_n(\epsilon)^{-1})
    \end{align*}
    for all $f\in \Fcal$. 
    Then $\wh \G_n^*$ is consistent for $\G_n$.
\end{proposition}
In particular, this allows to derive bootstrap consistency under local stationarity asymptotics under standard conditions. 

As explained above, this type of consistency should not be expected for the asymptotics of the observed process. However, the bootstrap still provides valid, but conservative inference, 
as shown in the following propositions.

Assume that $\wh \mu_n(i, f)$ converges to $\bar \mu_n(i, f)$ in the following sense:
\begin{align}\label{ass:multiplied-consistent-estimator-of-E}
    \|\wh \G_n^* - \bar \G_n^*\|_{S \times \Fcal}=\sup_{(s,f)\in S\times \Fcal}\Bigl|\frac{1}{\sqrt{n}} \sumin V_{n,i} w_{n,i}(s) \bigl(\wh \mu_n(i, f) - \bar \mu_n(i, f)\bigr)\Bigr|  \overset{\Pr^*}{\to} 0.
\end{align}
Such type of consistency typically holds if $\wh \mu_n - \bar \mu_n$ converges uniformly to zero (e.g., \Cref{prop:multiplier-asy-equivalent}). 

\begin{proposition} \label{prop:boot-conservative-1}
    Define $\wh q_{n, \alpha}^*$ as the $(1 - \alpha)$-quantile of $\|\wh \G_n^*\|_{S \times \Fcal}$.
    Suppose that \eqref{ass:multiplied-consistent-estimator-of-E} and the conditions of \cref{thm:bstrap-univ} hold,
    $\bar \G_n^*$ satisfies a relative CLT and
    $$\var[\bar \G_n^*(s, f)], \var[\G_n(s, f)] \ge \underline \sigma > 0,$$ for all $(s, f) \in S \times \Fcal$ and $n$ large.
    Then,
    \begin{align*}
        \liminf_{n \to \infty} \Pr(\|\G_n\|_{S \times \Fcal} \le \wh q_{n, \alpha}^*) \ge 1 - \alpha.
    \end{align*}
\end{proposition}
If, on the other hand, $\bar \G_n^*$ is \emph{not} relatively compact, it usually holds
\begin{align} \label{eq:boot-inconsistent}
    \Pr(\|\bar \G_n^*\|_{S \times \Fcal} > t_n) \to 1,
\end{align}
for some $t_n \to \infty$. In this case, the bootstrap is over-conservative.

\begin{proposition} \label{prop:boot-conservative-2}
    If \eqref{ass:multiplied-consistent-estimator-of-E}, \eqref{eq:boot-inconsistent} and the conditions of \cref{thm:bstrap-univ}  hold, then
    \begin{align*}
        \lim_{n \to \infty} \Pr(\|\G_n\|_{S \times \Fcal}  \le \wh q_{n, \alpha}^*) = 1.
    \end{align*}
\end{proposition}
    Although the bootstrap quantiles 
    may be conservative, they are still informative: as $n$ tends to infinity, it usually holds
    $\wh q_{n, \alpha}^*/\sqrt{n}\to 0$ (see the proof of \cref{lem:vn-exp-zero}).
    In this sense, the bootstrap quantiles yield potentially too large, 
    yet asymptotically vanishing, uniform confidence intervals for the weighted mean 
    $n^{-1}\sum_{i=1}^{n}w_{n,i}\mu_n(i,\cdot)$.

\section{Applications}
\label{sec:applications}

To end, we explore some exemplary applications that cannot be handled by previous results. The methods are illustrated by monthly mean temperature anomalies in the Northern Hemisphere from 1880 to 2024 provided by NASA \citep{GISTEMP2025,Lenssen2024}, and shown in \cref{fig:temperature-anomalies}.
All proofs are found in \Cref{app:applications}.

\begin{figure}
    \centering
    \fig[0.8]{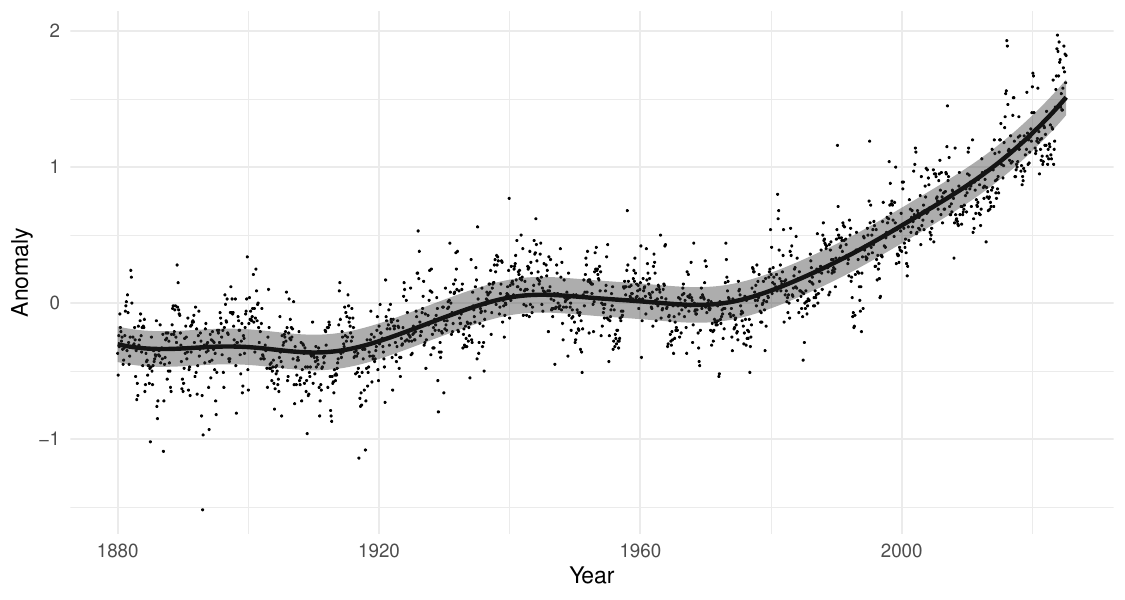}
    \caption{Monthly mean temperature anomalies in the Northern Hemisphere from 1880 to 2024 (dots), kernel estimate of the trend (solid line), and uniform 90\%-confidence interval (shaded area).}
    \label{fig:temperature-anomalies}
\end{figure}

\subsection{Uniform confidence bands for a nonparametric trend}
\label{par:trend-estimation}

Consider a sequence $X_n$ of random variables in $\R$. 
Nonparametric estimation of the trend function $\mu(i) = \E[X_i]$ is a key problem in non-stationary time series analysis.
As explained in \cref{sec:bstrap}, $\mu(i) = \E[X_i]$ is not estimable consistently in general, at least not under the asymptotics of the observed process $(X_i)_{i \in \N}$.
The local stationarity assumption \citep{dahlhaus2012locally} resolves this issue by considering the asymptotics of a hypothetical sequence of time series $(X_{n, i})_{i \in \N}$ in which $\mu_n(i) = \E[X_{n, i}]$ becomes flat as $n \to \infty$. 

The relative CLT framework allows us to consider the asymptotics of the observed process $(X_i)_{i \in \N}$, but we have to accept the fact that $\mu(i) = \E[X_i]$ cannot be estimated consistently. Instead, we aim for estimating the smoothed mean
\begin{align*}
    \mu_b(i) = \frac{1}{nb} \sum_{j=1}^{n} K\left(\frac{j - i}{nb} \right) \mu(j) \quad \text{by} \quad
        \wh \mu_b(i) & = \frac{1}{nb} \sum_{j=1}^{n} K\left(\frac{j - i}{nb} \right) X_j, 
\end{align*}
where $K$ is a kernel function supported on $[-1, 1]$ and $b > 0$ is the bandwidth parameter.
In our setting, the bandwidth $b$ plays a different role than in the local stationarity framework. Since $\E[\wh \mu_b(i)] = \mu_b(i)$ for every value of $b$, we do not require $b \to 0$ as $n \to \infty$. Instead, we may choose $b$ to be a fixed value, quantifying the scale (as a fraction of the overall sample) at which we look at the series. 
\cref{fig:temperature-anomalies} shows the kernel estimate of the trend function $\mu_b(i)$ for $b = 0.05$ as a solid line. Here, $b = 0.05$ means that the smoothing window spans $0.1 \times n$ months, or roughly $15$ years.
Holding $b$ fixed conveniently allows for uniform-in-time asymptotics as shown in the following.

\begin{corollary} \label{prop:kernel1}
    Suppose the kernel $K$ is a twice continuously differentiable probability density function supported on $[-1, 1]$, $\sup_i\E[|X_i|^5] < \infty$, and $\sup_n \beta_n^X(k) \lesssim k^{-7}$. Then for any $b > 0$, the process
    \begin{align*}
       s \mapsto \sqrt{n} \left( \wh \mu_b( sn ) - \mu_b( sn) \right) 
    \end{align*}
    satisfies a relative CLT in $\ell^{\infty}([0, 1])$.
\end{corollary}
To quantify uncertainty of the estimator, we can use the bootstrap. Specifically, let
\begin{align*}
    \wh \mu_b^*(i) & = \frac{1}{nb} \sum_{j=1}^{n} K\left(\frac{j - i}{nb} \right) V_{n, j} (X_j - \wh \mu_{b}(j)), 
\end{align*}
with block multipliers $V_{n, i}$ and $m_n=n^{1/3}$ as in \cref{cor:block-bstrap}.
Let $\wh q_{n, \alpha}$ be the $(1 - \alpha)$-quantile of the distribution of $\sup_{s \in [0, 1]} \sqrt{n} |\wh \mu_b^*(sn)|$.
Then, for $\alpha \in (0, 1)$, we can construct a uniform confidence interval for $\mu_b(sn)$ by
\begin{align*}
    \wh \Ccal_n(\alpha) & = \left[\wh \mu_b - \wh q_{n, \alpha} / \sqrt{n}, \wh \mu_b  +  \wh q_{n, \alpha} / \sqrt{n}\right].
\end{align*}
\begin{corollary} \label{prop:kernel2}
    Suppose the conditions of \cref{prop:kernel1} hold and there is $s \in [0, 1]$ such that
    \begin{align*}
         \bar \sigma_n^2(s) = \var\left[\frac{1}{\sqrt{n}} \sum_{i=1}^{n} V_{n, i} K\left(\frac{i - sn}{nb} \right)  (\mu_{b}(i) - \mu(i)) \right] \to \infty.
    \end{align*}
    Then 
    \begin{align*}
        \liminf_{n\to \infty} \Pr(\mu_b \in \wh \Ccal_n(\alpha)) \ge 1 - \alpha.
    \end{align*}
\end{corollary}
The condition on $\bar \sigma_n^2(s)$ is usually satisfied in the 
time series setting, where a diverging number of the covariances 
$\cov[V_{n, i}, V_{n, j}]$ are close to 1 and $\mu_{b}(i) - \mu(i)$ does not vanish. 
If this is not the case, a similar result could be established using \cref{prop:boot-conservative-1}. 
The confidence interval $\wh \Ccal_n(\alpha)$ is shown as a shaded area in \cref{fig:temperature-anomalies}. The confidence interval is uniformly valid for all $s \in [0, 1]$ and shows a significant, strongly increasing trend in the last 50 years.

\subsection{Testing for time series characteristics}\label{sec:testing-ts-chara}

Suppose $Z_1, Z_2, \dots$ is a non-stationary time series and we want to test
\begin{align*}
    H_0\colon \;  \sup_i \sup_{f \in \Fcal}|\E[f(Z_i)] | = 0 \qquad \text{against}  \qquad    H_1\colon\;  \sup_i \sup_{f \in \Fcal}|\E[f(Z_i)] | \neq 0.
\end{align*}
The functions $f \in \Fcal$ determine which characteristics of the time series we want to control. This framework includes many important applications, two of which are discussed below.
\begin{example}[Equal characteristics of two series]
    Suppose $Z_i = (X_i, Y_i)$, $i \in \N$, and we want to test whether the two time series $X_1, X_2, \dots$ and $Y_1, Y_2, \dots$ have the same characteristics.
    To do so, let $\Fcal = \{f\colon g(x) - g(y), g \in \Gcal\}$,
    so that
    \begin{align*}
        H_0\colon \;  \sup_i \sup_{g \in \Gcal}|\E[g(X_i)] - \E[g(Y_i)] | = 0.
    \end{align*}
    Here $\Gcal$ describes the characteristics of the two time series $X_i, Y_i$ that we want to match.
    Common choices are monomials or indicator functions for testing equality of moments or distribution, respectively.
\end{example}

\begin{example}[Deterministic trends]

    Suppose we want to test for a deterministic trend in a time series  $(X_i)_{i \in \N}$.
    Let $\Delta_h X_i = X_{i + h} - X_{i }$ be the $h$-step forward difference
    operator, and define $\Delta_h^r X_i = \Delta_h^{r - 1} X_{i + h} - \Delta_h^{r - 1} X_{i}$, for $r \ge 2$.
    The null hypothesis is $H_0: \E[\Delta_h^r X_i] = 0$ for all $i \in \N$  and $1 \le r \le R$, and fixed  $h, R \in \N$.
    The parameter $R$ determines the order of the polynomial trend we want to test for. The step-size $h$ allows focusing on long-term trends in the presence of deterministic seasonality.
    This fits into the above framework by letting $Z_i = (\Delta_h X_{i}, \dots, \Delta_h^R X_{i})$, with the convention $\Delta_h^r X_{i} = 0$ for $hr \ge i$, and
    $$\Fcal = \{f\colon f(z_1, \dots, z_R) =  z_i, 1 \le i \le R\}.$$
\end{example}

The multiplier bootstrap allows to construct a test for the general null hypothesis above.
Define the test statistic and its bootstrap version
\begin{align*}
    T_n = \sup_{s \in[0, 1], f \in \Fcal} \left| \frac{1}{\sqrt{n}}\sum_{i=1}^{n} w_{n, i}(s) f(Z_i)\right|, \qquad     T_n^{*} =   \sup_{s \in[0, 1], f \in \Fcal} \left| \frac{1}{\sqrt{n}}\sum_{i=1}^{n} V_{n, i} w_{n, i}(s) f(Z_i)\right|,
\end{align*}
where $w_{n, i}(s)$ are some weights. For example, $w_{n, i}(s) = K((i - sn) / nb)$ allows focusing on time-local deviations from the null hypothesis.

Let $\alpha \in (0, 1)$, and $c_n^*(\alpha)$ be the $(1 - \alpha)$-quantile of the distribution of $T_n^{*}$. We reject $H_0$ iff $T_n > c_n^*(\alpha)$.
Level and consistency of the test can be straightforwardly derived from our general results.

\begin{corollary} \label{appl:test}
    Let the sequence of weights $w_{n, i}(s)$ and $\Fcal$ satisfy the conditions of
    \cref{thm:multiplier-rel-clt}.
    It holds
    $\Pr(T_n > c_n^*(\alpha)) \to \alpha$ under $H_0$,
    and
    $ \Pr(T_n > c_n^*(\alpha)) \to 1 $
    whenever
    \begin{align*}
        \liminf_{n \to \infty} \sup_{s \in \Scal, f \in \Fcal} \left| \frac{1}{n} \sumin w_{n, i}(s) \E[f(Z_i)] \right| = \delta >  0.
    \end{align*}
\end{corollary}

Because $\sup_i |\E[f(Z_i)]| \neq 0$ under $H_1$, 
the distribution of the bootstrap statistic $T_n^*$ 
does not resemble the distribution of $T_n$ under the alternative. 
Consistency still follows from the fact that $T_n / \sqrt{n}$ is (by assumption) bounded away from zero with probability tending to 1, and $T_n^* / \sqrt{n} \overset{\Pr^*}{\to} 0$. The power of the test can be improved if we center by some (non-consistent) estimator $\wh \mu_n(i, f)$ as discussed in \cref{sec:bstrap-practical-inference}.

As an illustration, we apply the above procedure to test for nonstationarity of the monthly mean anomalies. For example, let $Z_i = (X_i, X_{i - 120})$ be a pair of anomalies 10 years apart,  $\Fcal = \{f\colon f(x, y) = \ind(x < t) - \ind(y < t) \colon t \in [-5, 5]\}$, and $w_{n, i}(s) = K_b((i - sn) / nb)$. This gives a Kolmogorov-Smirnov-type statistic
\begin{align*}
    T_n = \sup_{t \in [-5, 5], s \in [0, 1]} \left| \frac{1}{\sqrt{n}} \sum_{i=1}^{n} K\left(\frac{i - sn}{bn}\right) \left( \ind_{X_i < t} - \ind_{X_{i - 120} < t} \right) \right|.
\end{align*}
It is straightforward to show that the conditions of \cref{appl:test} hold, and we can use the multiplier bootstrap to construct a test for the null hypothesis of no nonstationarity. Using $b = 0.05$ and a kernel estimator for $\wh \mu_n(s, f)$ as in the previous section, we get $T_n = 0.69$ and $c_n^*(0.05) = 0.30$, and a bootstrapped $p$-value smaller than $0.0001$, providing strong evidence against the null hypothesis of stationarity.

\bibliographystyle{apalike}
\bibliography{bibliography}
\newpage
\appendix
\section{Proofs for relative weak convergence and CLTs}\label{ap:rwc-rclts}
\begin{lemma}\label{lem:rc-asy-meas-and-rat}
    The sequence $X_n$ is relatively compact if and only if it
    is asymptotically measurable and relatively asymptotically tight.
\end{lemma}

\begin{proof}
    If $X_n$ is relatively compact, it converges to some
    tight Borel law along subsequences.
    Along such subsequences $n_{k}$, $X_{n_{k}}$ is
    asymptotically tight and measurable by Lemma 1.3.8 of \cite{van2023weak}.
    We obtain the sufficiency.
    Note that asymptotic measurability along subsequences implies
    (global) asymptotic measurability.

    For the necessity, for any subsequence there exists a further subsequence $n_{k}$
    such that $X_{n_{k}}$ is asymptotically tight and measurable.
    By Prohorov's theorem \citep[Theorem 1.3.9]{van2023weak},
    there exists a further subsequence $n_{k_i}$ such that $X_{n_{k_i}}$
    converges weakly to some tight Borel law.
    This implies relative compactness of $X_n$.
\end{proof}

\begin{proof}[Proof of \Cref{prop:chara-rc-rwc}]
    Recall that $X_n$ converges weakly to a Borel law $X$ iff $$\E^*[f(X_n)]\to \E[f(X)]$$
    for all $f:\mathbb{D}\to \R$ bounded and continuous.

    $1.\Rightarrow 2.:$\\
    Assume that $X_n\rd Y_n$.
    Then, for all $f:\D \to \R$ bounded and continuous
    \begin{align*}
        \left|\E^*[f(Y_{n_{k}})]-\E[f(X)]\right|
         & \leq \left|\E^*[f(X_{n_{k}})]-\E^*[f(Y_{n_{k}})]\right|+\left|\E^*[f(X_{n_{k}})]-\E[f(X)]\right|
        \\
         & \to 0
    \end{align*}
    whenever $X_{n_{k}}\dc X$ and $X$ Borel measurable.

    $2.\Rightarrow 3.:$\\
    Since $X_n$ is relatively compact,
    every subsequence $X_{n_{k}}$ contains a weakly convergent subsequence $X_{n_{k_i}}\dc X$ with $X$ tight and Borel measurable.
    By assumption also $Y_{n_{k_i}}\dc X$.

    $3. \Rightarrow 1.:$\\
    Given $f$, it suffices to prove that for all subsequences $n_{k}$ there exists
    a further subsequence $n_{k_i}$ such that
    $$\left|\E^*[f(X_{n_{k_i}})]-\E^*[f(X_{n_{k_i}})]\right|\to 0.$$
    Pick $n_{k_i}$ such that both $$X_{n_{k_i}}\dc X\text{ and }Y_{n_{k_i}}\dc X$$
    with $X$ tight and Borel measurable.
    Then,
    \begin{align*}
        \left|\E^*[f(Y_{n_{k_i}})]-\E^*[f(Y_{n_{k_i}})]\right|
         & \leq \left|\E[f(X)]-\E^*[f(Y_{n_{k_i}})]\right|+\left|\E^*[f(X_{n_{k_i}})]-\E[f(X)]\right|
        \\
         & \to 0
    \end{align*}
    At last, characterization \ref{A3:chara-rc-rwc} implies relative compactness of $Y_n$.
\end{proof}

\begin{proof}[Proof of \Cref{prop:portmanteau}]
    We prove this statement by contradiction. Suppose that
    $$ \limsup_{n \to \infty}[ \Pr^*( X_{n} \in S_{n}) - \Pr^*( Y_{n} \in S_n)]  > 0.$$
    Then there is a subsequence $n_k$ of $n$ such that
    \begin{align} \label{eq:portmanteau-contradiction}
        \lim_{i \to \infty} [\Pr^*( X_{n_{k_i}} \in S_{n_{k_i}}) - \Pr^*( Y_{n_{k_i}} \in S_{n_{k_i}})]  > 0,
    \end{align}
    for every subsequence $n_{k_i}$ of $n_k$.
    By \cref{prop:chara-rc-rwc}, $n_k$ has a subsequence $n_{k_i}$ on which $X_{n_{k_i}} \dc Y$ and $Y_{n_{k_i}} \dc Y$ for some tight Borel law $Y$. We may further assume that $S_{n_{k_i}}$ converges to some set $S$ satisfying \eqref{eq:continuity-set}.
    It holds
    \begin{align*}
        &\quad \limsup_{i \to \infty} [\Pr^*( X_{n_{k_i}} \in S_{n_{k_i}}) - \Pr^*( Y_{n_{k_i}} \in S_{n_{k_i}})] \\
        &\le \limsup_{i \to \infty} \Pr^*( X_{n_{k_i}} \in S_{n_{k_i}}) - \liminf_{i \to \infty} \Pr^*( Y_{n_{k_i}} \in S_{n_{k_i}})  \\
        &\le  \limsup_{i \to \infty} \Pr^*\left( X_{n_{k_i}} \in \limsup_{i \to \infty} S_{n_{k_i}}\right) - \liminf_{i \to \infty}\Pr^*\left( Y_{n_{k_i}} \in  \liminf_{i \to \infty}  S_{n_{k_i}}\right)  \\
        &= \limsup_{i \to \infty} \Pr^*( X_{n_{k_i}} \in  S) - \liminf_{i \to \infty}\Pr^*( Y_{n_{k_i}} \in   S).
    \end{align*}
    Further, the Portmanteau theorem \citep[Theorem 1.3.4]{van2023weak} gives
    \begin{align*}
        \Pr^*(Y \in \partial S) \le \Pr^*(Y \in (\partial S)^\delta) \le \limsup_{i \to \infty} \Pr^*(Y_{n_{k_i}} \in (\partial S)^\delta).
    \end{align*}
    Taking $\delta \to 0$, we obtain  $\Pr(Y \in \partial S) = 0$,
    so $S$ is a continuity set of $Y$.
    Now the Portmanteau theorem implies
    \begin{align*}
        \limsup_{i \to \infty} \Pr^*( X_{n_{k_i}} \in  S) - \liminf_{i \to \infty}\Pr^*( Y_{n_{k_i}} \in   S) = 
         \Pr^*(Y \in S) - \Pr^*(Y \in S) = 0,
    \end{align*}
    which contradicts \eqref{eq:portmanteau-contradiction}.
    The case where \eqref{eq:portmanteau-contradiction} holds with reverse sign is treated analogously.
\end{proof}

\begin{proof}[Proof of \Cref{thm:rwc-marginals}]
    If $X_n\rd Y_n$ then $X_n$ is relatively compact
    by \Cref{prop:chara-rc-rwc}.
    This is equivalent to relative asymptotic tightness and asymptotic measurability
    by \Cref{lem:rc-asy-meas-and-rat}.
    Lastly, the relative continuous mapping theorem implies 
    marginal relative weak convergence of $X_n$ and $Y_n$. 

    For the reverse direction, let $n_{k}$ be as subsequence. Let $n_{k_i}$ be a subsequence of $n_{k}$ 
	such that $Y_{n_{k_i}}\dc Y$ with $Y$ a tight Borel law and $X_{n_{k_i}}$ is asymptotically tight.
    In particular, all marginals of $Y_{n_{k_i}}$ converge weakly to the marginals of $Y$ by the continuous 
	mapping theorem. 
	Since all marginals of $X_n$ and $Y_n$ are relatively weakly convergent, 
    this implies the convergence of all marginals of $X_{n_{k_i}}$
    to the marginals of $Y$.
    Together with asymptotic tightness of $X_{n_{k_i}}$, this implies 
    the convergence $X_{n_{k_i}}\dc Y$ by Theorem 1.5.4 of \cite{van2023weak}.
    By characterization \ref{A3:chara-rc-rwc} of \Cref{prop:chara-rc-rwc} we obtain $X_n\rd Y_n$.
\end{proof}

\begin{proof}[Proof of \Cref{prop:rcm}]
    For all $f:\E\to \R$ bounded and continuous, the composition
    $f\circ g:\D\to \R$ is bounded and continuous.
    Thus, $\left|\E^*\left[f\circ g(X_n)\right]-\E^*[f\circ g\left(Y_n\right)]\right|\to 0$
    for all such $f$ by definition of $X_n \rd Y_n$.
    This yields the claim.
\end{proof}

\begin{proof}[Proof of \Cref{prop:ext-cont-map}]
    Any subsequence of $n$ contains a further subsequence such that
    $Y_{n_k}\dc Y$ and there exists $g:\D\to \E$ such that $g_{n_k}(x_k)\to g(x)$ for all $x_k\to x$ in $\D$.
    Theorem 1.11.1 of \cite{van2023weak} implies $g_n(Y_{n_k})\dc g(Y)$.
    In particular, $g_n(Y_n)$ is relatively compact and
    \Cref{prop:chara-rc-rwc} yields the second claim.
\end{proof}

\begin{definition} \label{def:hadamard-diff}
    Let $\D,\E$ be metrizable topological vector spaces, i.e.,
    metric spaces equipped with a vector space structure such that addition and scalar multiplication are continuous.
    A map $\phi: \D\to \E$ is called Hadamard-differentiable at $\theta\in \D$
    if there exists $\phi^\prime_\theta:\D\to \E$ continuous and linear
    such that
    \begin{align*}
        \frac{\phi(\theta+t_nh_n)-\phi(\theta)}{t_n}\to \phi_\theta^\prime(h)
    \end{align*}
    for all $t_n\to t$ in $\R$ and $h_n\to h$ in $\D$.
    $\phi$ is continuously Hadamard-differentiable in an open subset $U\subset \D$
    if $\phi$ is Hadamard-differentiable for all $\theta\in U$ and
    $\phi^\prime_\theta$ is continuous in $\theta\in U$.
\end{definition}

\begin{proof}[Proof of \cref{prop:rel-delta-meth}]
	Note that $g_n:\D\to \E,x\mapsto\phi^\prime_{\theta_n}(x)$
	satisfies the condition of \Cref{prop:ext-cont-map}
	since $\phi$ has continuous Hadamard-differentials and $\theta_n\in \D_0$ is relatively compact.
	Thus, $\phi^\prime_{\theta_n}(Y_n)$ is relatively compact since
	$Y_n$ is.
	By \ref{A3:chara-rc-rwc} of \Cref{prop:chara-rc-rwc} and descending to subsequences, we may assume $Y_n\dc Y$, $\theta_n\to \theta$
	and $\phi^\prime_{\theta_n}(Y_n)\dc \phi^\prime_{\theta}(Y)$.
	By Theorem 3.10.4 in \cite{van2023weak},
	we obtain
	$$r_n\left(\phi(X_n)-\phi(\theta_n)\right)\dc \phi^\prime_{\theta}(Y).$$
	Then, \ref{A3:chara-rc-rwc} of \Cref{prop:chara-rc-rwc} yields the claim.
\end{proof}

\begin{lemma}\label{lem:equ-rel-comp}
    If $X_n\in \ell^\infty(T)$, then, the sequence $X_n$ is
    relatively compact if and only if it is relatively asymptotically tight and
    $X_n(t)$ is asymptotically measurable for all $t\in T$.
\end{lemma}

\begin{proof}
    By \Cref{lem:rc-asy-meas-and-rat}, $X_n$ is relatively compact if and only if it is
    relatively asymptotically tight and asymptotically measurable.
    By definition, any sequence $X_n$ is asymptotically measurable
    if and only if any subsequence $n_{k}$ contains a further subsequences $n_{k_i}$
    such that $X_{n_{k_i}}$ is asymptotically measurable.
    By Lemma 1.5.2 of \cite{van2023weak}
    being asymptotically measurable is equivalent to
    $X_n(t)$ being asymptotically measurable for all $t\in T$
    whenever $X_n$ is relatively asymptotically tight.
    All together, this implies the equivalence.
\end{proof}

\begin{corollary}[Relative Cramer-Wold device]\label{cor:rel-cramer-wold}
    Let $X_n$ and $Y_n$ be two sequences of $\R^d$-valued random variables.
    If $X_n$ is uniformly tight, then,
    $$X_n\rd Y_n \text{ if and only if }
        t^TX_n\rd t^TY_n$$ for all $t\in \R^d$.
\end{corollary}

\begin{proof}
    The only if part follows by the relative continuous mapping theorem. 

    For the other direction, assume $t^TX_n\rd t^TY_n$ for all $t\in \R^d$.
    Note that $t^TX_n$ is uniformly tight, i.e., relatively compact, for all $t\in \R^d$.
    We use characterization \ref{A3:rel-clt-chara} of \Cref{prop:chara-rc-rwc}.
    Let $n_k$ be a subsequence.
    Since $X_n$ is uniformly tight, there exists a subsequence $X_{n_{k_j}}\dc X$.
    Then, also $t^TX_{n_{k_j}}\dc t^TX$.
    By characterization \ref{A2:rel-clt-chara} of \Cref{prop:chara-rc-rwc} and $t^TX_n\rd t^TY_n$,
    it follows $t^TY_{n_{k_j}}\dc t^TX$ for all $t\in \R^d$.
    By the Cramer-Wold device, we derive $Y_{n_{k_j}}\dc X$.
    This proves the claim by characterization \ref{A3:rel-clt-chara} of \Cref{prop:chara-rc-rwc}.
\end{proof}

\subsection{Relative central limit theorems}

\begin{proof}[Proof of \Cref{lem:rclt-marginal-rclt}]
	Let $Y_n$ satisfy a relative CLT. \ref{A1:rclt-marginal-rclt}
	follows by definition. By \ref{A3:chara-rc-rwc}, $Y_n$ is also relatively compact. 
	Then, \ref{A2:rclt-marginal-rclt} follows by \Cref{lem:equ-rel-comp}.
	For \ref{A3:rclt-marginal-rclt}, observe that the marginals of $N_{Y_n}$ 
	are relatively weakly convergent to the marginals of $Y_n$ by the relative continuous mapping theorem. 
	Furthermore, the marginals of $N_{Y_n}$ are tight and measurable multivariate Gaussians corresponding 
	to the marginals of $Y_n$. Thus $Y_n$ satisfies marginal relative CLTs. This proves the necessity.

	For the sufficiency, it suffices to prove $Y_n\rd N_{Y_n}$ and that $N_{Y_n}$
	is relatively compact. 
	Recall that $Y_n$ and $N_{Y_n}$ are stochastic processes, hence, $N_{Y_n}(t)$ and $Y_n(t)$ are measurable by assumption. 
	Thus, $N_{Y_n}$ and $Y_n$ are relatively compact by \Cref{lem:equ-rel-comp} and \ref{A2:rclt-marginal-rclt}.
	Next, observe that marginal relative CLTs of $Y_n$ imply 
	$$\left(Y_n(t_1),\ldots,Y_n(t_k)\right)\rd \left(N_{Y_n}(t_1),\ldots,N_{Y_n}(t_k)\right)$$
    since corresponding Gaussians are unique in distribution.
    Then, \Cref{thm:rwc-marginals} proves $Y_n\rd N_{Y_n}$ which finishes the proof.
\end{proof}

\begin{proposition}\label{prop:approx-gps-implies-conv-cov}
    If there exists a relatively asymptotically tight sequence of tight and Borel measurable
    GPs corresponding to $Y_n$, then, every subsequence
    of $n$ contains a further subsequence $n_{k_i}$ such that
    $\cov[Y_{n_{k_i}}(t),Y_{n_{k_i}}(s)]$ converges for all $s,t\in T$.
\end{proposition}

\begin{proof}
    Denote by $N_{Y_n}$ a relatively asymptotically tight sequence of GPs corresponding to $Y_n$.
    Any subsequence contains a further subsequence $n_{k_i}$ such that
    $N_{Y_{n_{k_i}}}$ converges weakly.
    In particular, all marginals of $N_{Y_{n_{k_i}}}$ converge weakly.
    Recall that a sequence of centered multivariate Gaussians converges weakly
    if and only if their corresponding covariances converges, for instance, by considering 
	characteristic functions.
    Thus, we obtain convergence of all covariances
    $$\cov[N_{Y_{n_{k_i}}}(t),N_{Y_{n_{k_i}}}(s)]=\cov[Y_{n_{k_i}}(t),Y_{n_{k_i}}(s)].$$
\end{proof}

\begin{corollary}\label{cor:gaussian-dominated-entropy}
    If $T$ is finite,
    the following are equivalent:
    \begin{enumerate}
        \item there exists a relatively asymptotically tight sequence of tight and Borel measurable
              GPs corresponding to $Y_n$.
        \item  $\sup_n\var[Y_{n}(t)]<\infty$ for all $t\in T$.
    \end{enumerate}
\end{corollary}

\begin{proof}
    For the sufficiency, \Cref{prop:approx-gps-implies-conv-cov} implies
    that all sequences of covariances $\cov[Y_n(s),Y_n(t)]$
    are relatively compact, equivalently, uniformly bounded.

    For the necessity, identify $Y_n$ with $\left(Y_n(t_1),\ldots,Y_n(t_d)\right)$ for $T=\{t_1,\ldots,t_d\}$.
    Construct $N_{Y_n}\sim \mathcal{N}(0,\Sigma_n)$ with $\Sigma_n$ the covariance matrix of $Y_n$.
    Then, each $N_{Y_n}$ is measurable and tight and $\sup_n\var[Y_{n}(t)]<\infty$ implies that all covariance $\cov[Y_n(s),Y_n(t)]$
    are relatively compact.
    Thus, every subsequence $n_k$ contains a further subsequence $n_{k_i}$ such that
    all covariances $\cov[Y_{n_{k_i}}(s),Y_{n_{k_i}}(t)]$ converge.
    Thus, $N_{Y_{n_{k_i}}}$ converges weakly.
    We obtain relative compactness, hence, asymptotic tightness
    of $N_{Y_n}$.
\end{proof}

\begin{proof}[Proof of \Cref{prop:rel-clt-chara}]
    By \Cref{cor:gaussian-dominated-entropy},
    there exists an asymptotically tight sequence of tight Borel measurable
    GPs $N_{Y_n}$ corresponding to $Y_n$ if and only if $$\sup_{n\in \N,i\leq d}\var[Y_{n}^{(i)}]<\infty.$$
    Equivalently, if all subsequences $n_k$ contain a further subsequence $n_{k_i}$ such that $\Sigma_{n_{k_i}}$
    converges.
    Combined with the fact that a sequence of centered Gaussians converges
    weakly to some centered Gaussian with covariance matrix $\Sigma$ if and only if its corresponding sequence of covariances converges to $\Sigma$,
    the equivalences follow from \Cref{prop:chara-rc-rwc}.
\end{proof}

\begin{theorem}[Relative Lindeberg CLT]\label{thm:rel-lind-feller}
    Let $X_{n,1},\ldots,X_{n,k_n}\in \R^d$ be a triangular array 
    of independent random vectors with finite variance. 
    Assume 
    \begin{align*}
        \frac{1}{k_n} & \sum_{i=1}^{k_n}\E\left[\|X_{n,i}\|^2\ind_{\{\|X_{n,i}\|^2>k_n\epsilon\}}\right]\to 0
    \end{align*}
    for all $\eps>0$ and for all  $n\in \N, l=1,\ldots,d$
    \begin{align*}
        \frac{1}{k_n} & \sum_{i=1}^{k_n}\var\left[X_{n,i}^{(l)}\right]\leq K\in \R.
    \end{align*}
    Then, the scaled sample average $\sqrt{k_n}\left(\bar{X}_n-\E\left[\bar{X}_n\right]\right)$ satisfies a relative CLT.
\end{theorem}

\begin{proof}
    Let $k_{l_n}$ be a subsequence of $k_n$ such that
    $$\frac{1}{k_{l_n}}\sum_{i=1}^{k_{l_n}}\cov[X_{l_n,i}]\to \Sigma$$
    converges.
    Observe that the Lindeberg condition implies
    $$\frac{1}{k_{l_n}}\sum_{i=1}^{k_{l_n}}\E\left[\|X_{l_n,i}\|^21_{\{\|X_{l_n,i}\|^2>k_{l_n}\epsilon\}}\right]\to 0$$
    for all $\epsilon>0$.
    We apply Proposition 2.27 of \cite{van2000asymptotic}
    to the triangular array
    $Y_{n,1},\ldots Y_{n,k_{l_n}}$
    with $Y_{n,i}=k_{l_n}^{-1/2}X_{l_n,i}$
    to derive
    $$\sqrt{k_{l_n}}\left(\overline{X}_{l_n}-\E\left[\overline{X}_{l_n}\right]\right)\dc \mathcal{N}(0,\Sigma).$$
    By \ref{A2:rel-clt-chara} of \Cref{prop:rel-clt-chara} we derive the claim.
\end{proof}

\begin{proof}[Proof of \Cref{prop:ex-tight-GPs}]
    By Kolmogorov's extension theorem, there exist centered GPs $\{N_{Y_n}(t)\colon t\in T\}$
    with covariance function given by $(s,t)\mapsto \cov[Y_n(s),Y_n(t)]$.
    Since $(T,\rho_{n})$ is totally bounded (by finiteness of the covering numbers), $(T,\rho_n)$ is separable and, thus there exists a separable
    version of $\{N_{Y_n}(t)\colon t\in T\}$
    with the same marginal distributions (Section 2.3.3 of \cite{van2023weak}).
    Without loss of generality, assume that $\{N_{Y_n}(t)\colon t\in T\}$ is separable.
    Then, 
    \begin{align*}
        \E[\|N_{Y_n}\|_T]&\leq C\int_{0}^{\infty}\sqrt{\ln N(\epsilon/2,T,\rho_{n})}d\epsilon<\infty
        \\
        \E\left[\sup_{\rho_{n}(s,t)\leq \delta}|N_{Y_n}(t)-N_Y{_n}(s)|\right]&\leq C\int_{0}^{\delta}\sqrt{\ln N(\epsilon/2,T,\rho_{n})}d\epsilon 
    \end{align*}
    for some constant $C$ by Corollary 2.2.9 of \cite{van2023weak}.
    The first inequality implies that each $\{N_{Y_n}(t)\colon t\in T\}$
    has bounded sample paths almost surely, hence, without loss of generality
    $\{N_{Y_n}(t)\colon t\in T\}$ induces a map $N_{Y_n}$ with values in $\ell^\infty(T)$.
    The second together with Markov's inequality imply that each $N_{Y_n}$ viewed as a constant sequence is (asymptotically)
    uniformly $\rho_n$-equicontinuous in probability. 
    Hence, there exists a version of $N_{Y_n}$ which is tight and Borel measurable (Example 1.5.10 of \cite{van2023weak}).
\end{proof}

\begin{proof}[Proof of \Cref{prop:asy-tight-GPs}]
    We derive
    \begin{align*}
        \E\left[\sup_{\rho_{n}(s,t)\leq \delta}|N_{Y_n}(t)-N_{_n}(s)|\right]&\leq C\int_{0}^{\delta}\sqrt{\ln N(\epsilon/2,T,\rho_{n})}d\epsilon 
    \end{align*}
    for some constant $C$ independent of $n$ by Corollary 2.2.9 of \cite{van2023weak}.
    By \ref{A3:asy-tight-GPs}, for every sequence $\delta\to 0$ there exists $\epsilon(\delta)\to 0$ such that
    $d(s,t)<\delta$ implies $\rho_n(s,t)<\epsilon(\delta)$ for all $n$ large.
    Accordingly,
    \begin{align*}
        \limsup_n\E\left[\sup_{d(s,t)\leq \delta}|N_{Y_n}(t)-N_{Y_n}(s)|\right]&\leq \limsup_n\E\left[\sup_{\rho_{n}(s,t)\leq \epsilon(\delta)}|N_{Y_n}(t)-N_{_n}(s)|\right]\\
        &\leq \limsup_nC\int_{0}^{\epsilon(\delta)}\sqrt{\ln N(\epsilon/2,T,\rho_{n})}d\epsilon 
    \end{align*}
    Taking the limit $\delta\to0$ the right hand side converges to zero by \ref{A2:asy-tight-GPs}.
    Together with Markov's inequality, we obtain that
    $N_{Y_n}$ is asymptotically uniformly $d$-equicontinuous in probability.

    Since $\sup_{n}\var[Y_n(t)]<\infty$ for all $t$, all sequences $N_{Y_n}(t)$ are relatively asymptotically tight.
    In $\R$ relative asymptotic tightness, relative compactness and asymptotic tightness agree. 
    Then, Theorem 1.5.7 of \cite{van2023weak} proves that 
    $N_{Y_n}$ is asymptotically tight. 
\end{proof}

\section{Tightness under bracketing entropy conditions}
\label{sec:proof-bracketing}

The proof of \cref{thm:asy-tightness-sup-br} is based on a long sequence of well-known arguments: 
We group the observations $X_{n,i}$ in alternating blocks of equal size 
and apply maximal coupling. 
This yields random variables $X_{n,i}^*$ which corresponding blocks 
are independent and the empirical process $\G_n^*$ where $X_{n,i}$ is replaced by $X_{n,i}^*$. 
We obtain a bound on the first moment of $\G_n$ in terms of 
the first moment of $\G_n^*$ (\Cref{lem:coupling}). 
Because $\G_n^*$ consists of independent blocks, we can derive a
Bernstein type inequality bounding 
the first moment of $\G_n^*$, hence of $\G_n$, in terms of $\Fcal_n$ provided that 
$\Fcal_n$ is finite (\Cref{lem:bernstein}). 
For any fixed $n$, we use a chaining argument in order to reduce to finite $\Fcal_n$
which, in combination with the Bernstein inequality, yields 
a bound of the first moment of $\G_n$ in terms of the bracketing entropy (\Cref{thm:chaining}). 
Under the conditions of \Cref{thm:asy-tightness-sup-br}, this yields 
asymptotic equicontinuity of $\G_n$ which implies relative compactness 
of $\G_n$. In combination with \Cref{prop:ex-tight-GPs,prop:asy-tight-GPs}, we obtain
\Cref{thm:asy-tightness-sup-br}. 

\subsection{Coupling}\label{sec:coupling}
Let $m_n$ be a sequence in $\N$. 
Suppose for simplicity that $k_n$ is a multiple of $2m_n$
and group the observations $X_{n,1},\ldots,X_{n,k_n}$ in alternating blocks of size $m_n$.
By maximal coupling \citep[Theorem 5.1]{rio2017asymptotic}, there are random vectors $$U_{n,j}^* = \left( X_{n,\left( j - 1\right) m_n + 1}^*, \ldots, X_{n,jm_n}^*\right)  \in \Xcal^n$$ such that
\begin{itemize}
	\item $U_{n,j} = \left( X_{n,\left( j - 1\right) m_n + 1}, \ldots, X_{n,jm_n}\right)   \stackrel{d}{=} U^*_{n,j}$ for every $j = 1, \ldots, k_n/m_n$,
	\item each of the sequences  $\left( U_{n,2j}^*\right) _{j = 1, \ldots, k_n/\left( 2m_n\right)  }$ and $\left( U_{n,2j - 1}^*\right) _{j = 1, \ldots, k_n/\left( 2m_n\right) }$ are independent,
	\item $\Pr\left( X_{n,j}\neq X_{n,j}^*\right)  \leq \beta_{n}\left( m_n\right) $ for all $j$,
	\item $\Pr\left( \exists j \colon U_{n,j}\neq U_{n,j}^*\right)  \leq \left( k_n/m_n\right) \beta_{n}\left( m_n\right) $.
\end{itemize}
Define the coupled empirical process $\G_n^{*}\in \ell^\infty\left( T\right) $ as $\G_n $, but with all $X_{n,j}$ replaced by $X_{n,j}^*$.

In what follows, we will replace $\E^*$ (indicating the outer expectation) 
by $\E$ for better readability.
We will provide an upper bound on $\E\|\G_n^*\|_T$ 
for any fixed $n$. 
In such case, we identify $\G^*_n$ with the empirical process 
$$\G_n^*\left( f\right) =\frac{1}{\sqrt{k_n}}\sum_{i=1}^{k_n}f\left( X^*_{n,i}\right) -\E\left[f\left( X^*_{n,i}\right) \right]$$
indexed by the function class $\Fcal_n$.
For fixed $n$, we drop the index $n$, i.e., write
$\Fcal_n=\Fcal$, $X_{n,i}=X_i$, $m_n=m$ etc. and assume without loss of generality that $k_n=n$.

\begin{lemma} \label{lem:coupling}
	For any class $\Fcal$ of functions $f\colon  \Xcal \to \R$ with $\sup_{f \in \Fcal}\|f\|_{\infty} \le B$ and any integer $1\leq m \leq n/2$, it holds
	\begin{align*}
		\E \|\G_n\|_{\Fcal} \le \E \|\G_n^*\|_{\Fcal} + B \sqrt{n} \beta_{n}\left( m\right) .
	\end{align*}
\end{lemma}
\begin{proof}
	We have
	\begin{align*}
		\E \|\G_n\|_{\Fcal}
		 & \le \E \sup_{f \in \Fcal}| \G_n^{*} f |  + \frac{B}{\sqrt{n}} \E\left[\sum_{i = 1}^n \ind(X_j \neq X_j^*)\right] 
		  \le \E\| \G_n^{*}  \|_{\Fcal} + B \sqrt{n} \beta_{n}\left( m\right) .
	\end{align*}
\end{proof}

\subsection{Bernstein inequality}
\begin{lemma} \label{lem:bernstein}
	Let $\Fcal$ be a finite set of functions $f\colon \Xcal \to \R$ with
	\begin{align*}
		\quad \|f\|_\infty \le B, \quad \frac{1}{n}\sum_{i,j=1}^{n}|\cov[f(X_i),f(X_j)]|\leq K\delta^2
	\end{align*}
	for all $f\in \Fcal$.
	Then, for any $1\leq m\leq n/2$ it holds
	\begin{align*}
		\E \| \G_n   \|_{\Fcal}  \lesssim \delta   \sqrt{ \ln_+ \left( |\Fcal|\right) } + \frac{  m B    \ln_+ \left( |\Fcal|\right)  }{\sqrt{n}} + B \sqrt{n} \beta_{n}\left( m\right)
	\end{align*}
	where the constant only depends on $K$ and $\ln_+(x)=\ln(1+x)$.
\end{lemma}

\begin{proof}
	We have
	\begin{align*}
		\E \|\G_n\|_{\Fcal} \le \E\| \G_n^{*}  \|_{\Fcal} + B \sqrt{n} \beta_{n}\left( m\right),
	\end{align*}
	by \cref{lem:coupling}.
	Defining
	\begin{align*}
		A_{j, f} =  \sum_{i = 1}^m f\left( X_{\left( j - 1\right) m + i}^*\right)  - \E\left[f\left( X_{\left( j - 1\right) m + i}^*\right) \right],
	\end{align*}
	we can write
	\begin{align*}
		\sqrt{n}| \G_n^{*}(f) |
		 & = \left| \sum_{j = 1}^{n/m} A_{j, f}  \right|
		\le \left| \sum_{j = 1}^{n/\left( 2m\right) }  A_{2j, f} \right|  + \left| \sum_{j = 1}^{n/\left( 2m\right) }  A_{2j - 1, f}  \right|.
	\end{align*}
	The random variables in the sequence $\left( A_{2j, f}\right) _{j = 1}^{n/\left( 2m\right) }$ are independent and so are those in $\left( A_{2j - 1, f}\right) _{j = 1}^{n/\left( 2m\right) }$.
	We apply Bernstein's inequality \citep[Lemma 2.2.10]{van2023weak}.
	Note that $|A_{j, f}| \le 2 mB$,
	hence $\E\left[|A_{j, f}|^k\right]\leq \left( 2mB\right) ^{k-2}\var\left[A_{j, f}\right]$ for $k\geq 2$.
	We obtain 
	\begin{align*}
		\frac{2m}{n}\sum_{i=1}^{n/\left( 2m\right) } \E\left[|A_{j, f}|^k\right] 
		&\leq \left( 2mB\right) ^{k-2} \frac{2m}{n}\sum_{i=1}^{n/\left( 2m\right) } \var\left[A_{j, f}\right]
		\\
		&\leq \left( 2mB\right) ^{k-2}\frac{2m}{n}\sum_{i=1}^{n} |\cov[f(X_i),f(X_j)]|
		\\
		&\leq \left( 2mB\right) ^{k-2} 2Km \delta^2.
	\end{align*}
	Using Bernstein's inequality for independent random variables gives
	\begin{align*}
		\Pr\left( \biggl| \sum_{j = 1}^{n/\left( 2m\right) } A_{2j, f} \biggr| > t \right) 
		& \le 2\exp\left( - \frac{1}{2}\frac{t^2 }{Kn\delta^2 + 2t mB  }\right)  
	\end{align*}
	and the same bounds holds for the odd sums.
	Altogether we get 
	\begin{align*}
		\Pr\left( | \G_n^{*}(f) | > t\right) 
		&\le \Pr\left( \biggl| \sum_{j = 1}^{n/(2m)} A_{2j, f}  \biggr| > t \sqrt{n} / 2 \right)  + \Pr\left( \biggl| \sum_{j = 1}^{n/ (2m) } A_{2j - 1, f}  \biggr| > t \sqrt{n} / 2 \right)   \\
		&\le 4 \exp \left(  -\frac{1}{8}\frac{t^2}{ K\delta^2 + tmB / \sqrt{n}} \right)  
	\end{align*} 
	The result follows upon converting this to a bound on the expectation \citep[e.g.,][Lemma 2.2.13]{van2023weak}.
\end{proof}

\subsection{Chaining}
We will abbreviate 
\begin{align*}
	\|f\|_{a,n}&=\left(\frac{1}{n}\sum_{i=1}^{n}\E\left[|f\left( X_i\right) |^a\right]\right)^{1/a}\\
	N_{\left[\right]}(\epsilon)  &= N_{\left[\right]}\left( \epsilon, \Fcal, \|\cdot\|_{\gamma,n}\right) .
\end{align*}
Let us first collect some properties of $\|\cdot\|_{\gamma,n}$ and 
    $\|\cdot\|_{\gamma,\infty}$.

\begin{lemma}\label{lem:properties-norm}
	The following holds:
	\begin{enumerate}
		\item \label{C1:properties-norm}$\|\cdot\|_{a,n}$ defines a semi-norm. 
		\item \label{C2:properties-norm}$\|\cdot\|_{a,n}\leq \|\cdot\|_{b,n}$ for $a\leq b$.
		\item \label{C3:properties-norm}$\|f\ind_{|f|> K}\|_{a,n}\leq K^{a-b}\|f\|_{b,n}^{b/a}$ for all $K>0$ and $a\leq b$.
	\end{enumerate}
\end{lemma}
\begin{proof}
	Positivity and homogeinity of $\|\cdot\|_{\gamma,n}$ follow clearly and 
	the triangle inequality follows by
	\begin{align*}
		\|h+g\|_{\gamma,n}
		&=  \left(\frac{1}{n} \sum_{i= 1}^{n} \|(h+g)(X_i)\|_\gamma^{\gamma}\right)^{1/\gamma}
		\\
		&\leq \left(\frac{1}{n} \sum_{i= 1}^{n} \bigl[\|h(X_i)\|_\gamma+\|g(X_i)\|_\gamma\bigr]^{\gamma}\right)^{1/\gamma}
		\\
		&\leq \left(\frac{1}{n} \sum_{i= 1}^{n} \|h(X_i)\|_\gamma^{\gamma}\right)^{1/\gamma}+\left(\frac{1}{n} \sum_{i= 1}^{n} \|g(X_i)\|_\gamma^{\gamma}\right)^{1/\gamma}
		&\text{Minkowski's inequality}
		\\
		&=\|h\|_{\gamma,n}+\|g\|_{\gamma,n}
	\end{align*}
	for all $h,g$. 
	Next, 
	\begin{align*}
		\|f\|_{a,n}&=\left(\frac{1}{n}\sum_{i=1}^{n}\|f\left( X_i\right) \|_a^a\right)^{1/a}
		\\
		&\leq \left(\frac{1}{n}\sum_{i=1}^{n}\|f\left( X_i\right) \|_b^a\right)^{1/a}
		\\		
		&\leq \left(\frac{1}{n}\sum_{i=1}^{n}\|f\left( X_i\right) \|_b^b\right)^{1/b}
		&\text{by Jensen's inequality.}
	\end{align*}
	Lastly, 
	note 
	$$K^{b-a}|f(X_i)\ind_{|f(X_i)|> K}|^a\leq |f(X_i)|^b.$$
	Thus, 
	\begin{align*}
		\|f\ind_{|f|\geq K}\|_{a,n}
		&
		= \left(\frac{1}{n}\sum_{i=1}^{n}\E\left[|f(X_i)\ind_{|f(X_i)|> K} |^a\right]\right)^{1/a}
		\\
		&\leq 
		K^{a-b}\left(\frac{1}{n}\sum_{i=1}^{n}\E\left[|f(X_i)|^b\right]\right)^{1/a}
		\\
		&= K^{a-b}\|f\|_{b,n}^{b/a}
	\end{align*}
\end{proof}

\begin{theorem}\label{thm:chaining}
	Let $\Fcal$ be a class of functions $f\colon \Xcal \to \R$ with envelope $F$ and for some $\gamma \ge 2$,
	\begin{align*}
		\| f \|_{\gamma,n} \le \delta \qquad \frac{1}{n} \sum_{i,j= 1}^{n} |\cov[h(X_i),h(X_j)]|\leq K_1 \|h\|_{\gamma,n}^2
	\end{align*}
	for all $f\in \Fcal$ and $h:\Xcal\to \R$ bounded and measurable.
	Suppose that $\sup_n \beta_n(m) \leq K_2 m^{-\rho}$ for some $\rho \ge \gamma /(\gamma -2 )$.
	Then, for any $n \ge 5$, $m \ge 1, B \in (0, \infty)$, 
	\begin{align*}
		\E\|\G_n\|_{\Fcal} &\lesssim \int_0^\delta \sqrt{\ln_+ N_{\left[\right]}(\epsilon) } d \epsilon 
		\\&\quad + \frac{ m B  \ln_+ N_{\left[\right]}(\delta)  }{\sqrt{n}} + \sqrt{n} B\beta_n(m) + \sqrt{n}\|F\ind\{F>B\}\|_{1,n} + \sqrt{n} N_{[]}^{-1}(e^n)
	\end{align*}
	with constants only depending on $K_1,K_2$.
	If the integral is finite, then $\sqrt{n} N_{[]}^{-1}(e^n)\to 0$ for $n\to \infty$. 
\end{theorem}
Let us first derive some useful corollaries.
\begin{corollary}\label{cor:chaining-independent}
	Let $\Fcal$ be a class of functions $f\colon \Xcal \to \R$ with envelope $F$, $X_{n,i}$ are independent and 
		$\| f \|_{2,n} \le \delta$
	for all $f\in \Fcal$.
	Then, for any $n \ge 5$, $B \in (0, \infty)$, 
	\begin{align*}
		\E\|\G_n\|_{\Fcal} &\lesssim \int_0^\delta \sqrt{\ln_+ N_{\left[\right]}(\epsilon,\Fcal,\|\cdot\|_{2,n}) } d \epsilon 
		\\&\quad + \frac{B  \ln_+ N_{\left[\right]}(\delta,\Fcal,\|\cdot\|_{2,n})  }{\sqrt{n}} + \sqrt{n}\|F\ind\{F>B\}\|_{1,n} + \sqrt{n} N_{[]}^{-1}(e^n)
	\end{align*}
	with constants only depending on $K_1,K_2$.

\end{corollary}

\begin{proof}
	It holds $$\frac{1}{n} \sum_{i,j= 1}^{n} |\cov[h(X_i),h(X_j)]|=\frac{1}{n} \sum_{i= 1}^{n} \var[h(X_i)]\leq 2\|h\|_{2,n}$$
	and the $\beta$-coefficients are $0$ for all $m\geq 1$. 
	Applying \Cref{thm:chaining} with $m=1$ yields the claim. 
\end{proof}

\begin{theorem}\label{thm:bracketing-coupled}
	Let $\Fcal$ be a class of functions $f\colon \Xcal \to \R$ with envelope $F$ and for some $\gamma > 2$,
	\begin{align*}
		\| f \|_{\gamma,n} \le \delta \qquad \frac{1}{n} \sum_{i,j= 1}^{n} |\cov[h(X_i),h(X_j)]|\leq K_1 \|h\|_{\gamma,n}^2
	\end{align*}
	for all $f\in \Fcal$ and $h:\Xcal\to \R$ bounded and measurable.
	Suppose that $\max_n \beta_n(m) \leq K_2 m^{-\rho}$ for some $\rho \ge \gamma /(\gamma -2 )$.
	Then, for any $n \ge 5$,
	\begin{align*}
		\E\|\G_n \|_{\Fcal}   \lesssim  \int_0^\delta  \sqrt{\ln_+ N_{\left[\right]}(\epsilon) } d \epsilon + \frac{  \|F\|_{\gamma, n} [\ln N_{\left[\right]}(\delta)]^{ [1 - 1/(\rho + 1)](1 - 1/\gamma)}}{n^{-1/2 + [1 - 1/(\rho + 1)](1 - 1/\gamma)}} + \sqrt{n} N_{[]}^{-1}(e^n).
	\end{align*}
	with constants only depending on $K_1,K_2$.
	
	In particular, if the integral is finite, $\|F\|_{\gamma,\infty}<\infty$, $\rho > \gamma/(\gamma - 2)$ and $K_1,K_2$ can be chosen independent of $n$,
	then
	\begin{align*}
		\limsup_{n \to \infty}\E\|\G_n \|_{\Fcal}	 \lesssim  \int_0^\delta  \sqrt{\ln N_{\left[\right]}(\epsilon) } d \epsilon.
	\end{align*}
\end{theorem}

\begin{proof}
	It holds
	\begin{align*}
		\|F\ind\{F>B\}\|_{1,n} \leq \frac{\sqrt{n}\|F\|_{\gamma,n}^{\gamma}}{B^{\gamma-1}}.
	\end{align*}
	By \Cref{thm:chaining}
	\begin{align*}
		\E[\|\G_n\|_{\Fcal}] \lesssim \int_0^\delta \sqrt{\ln_+ N_{\left[\right]}(\epsilon) } d \epsilon + \frac{ m B  \ln_+ N_{\left[\right]}(\delta)  }{\sqrt{n}} + \sqrt{n} B\beta_n(m) + \frac{\sqrt{n}\|F\|_{\gamma, n}^\gamma}{B^{\gamma - 1}} + \sqrt{n} N_{[]}^{-1}(e^n).
	\end{align*}
	Choose $m = (n/ \ln_+ N_{[]}(\delta))^{ 1/(\rho + 1) }$, which gives
	\begin{align*}
		\frac{m  \ln_+ N_{[]}(\delta)}{\sqrt{n}} + \sqrt{n}\beta_n(m) \lesssim  n^{-1/2 + 1/(\rho + 1)} [\ln_+ N_{[]}(\delta)]^{ 1 - 1/(\rho + 1) },
	\end{align*}
	and, thus,
	\begin{align*}
		\E[\|\G_n\|_{\Fcal}] \lesssim \int_0^\delta \sqrt{\ln_+ N_{\left[\right]}(\epsilon) } d \epsilon + \frac{  B [\ln_+ N_{\left[\right]}(\delta)]^{ 1 - 1/(\rho + 1) }  }{n^{1/2 - 1/(\rho + 1)}} + \frac{\sqrt{n}\|F\|_{\gamma, n}^\gamma}{B^{\gamma - 1}} + \sqrt{n} N_{[]}^{-1}(e^n).
	\end{align*}
		Next, choose 
	\begin{align*}
		B = \left(\frac{n^{[1 - 1/(\rho  + 1)]} \|F\|_{\gamma, n}^\gamma}{[\ln_+ N_{\left[\right]}(\delta)]^{ 1 - 1/(\rho + 1) } }\right)^{1/\gamma}
	\end{align*}
	This gives
	\begin{align*}
		\frac{  B [\ln_+ N_{\left[\right]}(\delta)]^{ 1 - 1/(\rho + 1) }  }{n^{1/2 - 1/(\rho + 1)}} +  \frac{\sqrt{n}\|F\|_{\gamma, n}^\gamma}{B^{\gamma - 1}} &= \frac{  \|F\|_{\gamma, n} [\ln_+ N_{\left[\right]}(\delta)]^{ [1 - 1/(\rho + 1)](1 - 1/\gamma)}}{n^{-1/2 + [1 - 1/(\rho + 1)](1 - 1/\gamma)}}.
	\end{align*}
	Lastly, if $\rho > \gamma/(\gamma - 2)$, then 
	\begin{align*}
		-1/2 + [1 - 1/(\rho + 1)](1 - 1/\gamma) > 0,
	\end{align*}
	so the second term in the first statement vanishes as $n \to \infty$ and 
	the last term vanishes since the bracketing integral is finite. 
\end{proof}

\begin{proof}[Proof of \Cref{thm:chaining}]
	Let us first deduce the last statement.
	If the bracketing integral exists, then it must hold $\sqrt{\ln_+ N_{\left[\right]}(\delta) }\lesssim \delta^{-1} / (1 + |\ln(\delta)|)$ for $\delta \to 0$, because the upper bound is not integrable.
	For $\delta^{-1} = \sqrt{n \ln n }$, we have $\ln_+ N_{\left[\right]}(\delta) \lesssim n  \ln n / (1 + \ln n)^2 = o(n)$. So for large $n$, it must hold
		$N_{\left[\right]}^{-1}(e^n) \lesssim 1 / \sqrt{n \ln n }\to 0$.

	We now turn to the proof of the first statement.

	\subsubsection*{Truncation}
	We first 
	truncate the function class $\Fcal$ in order to apply Bernstein's 
	inequality in combination with a chaining argument.
	It holds
	\begin{align*}
		\E\| \G_n  \|_{\mathcal{F}}\leq \E\bigl[\| \G_n( f\ind\{|F|\leq B\})  \|_{\mathcal{F}}\bigr] + \E\bigl[\| \G_n\left( f\ind\{|F|> B\}\right)  \|_{\mathcal{F}}\bigr],
	\end{align*}
	and
	\begin{align*}
		\E\bigl[\| \G_n( f\ind\{|F|> B\})  \|_{\mathcal{F}}\bigr]
		&\leq 2\frac{1}{\sqrt{n}}\sum_{i=1}^{n}\E\bigl[F( X_i) \ind\{|F( X_i) |>B\}\bigr] 
		=2\sqrt{n}\|F\ind\{F> B\}\|_{1,n}.
	\end{align*}
	In summary,
	\begin{align*}
		\E\| \G_n( f)  \|_{\mathcal{F}}
		\leq \E\bigl[\| \G_n( f\ind\{|F|\leq B\})  \|_{\mathcal{F}}\bigr] + 2\sqrt{n}\|F\ind\{F> B\}\|_{1,n}.
	\end{align*}
	Note that $|f\ind\{|F|\leq B\}|\leq F\ind\{|F|\leq B\}\leq B$.
	By replacing $\Fcal$ with $$\Fcal_{trun}=\left\{f\ind\left\{|F|\leq B\right\}\colon f\in \Fcal\right\},$$
	we may without loss of generality assume 
	that $\Fcal$ has an envelope with $\|F\|_\infty \leq B$.
	Observe that the conditions of the theorem remain true for $\Fcal_{trun}$
	and that the bracketing numbers with respect to 
	$\Fcal_{trun}$ are bounded above by the 
	bracketing numbers with respect to $\Fcal$. 

	\subsubsection*{Chaining setup} %
	Fix integers $r_0\leq r_1$ such that $2^{-r_0-1} < \delta \le 2^{-r_0}$.
	For $r \ge r_0$ we construct a nested sequence of partitions $\Fcal = \bigcup_{k = 1}^{N_r} \Fcal_{r, k}$ 
	of $\Fcal$ into $N_r$ disjoint subsets
	such that for each $r \geq r_0$
	\begin{align*}
		 & \left\| \sup_{f, f' \in \Fcal_{r, k} } | f - f'| \right\|_{\gamma,n} < 2^{-r}.
	\end{align*}
	Clearly, we can choose the partition such that $$N_{r_0} \le N_{\left[\right]}\left(  2^{-r_0}\right)  \le N_{\left[\right]}\left(  \delta\right) .$$
	We may assume $\ln_+N_{r_0}\leq n$:
	If $n<\ln_+ N_{r_0}\leq \ln_+ N_{[]}(\delta)$ then 
	$$\E\|\G_n\|_\Fcal\lesssim \sqrt{n} B\leq \frac{mB \ln_+ N_{[]}(\delta)}{\sqrt{n}}$$
	which still implies the claim. 
	As explained in the proof of Theorem 2.5.8 of \citet{van2023weak}, we may assume without loss of generality that
	\begin{align*}
		\sqrt{\ln_+ N_r} \le \sum_{k = r_0}^r \sqrt{ \ln_+ N_{\left[\right]}\left( 2^{-k}\right) }.
	\end{align*}
	Then by reindexing the double sum,
	\begin{align*}
		\sum_{r = r_0}^{r_1} 2^{-r} \sqrt{\ln_+ N_r}
		 & \le \sum_{r = r_0}^{r_1}  2^{-r} \sum_{k = r_0}^r \sqrt{\ln_+ N_{\left[\right]}\left( 2^{-k}\right)  }                   \\
		 & = \sum_{k = r_0}^{r_1}   \sqrt{\ln_+ N_{\left[\right]}\left( 2^{-k}\right) } \sum_{r = k}^{r_1}  2^{-r}                   \\
		 & = \sum_{k = r_0}^{r_1}  2^{-k} \sqrt{\ln_+ N_{\left[\right]}\left( 2^{-k}\right)  }    \sum_{r = k}^{r_1} 2^{-\left( r - k\right) } \\
		 & \lesssim \sum_{k = r_0}^{r_1} 2^{-k} \sqrt{\ln_+ N_{\left[\right]}\left( 2^{-k}\right)   }             \\
		 & \lesssim \int_0^\delta \sqrt{\ln_+ N_{\left[\right]}(\epsilon) } d \epsilon.
	\end{align*}

	\subsubsection*{Decomposition} %
	For a given $f$, suppose that $\Fcal_{r, k}$ is the element of the partition that contains $f$. 
	Note that such $\Fcal_{r, k}$ is unique since all $\Fcal_{r, 1},\ldots,\Fcal_{r, N_r}$ are disjoint.
	Define $\pi_r\left( f\right) $ as some fixed element of this set and define
	\begin{align*}
		\Delta_{r}\left( f\right)  = \sup_{f_1, f_2 \in \Fcal_{r, k}}|f_1 - f_2|.
	\end{align*}
	Set 
	\begin{align} \label{eq:tau-m-def}
		\tau_r = \frac{2^{-r}}{m_{r+1}} \sqrt{\frac{n}{\ln_+ N_{r+1}}}, \quad m_r = \min\left\{\sqrt{\frac{\ln_+ N_{r}}{n}}, 1\right\}^{-(\gamma - 2)/ (\gamma - 1)}, 
	\end{align}
	and 
	\begin{align*}
		r_1 =  -\log_2 N_{[]}^{-1}(e^n).
	\end{align*}
	From the definition we see $ \ln_+ N_{r}  \le n$ for all $r \le r_1$ 
	since $N_r$ is increasing and 
	$r_0\leq r_1$ since $ \ln_+ N_{r_0}  \le n$.
	We will frequently apply Bernstein's inequality with $m=m_r$. Here, note
	$$m_r\leq \sqrt{n/\ln_+ N_r} \leq \sqrt{n/\ln(2)}\leq n/2$$ for all $n\geq 5$. 

	The following (in-)equalities are 
	the reason for the choices of $\tau_r$ and $m_r$: 
	for $r$ such that $\ln_+ N_{r+1}\leq n$, i.e., $r<r_1$ it holds
	\begin{align*}
		\frac{m_r\tau_{r-1}}{\sqrt{n}}
		&=2^{-r+1}\bigl(\sqrt{\ln_+ N_r}\bigr)^{-1}\\
		\sqrt{n}2^{-r\gamma}\tau_r^{-(\gamma-1)}
		&=2^{-r}\sqrt{n}\left(\frac{1}{m_{r+1}}\sqrt{\frac{n}{\ln_+ N_{r+1}}}\right)^{-(\gamma-1)}
		\\
		&=2^{-r}\sqrt{n}^{2-\gamma}\left(\sqrt{\ln_+ N_{r+1}}\right)^{\gamma-1}m_{r+1}^{\gamma-1}
		\\
		&=2^{-r}\sqrt{\ln_+ N_{r+1}}\left(\sqrt{\frac{\ln_+ N_{r+1}}{n}}\right)^{\gamma-2} m_{r+1}^{\gamma-1}
		\\
		&=2^{-r}\sqrt{\ln_+ N_{r+1}}
		\\
		\sqrt{n}\tau_{r-1}\beta_n(m_r)
		&= \sqrt{n} 2^{-r+1}  \frac{1}{m_r} \sqrt{\frac{n}{\ln_+ N_{r}}} \beta_n(m_r)
		\\
		&\lesssim 2^{-r+1}  \frac{1}{m_r} \frac{n}{\sqrt{\ln_+ N_{r}}} m_r^{-\rho}
		\\
		&=2^{-r+1} \sqrt{\ln_+ N_{r}}  \frac{n}{\ln_+ N_{r}} m_r^{-\rho-1}\\
		&=2^{-r+1} \sqrt{\ln_+ N_{r}} m_r^{\frac{2(\gamma-1)}{\gamma-2}}  m_r^{-\rho-1}
		\\
		&\leq 2^{-r+1} \sqrt{\ln_+ N_{r}}
	\end{align*}
	where the last inequality holds since $1\leq m_r$ and $\gamma/(\gamma-2)\leq \rho$, hence, 
	$m_r^{\frac{2(\gamma-1)}{\gamma-2}-\rho-1}\leq 1$.

	Decompose 
	\begin{align*}
		f
		 & =   \pi_{r_0}\left( f\right)  + \left[f - \pi_{r_0}\left( f\right) \right] \ind\{\Delta_{r_0}\left( f\right) /\tau_{r_0} > 1\}                                                                       \\
		 & \quad + \sum_{r = r_0 + 1}^{r_1} \left[f - \pi_{r}\left( f\right) \right] \ind\left\{\max_{r_0 \le k < r}\Delta_{k}\left( f\right) /\tau_{k} \le 1,  \Delta_{r}\left( f\right) /\tau_{r} > 1\right\} \\
		 & \quad + \sum_{r = r_0 + 1}^{r_1} \left[\pi_{r}\left( f\right)  - \pi_{r - 1}\left( f\right) \right] \ind\left\{\max_{r_0 \le k < r}\Delta_{k}\left( f\right) /\tau_{k} \le 1\right\}                 \\
		 & \quad + \left[f - \pi_{r_1}\left( f\right) \right] \ind\left\{\max_{r_0 \le k \le r_1}\Delta_{k}\left( f\right) /\tau_{k} \le 1\right\}                                               \\
		 & = T_1\left( f\right)  + T_2\left( f\right)  + T_3\left( f\right)  + T_4\left( f\right) .
	\end{align*}
	To see this, note that if $\Delta_{r_0}\left( f\right) /\tau_{r_0} > 1$ all terms but $T_1(f)$ vanish
	and $T_1(f)=f$. 
	Otherwise, define $\hat{r}$ as the maximal number $r_0\leq r\leq r_1$ such that 
	$\max_{r_0 \le k \leq r}\Delta_{k}\left( f\right) /\tau_{k} \le 1$. 
	Then, 
	\begin{align*}
		T_1(f)&=\pi_{r_0}(f)
	\end{align*}
	and if $\hat{r}<r_1$, then, 
	\begin{align*}
		T_2(f)=f-\pi_{\hat{r}+1}(f) \qquad
		T_3(f)=\pi_{\hat{r}+1}(f)-\pi_{r_0}(f)\qquad
		T_4(f)=0.
	\end{align*}
	If $\hat{r}=r_1$, then, 
	\begin{align*}
		T_2(f)=0 \qquad
		T_3(f)=\pi_{r_1}(f)-\pi_{r_0}(f) \qquad
		T_4(f)=f-\pi_{r_1}(f).
	\end{align*}

	We prove the theorem by separately bounding the four terms $\E\|\G_n T_j \|_\Fcal$.
	Note that $\G_n$ is additive by construction, i.e., $\G_n\left( f+g\right) =\G_n\left( f\right) +\G_n\left( g\right) $. 
	\subsubsection*{Bounding $T_1$}  %

    Note that for every $|g|\le h$ it follows 
    $$|\G_n\left( g\right) |\leq |\G_n\left( h\right) |+2\sqrt{n}\|h\|_{1,n}.$$
	In combination with the triangle inequality we obtain
	\begin{align*}
		\|\G_n T_1 \|_{\Fcal}
		 & \le   \|\G_n \pi_{r_0} \|_\Fcal + \|\G_n \Delta_{r_0}\|_\Fcal+ 2\sqrt{n} \sup_{f\in \Fcal}\|\Delta_{r_0}(f) \ind\{\Delta_{r_0}\left( f\right) /\tau_{r_0} > 1 \} \|_{1,n} .
	\end{align*}
	The sets $\{\Delta_{r_0}\left( f\right)  \colon f \in \Fcal\}$ and $\{\pi_{r_0}\left( f\right)  \colon f \in \Fcal\}$ contain at most $N_{r_0}$ different functions each.
	The construction implies
	\begin{align*}
		 & \|\pi_{r_0}\left( f\right) \|_{\gamma,n} \le \delta, \quad \|\pi_{r_0}\left( f\right) \|_\infty \lesssim B, \quad \|\Delta_{r_0}\left( f\right) \|_{\gamma,n} \le 2\delta, \quad \|\Delta_{r_0}\left( f\right) \|_\infty \lesssim B.
	\end{align*}
	Now the Bernstein bound from \cref{lem:bernstein} gives
	\begin{align*}
		\E \|\G_n\pi_{r_0} \|_{\Fcal} + \E \|\G_n\Delta_{r_0} \|_{\Fcal}
		& \lesssim \delta  \sqrt{\ln_+ N_{r_0}}+  \frac{mB}{\sqrt{n}} \ln_+ N_{r_0} + \sqrt{n} B \beta_n(m).
		\\
		 & \le \delta  \sqrt{\ln_+ N_{\left[\right]}(\delta) }+  \frac{mB}{\sqrt{n}} \ln_+ N_{\left[\right]}(\delta) + \sqrt{n} B\beta_n(m) .
	\end{align*}
	Since the bracketing numbers are decreasing,
	\begin{align*}
		\sqrt{\ln_+ N_{\left[\right]}(\delta) } \le \delta^{-1} \int_{0}^\delta \sqrt{\ln_+ N_{\left[\right]}(\epsilon) } d\epsilon 
	\end{align*}
	so 
	\begin{align*}
		\E \|\G_n\pi_{r_0} \|_{\Fcal} +  \E \|\G_n\Delta_{r_0}\|_{\Fcal}
		 & \lesssim  \int_{0}^\delta \sqrt{\ln_+ N_{\left[\right]}(\epsilon) } d\epsilon  + \frac{ m B  \ln_+ N_{\left[\right]}(\delta)  }{\sqrt{n}} +  \sqrt{n} B \beta_n(m).
	\end{align*}
	Recall $\ln_+ N_{r+1} \le n$ for any $r < r_1$. For any such $r$, \ref{C3:properties-norm} of \Cref{lem:properties-norm} gives
	\begin{align*}
		\sqrt{n} \sup_{f\in \Fcal}\|\Delta_{r}(f) \ind\{\Delta_{r}\left( f\right) /\tau_{r} > 1 \} \|_{1,n} 
		& \le \sqrt{n} \tau_{r}^{-\left( \gamma - 1\right) } \sup_{f\in \Fcal}\|\Delta_{r}\left( f\right)  \|_{\gamma,n}^\gamma        \\
		& \le \sqrt{n} \tau_{r}^{-\left( \gamma - 1\right) }  2^{-r\gamma}      
	\end{align*}
	so that the final upper bound becomes
	\begin{align} \label{eq:Delta-bound}
		\sqrt{n} \sup_{f\in \Fcal}\|\Delta_{r}(f) \ind\{\Delta_{r}\left( f\right) /\tau_{r} > 1 \} \|_{1,n}  \lesssim 2^{-r} \sqrt{\ln_+ N_{r+1}},
	\end{align}
	for any $r < r_1$. In particular, using $\delta\leq 2^{-r_0}$, we get 
	\begin{align*}
		\sqrt{n} \sup_{f\in \Fcal}\|\Delta_{r_0}(f) \ind\{\Delta_{r_0}\left( f\right) /\tau_{r_0} > 1 \} \|_{1,n}
    & \lesssim  \delta \sqrt{\ln_+ N_{[]}(\delta)} \le \int_0^\delta \sqrt{\ln_+ N_{\left[\right]}(\epsilon) } d\epsilon.
	\end{align*}
	Combined,
	\begin{align*}
		\E \|\G_n T_1 \|_{\Fcal}\lesssim \int_{0}^\delta \sqrt{\ln_+ N_{\left[\right]}(\epsilon) } d\epsilon  + \frac{ m B  \ln_+ N_{\left[\right]}(\delta)  }{\sqrt{n}} +  \sqrt{n} B \beta_n(m).
	\end{align*}

	\subsubsection*{Bounding $T_2$} %
	Next,
	\begin{align*}
		\E \|\G_n T_2\|_{\Fcal}
		 & \le  \sum_{r = r_0 + 1}^{r_1} \E \left\| \G_n  \Delta_{r} \ind\left\{\max_{r_0 \le k < r}\Delta_{k}/\tau_{k} \le 1, \Delta_{r}/\tau_{r} > 1\right\}\right\|_{\Fcal}                  \\
		 & \quad + 2\sqrt{n} \sum_{r = r_0 + 1}^{r_1} \sup_{f\in \Fcal} \left\|\Delta_{r}(f)  \ind\left\{\max_{r_0 \le k < r}\Delta_{k}(f)/\tau_{k} \le 1, \Delta_{r}/\tau_{r} > 1\right\} \right\|_{1,n}        \\
		 & = T_{2, 1} + T_{2, 2}.
	\end{align*}
	We start by bounding the first term. 
	It holds
	\begin{align*}
		\left\|\Delta_{r}( f)   \ind\left\{\max_{r_0 \le k < r}\Delta_{k}\left( f\right) /\tau_{k} \le 1, \Delta_{r}\left( f\right) /\tau_{r} > 1\right\} \right\|_{\gamma,n}
		\leq 2^{-r}
	\end{align*}
	by construction of $\Delta_{r}( f)$.
    Since the partitions are nested $\Delta_{r} \le \Delta_{r - 1}$. Thus,
	\begin{align*}
		\left\| \Delta_{r} \ind\left\{\max_{r_0 \le k < r}\Delta_{k}/\tau_{k} \le 1, \Delta_{r}/\tau_{r} > 1\right\} \right\|_\Fcal\le   \tau_{r - 1}.
	\end{align*}
	Since there are at most $N_r$ functions in $\{\Delta_r\left( f\right) \colon f \in \Fcal\}$,
	the Bernstein bound from \cref{lem:bernstein} yields
	\begin{align*}
		T_{2, 1}
		& \lesssim  \sum_{r = r_0 + 1}^{r_1} \left[ 2^{-r}\sqrt{\ln_+ N_{r}}
		+ \frac{m_r \tau_{r-1} }{\sqrt{n}} \ln_+ N_{r} + \sqrt{n} \tau_{r-1}\beta_n(m_r)\right]                                             \\
		& \lesssim  \sum_{r = r_0 + 1}^{r_1}  2^{-r}\sqrt{\ln_+ N_{r}}      \\
		&\lesssim   \int_0^\delta \sqrt{\ln_+ N_{\left[\right]}(\epsilon) } d \epsilon.
	\end{align*}
	Further, \eqref{eq:Delta-bound} gives 
	\begin{align*}
		\sqrt{n} \sup_{f\in \Fcal} \left\|\Delta_{r}(f)  \ind\left\{\max_{r_0 \le k < r}\Delta_{k}(f)/\tau_{k} \le 1, \Delta_{r}/\tau_{r} > 1\right\} \right\|_{1,n}
		 & \lesssim 2^{-r} \sqrt{\ln_+ N_{r}}.
	\end{align*}
	for $r<r_1$ and 
	\begin{align*}
		\sqrt{n} \sup_{f\in \Fcal} \left\|\Delta_{r_1}(f)  \ind\left\{\max_{r_0 \le k < r_1}\Delta_{k}(f)/\tau_{k} \le 1, \Delta_{r_1}/\tau_{r_1} > 1\right\} \right\|_{1,n}
		&\leq \sqrt{n} \sup_{f\in \Fcal} \left\|\Delta_{r_1}(f) \right\|_{\gamma,n}^{\gamma}
		\\ & \leq \sqrt{n} 2^{-r_1} 
		\\
		& = \sqrt{n}N^{-1}_{[]}(e^n)
	\end{align*}
	by the definition of $r_1$.
	Thus, 
	\begin{align*}
		T_{2, 2} & \lesssim \sum_{r = r_0 + 1}^{r_1-1} 2^{-r} \sqrt{\ln_+ N_{r}} + \sqrt{n}N^{-1}_{[]}(e^n) \lesssim \int_0^\delta \sqrt{\ln_+ N_{\left[\right]}(\epsilon) } d \epsilon + \sqrt{n}N^{-1}_{[]}(e^n),
	\end{align*}
	and, in summary,
	\begin{align*}
		\E \|\G_n T_2\|_\Fcal\lesssim  \int_0^\delta \sqrt{\ln_+ N_{\left[\right]}(\epsilon) } d \epsilon  + \sqrt{n}N^{-1}_{[]}(e^n).
	\end{align*}

	\subsubsection*{Bounding $T_3$}  %

	Next,
	\begin{align*}
		\E \|\G_n T_3\|_{\Fcal}
		 & \leq  \sum_{r = r_0 + 1}^{r_1} \E \left\|\G_n  \left[\pi_{r} - \pi_{r -
				1}\right] \ind\left\{\max_{r_0 \le k < r}\Delta_{k}/\tau_{k} \le 1\right\}\right\|_{\Fcal}.
	\end{align*}
	There are at most $N_r$ functions $\pi_{r}\left( f\right) $ and at most $N_{r - 1}$ functions $\pi_{r -1}\left( f\right) $ as $f$ ranges over $\Fcal$.
	Since the partitions are nested, $|\pi_r\left( f\right)  - \pi_{r - 1}\left( f\right) | \le \Delta_{r - 1}\left( f\right) $ and
	\begin{align*}
		|\pi_{r}\left( f\right)  - \pi_{r - 1}\left( f\right) | \ind\left\{\max_{r_0 \le k < r}\Delta_{k}\left( f\right) /\tau_{k} \le 1\right\}
		 & \le  |\Delta_{r - 1}\left( f\right) | \ind\left\{\Delta_{r-1}\left( f\right) /\tau_{r - 1} \le 1\right\}
		\le \tau_{r -1}.
	\end{align*}
	Further,
	\begin{align*}
		\|\pi_{r}\left( f\right)  - \pi_{r - 1}\left( f\right) \|_{\gamma,n} \le  \| \Delta_{r - 1}\left( f\right)  \|_{\gamma,n} \le 2^{-r + 1}.
	\end{align*}
	Just as for $T_{2, 1}$, the Bernstein bound (\cref{lem:bernstein})  gives
	\begin{align*} 
		\E \|\G_nT_3\|_\Fcal\lesssim  \int_0^\delta \sqrt{\ln_+ N_{\left[\right]}(\epsilon)  d\epsilon}.
	\end{align*}

	\subsubsection*{Bounding $T_4$}  %

	Finally,
	\begin{align*}
		\E\|\G_n T_4\|_{\Fcal}
		 & =  \E\left\|\G_n \left[f - \pi_{r_1}\left( f\right) \right] \ind\left\{\max_{r_0 \le k \le r_1}\Delta_{k}\left( f\right) /\tau_{k} \le 1\right\} \right\|_{f\in\Fcal}                                                                                \\
		 & \lesssim \E\left\|\G_n  \Delta_{r_1} \ind\left\{\Delta_{r_1}\le \tau_{r_1} \right\} \right\|_\Fcal+ \sqrt{n} \sup_{f\in \Fcal} \left\| \Delta_{r_1}(f) \ind\left\{\Delta_{r_1}(f) \le \tau_{r_1} \right\} \right\|_{1,n} \\
		 & \lesssim \E\left\|\G_n  \Delta_{r_0} \right\|_\Fcal+ \sqrt{n} \tau_{r_1}.      
		\end{align*}
	and the first term is bounded by $T_1$.
	Finally, observe that, by the definition of $r_1$,
	\begin{align*}
		\sqrt{n} \tau_{r_1} \le \sqrt{n} 2^{-r_1} = \sqrt{n} N_{[]}^{-1}(e^n).
	\end{align*}
	Combining the bounds yields the claim. 
\end{proof}

\begin{lemma}\label{rem:sum-implies-cov-bounded-by-moment}
	For $\gamma>2$ and 
	\begin{align*}
		\sum_{i = 1}^n \beta_{n}\left( i\right) ^{\frac{\gamma-2}{\gamma}}\leq K
	\end{align*}
	it holds
	\begin{align*}
		\frac{1}{n} \sum_{i,j= 1}^{n} |\cov[h(X_i),h(X_j)]|
		\leq 8K \|h\|_{\gamma,n}^2
	\end{align*}
	for all $h:\Xcal\to \R$ measurable.

	Furthermore, if $\sup_{n\in \N} \max_{m\leq n} m^{\rho}\beta_n(m)<\infty$ for some 
	$\rho>\gamma/(\gamma-2)$, then, $$\sup_n\sum_{i = 1}^n \beta_{n}\left( i\right) ^{\frac{\gamma-2}{\gamma}}<\infty.$$
\end{lemma}

\begin{proof}
	Let us prove the latter claim first. Note that
	if $\sup_{n\in \N} \max_{m\leq n} m^{\rho}\beta_n(m)<\infty$ then 
	$\sum_{i = 1}^n \beta_{n}\left( i\right) ^{\frac{\gamma-2}{\gamma}}\lesssim \sum_{i = 1}^nm^{-\rho\frac{\gamma-2}{\gamma}} $
	and the latter display converges for $\rho>\gamma/(\gamma-2)$.

	For the first claim, it holds
	\begin{align*}
		&\frac{1}{n} \sum_{i,j= 1}^{n} |\cov[h(X_i),h(X_j)]| \\
        &\leq
        \frac{1}{n} \sum_{i,j= 1}^{n} \beta_{n}(|i-j|)^{\frac{\gamma-2}{\gamma}}\|h(X_i)\|_{\gamma}
        \|h(X_j)\|_{\gamma} \\
        &\leq
        \frac{1}{n}\sum_{i,j= 1}^{n} \beta_{n}(|i-j|)^{\frac{\gamma-2}{\gamma}}(\|h(X_i)\|_{\gamma}^2 + \|h(X_j)\|_{\gamma}^2) \\
        &\leq
        \frac{1}{n} \sum_{i= 1}^{n}\|h(X_i)\|_{\gamma}^2 \sum_{j=1}^{ n}\beta_{n}(|i-j|)^{\frac{\gamma-2}{\gamma}} + \frac{1}{n}\sum_{j=1}^{n}\|h(X_j)\|_{\gamma}^2  \sum_{i= 1}^{n} \beta_{n}(|i-j|)^{\frac{\gamma-2}{\gamma}}\\
        &\leq
        \frac{8K}{n} \sum_{i= 1}^{n} \|h(X_i)\|_{\gamma}^2
		\\
		&\leq 8K \|h\|_{\gamma,n}^2.   
	\end{align*}
	by Theorem 3 of \cite{doukhan2012mixing} and 
	where the last inequality follows from
	\begin{align*}
		\left(\frac{1}{n} \sum_{i= 1}^{n} \|h(X_i)\|_{\gamma}^2\right)^{1/2}
		&=\left(\frac{1}{n} \sum_{i= 1}^{n} \E[|h(X_i)|^{\gamma}]^{2/\gamma}\right)^{1/2}
		\\
		&\leq \left(\frac{1}{n} \sum_{i= 1}^{n} \E[|h(X_i)|^{\gamma}]\right)^{1/\gamma}
		&\text{$2/\gamma\leq1$ and Jensen's inequality.}
	\end{align*}
\end{proof}

\begin{remark}\label{rem:br-bound-cn}
    Given a semi-metric $d$ on $\Fcal$ induced by a semi-norm $\|\cdot\|$ satisfying 
    $$|f|\leq |g|\Rightarrow \|f\|\leq \|g\|,$$ any 
    $2\eps$-bracket $[f,g]$ is contained in the $\epsilon$-ball around 
    $(f-g)/2$. Then, 
    $$N(\eps,\Fcal,d)\leq N_{[]}(2\eps,\Fcal,\|\cdot\|).$$
    Both, $\|\cdot\|_{\gamma,n}$ and 
    $\|\cdot\|_{\gamma,\infty}$, satisfy this property. 
\end{remark}

\subsection{Proof of Theorem \ref{thm:asy-tightness-sup-br}}
	We first prove the existence of an asymptotically tight sequence of GPs. 
		We will conclude by \Cref{prop:ex-tight-GPs,prop:asy-tight-GPs}:
		There exists $K\in \R$ such that $$\sum_{i=1}^{k_n}\beta_n(i)^{\frac{\gamma-2}{\gamma}}\leq K$$ for all $n$ by \Cref{rem:sum-implies-cov-bounded-by-moment}. It holds
		\begin{align*}
				\rho_{n}(s,t)^2&=\var[\mathbb{G}_n(s)-\mathbb{G}_n(t)]\\
				&=\frac{1}{k_n}\var\left[\sum_{i=1}^{k_n}(f_{n,s}-f_{n,t})(X_{n,i})\right]\\
				&\leq \frac{1}{k_n}\sum_{i,j=1}^{k_n}\left|\cov\left[(f_{n,s}-f_{n,t})(X_{n,i}),(f_{n,s}-f_{n,t})(X_{n,j})\right]\right|\\
				&\leq 4K\|f_{n,s}-f_{n,t}\|^2_{\gamma,n}\\
				&=4Kd_n(s,t)^2
		\end{align*}
		by \Cref{rem:sum-implies-cov-bounded-by-moment}.
		Thus, 
		\begin{align*}
			N(2\sqrt{K}\epsilon,T,\rho_n)
			&\leq N(\epsilon,T,d_n)
			\leq N_{[]}(2\epsilon,\Fcal_n,\|\cdot\|_{\gamma,n})
		\end{align*}
		by \Cref{rem:br-bound-cn}.
		Next, observe that 
		$$\ln N_{[]}(\epsilon,\mathcal{F}_n,\|\cdot\|_{\gamma,n})=0$$
		for all $\eps\geq 2\|F\|_{\gamma,n}$. 
		Thus,
		$$\int_{0}^{\infty}\sqrt{\ln N_{[]}(\epsilon,\mathcal{F}_n,\|\cdot\|_{\gamma,n})}d\epsilon < \infty$$
		for all $n$ by the entropy condition \ref{A3n:asy-tightness-sup-br}.
		In summary, 
		\begin{itemize}
			\item $(T,d)$ is totally bounded.
			\item $\lim_{n\to \infty}\int_{0}^{\delta_n}\sqrt{\ln N(\epsilon,T,\rho_{n}) }d\epsilon \lesssim \lim_{n\to \infty}\int_{0}^{\delta_n}\sqrt{\ln N_{[]}(\epsilon,\Fcal_n,\|\cdot\|_{\gamma,n})}d\epsilon= 0$ for all $\delta_n\downarrow 0.$
			\item $\lim_{n\to \infty}\sup_{d(s,t)<\delta_n}\rho_{n}(s,t)\leq \lim_{n\to \infty}\sup_{d(s,t)<\delta_n}2\sqrt{K}d_{n}(s,t) =0$ for every $\delta_n\downarrow 0.$
			\item $\int_{0}^{\infty}\sqrt{\ln N(\epsilon,T,\rho_{n})}d\epsilon\lesssim \int_{0}^{\infty}\sqrt{\ln N_{[]}(\epsilon,\mathcal{F}_n,\|\cdot\|_{\gamma,n})}d\epsilon<\infty$ for all $n$.
		\end{itemize}
		by the assumptions.
		Lastly, 
		$$\sup_n\var[\G_n(t)]\lesssim\sup_n\|F\|_{\gamma,n}^2=\|F\|_{\gamma,\infty}^2<\infty$$
		for all $t\in T$ by the same argument as above.
		Combined, we derive the claim by \Cref{prop:ex-tight-GPs,prop:asy-tight-GPs}.

	Next, we prove 
	asymptotic tightness of $\mathbb{G}_n$.
	We derive that $\sup_n\E\left[\left|\G_n(s) \right|^2\right]<\infty$ for all $s\in T$ again 
	by the moment condition \ref{A1:asy-tightness-sup-br} and the summability 
	condition \Cref{rem:sum-implies-cov-bounded-by-moment}.
	Thus, each $\G_n(s) $ is asymptotically tight. 
	By Markov's inequality and Theorem 1.5.7 of \cite{van2023weak}, it suffices to prove uniform $d$-equicontinuity, i.e.
	that 
	$$ \limsup_{n\to \infty}\E^*\sup_{d\left( f,g\right) <\delta_n}|\G_n(s) -\G_n(t)  | = 0$$
	for all $\delta_n \downarrow 0$.
	By $$\lim_{n\to \infty}\sup_{d\left( s,t\right) <\delta_n}d_{n}\left( s,t\right) =0$$ for all $\delta_n\downarrow 0$, for every sequence $\delta\to 0$ there exists a sequence $\epsilon\left( \delta\right) \to 0$ such that 
    $d\left( s,t\right) <\delta$ implies $d_n\left( s,t\right) <\epsilon\left( \delta\right)$.
	Thus, 
	$$\limsup_{n\to \infty}\E^*\sup_{d\left( s,t\right) <\delta_n}|\G_n(s) -\G_n(t) | \leq \limsup_{n\to \infty}\E^*\sup_{d_n\left( s,t\right) <\epsilon\left( \delta_n\right) }|\G_n(s) -\G_n(t) |.$$
	Accordingly, it suffices to prove that 
	$$\limsup_{n\to \infty}\E^*\sup_{d_n\left( s,t\right) <\delta_n}|\G_n(s) -\G_n(t) | = 0.$$
	
	For fixed $n$ we
	again identify $\G_n$ with the empirical process $\G_n$ 
	indexed by $\Fcal_n$ and similarly for $d_n$.
	Note that
	$\G_n\left( f\right) -\G_n\left( g\right) =\G_n\left( f-g\right) $ 
	and the bracketing number with respect to 
	the function class $$\Fcal_{n,\delta}=\{f-g \colon f,g\in \Fcal_n, \|f-g\|_{\gamma,n}<\delta\}$$ 
	satisfies 
	$$N_{\left[\right]}\left( \epsilon,\Fcal_{n,\delta},\|\cdot\|_{\gamma,n}\right) \leq N_{\left[\right]}\left( \epsilon/2,\Fcal_n,\|\cdot\|_{\gamma,n}\right) ^2.$$
	Indeed, given $\epsilon/2$ brackets $\left[l_f,u_f\right]$ and $\left[l_g,u_g\right]$ for 
	$f$ and $g$, $\left[l_f-u_g,u_f-l_g\right]$ is an $\eps$-bracket for $f-g$.
	By \Cref{thm:bracketing-coupled}, 
	\begin{align*}
		\limsup_n\E^*\sup_{d_n\left( s,t\right) <\delta_n}|\G_n(s) -\G_n(t)  |
		&\lesssim \lim_{n\to \infty}\int_0^{2\delta_n} \sqrt{\ln N_{\left[\right]}\left( \epsilon,\Fcal_n,\|\cdot\|_{\gamma,n}\right) } d \epsilon 
		\\
		&=0
	\end{align*}
	by \ref{A3n:asy-tightness-sup-br} which proves the claim. \qed

\section{Proofs for relative CLTs under mixing conditions}

\subsection{Proof of Theorem \ref{thm:multi-rel-clt}}\label{sec:multi-rel-clt}
We first restrict to univariate random variables which, in combination 
with the relative Cramer-Wold device (\Cref{cor:rel-cramer-wold}), yields \Cref{thm:multi-rel-clt}. 
The idea is to split the scaled sample average 
\begin{align*}
	\frac{1}{\sqrt{n}}\sum_{i=1}^{n}\left(X_{n,i}-\E\left[X_{n,i}\right]\right)=\frac{1}{\sqrt{n}}\sum_{i=1}^{r_n}\left( Z_{n,i}-\E\left[Z_{n,i}\right]+\tilde{Z}_{n,i}-\E\left[\tilde{Z}_{n,i}\right]\right)  
\end{align*}
into alternating
long and short block sums. 
By considering a small enough length of the short blocks, 
the short block sums are asymptotically negligible. 
It then suffices to prove a relative CLT 
for the sequence of long block sums (\Cref{lem:slutsky}) . 
By maximal coupling and \Cref{lem:slutsky}, the sequence of long block sums can be considered 
independent and Lindeberg's CLT (\Cref{thm:rel-lind-feller})  applies. 

\begin{lemma}\label{lem:slutsky}
	Let $Y_n$ and $Y_n^*$ be sequences of random variables in $\R$ such that 
	\begin{enumerate}
		\item $|Y_n-Y_n^*|\overset{P}{\to}0,$
		\item $\sup_n\var\left[Y_n\right]<\infty$ and 
		\item $\left|\var\left[Y_n\right]-\var\left[Y_n^*\right]\right|\to 0$.
	\end{enumerate}
	Then, $Y_n$ satisfies a relative CLT if and only if $Y_n^*$ does. 
\end{lemma}

\begin{proof}
	It suffices to prove the if direction since the statement is symmetric. 
	Let $k_n$ be a subsequence of $n$ and $l_n$ be a further subsequence such that 
	$Y_{l_n}^*\dc N$ converges weakly to some Gaussian with $$\var\left[Y^*_{l_n}\right]\to \var\left[N\right].$$
	Such $l_n$ exists by \ref{A3:rel-clt-chara} of \Cref{prop:rel-clt-chara}.
	Since $|Y_n-Y_n^*|\overset{P}{\to}0$, we obtain $Y_{l_n}\dc N$.
	Note $$\var\left[N\right]=\lim_{n\to \infty}\var\left[Y^*_{l_n}\right]=\lim_{n\to \infty}\var\left[Y_{l_n}\right]$$ by assumption.
	Thus, $Y_n$ satisfies a relative CLT by \ref{A3:rel-clt-chara} of \Cref{prop:rel-clt-chara}.
\end{proof}

\begin{theorem}[Univariate relative CLT]\label{thm:rel-univ-clt}
    Let $X_{n,1},\ldots,X_{n,k_n}$ be a triangular array of univariate random variables.
	For some $\gamma>2$ and $\alpha<(\gamma-2)/2(\gamma-1)$ 
    assume
    \begin{enumerate}
        \item \label{A1:rel-univ-clt}$k_n^{-1}\sum_{i,j=1}^{k_n}\left|\cov\left[X_{n,i},X_{n,j}\right]\right|\leq K$ for all $n$.
        \item \label{A2:rel-univ-clt}$\sup_{n,i}\E\left[|X_{n,i}|^\gamma\right]<\infty$.
        \item \label{A3:rel-univ-clt}$k_n\beta_{n}\left( k_n^{\alpha}\right) ^\frac{\gamma-2}{\gamma}\to 0$.
    \end{enumerate}
    Then, the scaled sample average $\sqrt{k_n}\left( \bar{X}_n-\E\left[\bar{X}_n\right]\right) $ satisfies a relative CLT. 
\end{theorem}

\begin{proof}
	There exists some $\delta$ with $0<\delta<1/2-\alpha$ and $1+\left( 1-2\alpha\right) ^{-1} <1+\left( 2\delta\right) ^{-1}<\gamma$.
	Define $q_n= k_n^{\alpha}$, $p_n=k_n^{1/2-\delta}-q_n$ and $r_n=k_n^{1/2+\delta}$.
	Group the observations in alternating blocks of size $p_n$ resp. $q_n$, i.e.,
	\begin{align*}
		U_{n,i}&=\left( X_{n,1+\left( i-1\right) \left( p_n+q_n\right) },\ldots,X_{n,p_n+\left( i-1\right) \left( p_n+q_n\right) }\right)  \in \R^{p_n} &\text{(long blocks) }\\
		\tilde{U}_{n,i}&=\left( X_{n,1+ip_n+\left( i-1\right) q_n},\ldots,X_{n,q_n+ip_n+\left( i-1\right) q_n}\right)  \in \R^{q_n} &\text{(short blocks) }
	\end{align*}
	Define
	\begin{align*}
			Z_{n,i}&=\sum_{j=1}^{p_n}U_{n,i}^{\left( j\right) }&\text{(long block sums)}\\
			\tilde{Z}_{n,i}&=\sum_{j=1}^{q_n}\tilde{U}_{n,i}^{\left( j\right) }& \text{(short block sums)}
	\end{align*}
	where the upper index $j$ denotes the $j$-th component. 
	Then, 
	$$\sum_{i=1}^{k_n}\left(X_{n,i}-\E\left[X_{n,i}\right]\right)=\sum_{i=1}^{r_n}\left( Z_{n,i}-\E\left[Z_{n,i}\right]+\tilde{Z}_{n,i}-\E\left[\tilde{Z}_{n,i}\right]\right) $$
	and it holds
	\begin{align*}
		k_n^{-1}\var\left[\sum_{i=1}^{k_n}X_{n,i}\right]&\leq k_n^{-1}\sum_{i,j=1}^{k_n}|\cov\left[X_{n,i},X_{n,j}\right]|\leq K
		\\
		k_n^{-1}\var\left[\sum_{i=1}^{r_n}Z_{n,i}\right]&\leq K
		\\
		k_n^{-1}\cov\left[\sum_{i=1}^{r_n}\tilde{Z}_{n,i},\sum_{i=1}^{r_n}Z_{n,i}\right]&=\mathcal{O}\left( r_nq_n/k_n\right) =\mathcal{O}\left( k_n^{1/2+\delta+\alpha-1}\right) =o\left( 1\right) 
		\\
		k_n^{-1}\var\left[\sum_{i=1}^{r_n}\tilde{Z}_{n,i}\right]&=\mathcal{O}\left( r_nq_n/k_n\right) =o\left( 1\right) 
	\end{align*}
	by assumption.
	Thus, $$k_n^{-1/2}\sum_{i=1}^{r_n}\left(\tilde{Z}_{n,i}-\E\left[\tilde{Z}_{n,i}\right]\right)\overset{P}{\to} 0$$
	by Markov's inequality.
	Hence, $$\left|k_n^{-1/2}\sum_{i=1}^{k_n}\bigl(X_{n,i}-\E\left[X_{n,i}\right]\bigr)-k_n^{-1/2}\sum_{i=1}^{r_n}\bigl(Z_{n,i}-\E\left[Z_{n,i}\right]\bigr)\right|\overset{P}{\to}0.$$
	Furthermore, 
	\begin{align*}
		\left|\var\left[k_n^{-1/2}\sum_{i=1}^{k_n}X_{n,i}\right]-\var\left[k_n^{-1/2}\sum_{i=1}^{r_n}Z_{n,i}\right]\right|
		&=k_n^{-1}\left|\var\left[\sum_{i=1}^{r_n}\tilde{Z}_{n,i}\right]+2\cov\left[\sum_{i=1}^{r_n}\tilde{Z}_{n,i},\sum_{i=1}^{r_n}Z_{n,i}\right]\right|
		\\&\to 0.
	\end{align*}
	By the previous lemma, 
	$k_n^{-1/2}\sum_{i=1}^{k_n}\left(X_{n,i}-\E\left[X_{n,i}\right]\right)$ satisfies 
	a relative CLT if and only if 
	$k_n^{-1/2}\sum_{i=1}^{r_n}Z_{n,i}-\E\left[Z_{n,i}\right]$ does.

	By maximal coupling (Theorem 5.1 of \cite{rio2017asymptotic}), for all $i=1,\ldots, r_n$ there exist random 
	vectors $U_{n,i}^*\in \R^{p_n}$ such that
	\begin{itemize}
		\item $U_{n,i}\overset{d}{=}U_{n,i}^*$.
		\item the sequence $U_{n,i}^*$ is independent.
		\item $\Pr\left( \exists U_{n,i}\neq U_{n,i}^*\right) \leq r_n\beta_{n}\left( q_n\right) $.
	\end{itemize}
	Define the coupled long block sums
	$$Z_{n,i}^*=\sum_{j=1}^{p_n}U_{n,i}^{*\left( j\right) }.$$
	For all $\eps>0$ we obtain
	\begin{align*}
		\Pr\left( k_n^{-1/2}\left|\sum_{i=1}^{r_n}Z_{n,i}^*-\sum_{i=1}^{r_n}Z_{n,i}\right|>\epsilon\right)  
		&\leq \Pr\left( \exists U_{n,i}\neq U_{n,i}^*\right) 
		\\
		&\leq r_n\beta_{n}\left( q_n\right) 
		\\
		& =k_n^{1/2+\delta}\beta_n\left( k_n^{\alpha}\right) 
		\\
		& \leq k_n\beta_n\left( k_n^{\alpha}\right)  \to 0.
	\end{align*}
	Next, 
	\begin{align*}
		\var\left[\sum_{i=1}^{r_n}Z_{n,i}\right]=\var\left[\sum_{i=1}^{r_n}Z_{n,i}^*\right]+\sum_{i\neq j}^{r_n}\cov\left[Z_{n,i},Z_{n,j}\right]
	\end{align*}
	by independence of $Z_{n,i}^*$ and $P_{Z_{n,i}}=P_{Z_{n,i}^*}$.
	Since $\sup_{n,k}\|X_{n,k}\|_\gamma<\infty$ and 
	$$|\cov\left[X_{n,i},X_{n,j}\right]|\lesssim \beta_n\left( |i-j|\right) ^{\frac{\gamma-2}{\gamma}}\sup_{n,k}\|X_{n,k}\|_\gamma^2$$
	by Theorem 3. of \cite{doukhan2012mixing}, for $i\neq j$
	we obtain 
	\begin{align*}
		\left|\cov\left[Z_{n,i},Z_{n,j}\right]\right|&\leq \mathcal{O}\left( p_n^2\beta_n\left( q_n\right) ^\frac{\gamma-2}{\gamma}\right)   \\
		\left|\frac{1}{k_n}\sum_{i\neq j}^{r_n}\cov\left[Z_{n,i},Z_{n,j}\right]\right|&\leq \mathcal{O}\left( k_n\beta_n\left( q_n\right) ^\frac{\gamma-2}{\gamma}\right)  =o\left( 1\right) .
	\end{align*}
	Thus, 
	\begin{align*}
		\frac{1}{k_n}\left|\var\left[\sum_{i=1}^{r_n}Z_{n,i}\right]-\var\left[\sum_{i=1}^{r_n}Z_{n,i}^*\right]\right|\to 0.
	\end{align*}
	Combined with $P_{Z_{n,i}}=P_{Z_{n,i}^*}$, hence $k_n^{-1}\var\left[\sum_{i=1}^{r_n}Z_{n,i}^*\right]\leq K$ and $\E\left[Z_{n,i}\right]=\E\left[Z_{n,i}^*\right]$,
	the previous lemma yields that $k_n^{-1/2}\sum_{i=1}^{r_n}Z_{n,i}-\E\left[Z_{n,i}\right]$ satisfies 
	a relative CLT if and only if $k_n^{-1/2}\sum_{i=1}^{r_n}Z^*_{n,i}-\E\left[Z^*_{n,i}\right]$ does. 

	Next, the moment assumption together with $P_{Z_{n,i}}=P_{Z_{n,i}^*}$ imply that the sequence $\left( r_n/k_n\right) ^{1/2}Z_{n,i}^*$ 
	satisfies the Lindeberg condition given in \Cref{thm:rel-lind-feller}.
	More specifically, 
	\begin{align*}
		\frac{1}{r_n}\sum_{i=1}^{r_n}\E\left[|\left( r_n/k_n\right) ^{1/2}Z_{n,i}^*|^2\ind_{\{|\left( r_n/k_n\right) ^{1/2}Z_{n,i}|^2>r_n\eps^2\}}\right]
		&=\frac{1}{k_n}\sum_{i=1}^{r_n}\E\left[|Z_{n,i}|^{2}\ind_{\{|Z_{n,i}|^2>k_n\eps^2\}}\right]\\
		&\leq \epsilon^{1-\gamma/2}\frac{1}{k_n^{\gamma/2}}\sum_{i=1}^{r_n}\E\left[|Z_{n,i}|^{\gamma}\right]\\
		&\leq \epsilon^{1-\gamma/2}C\frac{r_np_n^\gamma}{k_n^{\gamma/2}}
	\end{align*}
	for $C=\sup_i\E\left[|X_{n,i}|^\gamma\right]<\infty$ where we used
	\begin{align*}
		\E\left[|Z_{n,1}|^{\gamma}\right]
		&=\|Z_{n,1}\|_\gamma^{\gamma}\\
		&=\left\|\sum_{i=1}^{p_n}X_{n,i}\right\|_\gamma^{\gamma}
		\\
		&\leq
		\left(\sum_{i=1}^{p_n}\|X_{n,i}\|_\gamma\right)^{\gamma}
		\\
		&\leq C p_n^\gamma
	\end{align*}
	and similarly $\E\left[|Z_{n,i}|^{\gamma}\right]\leq C p_n^\gamma$.
	It holds 
	$$\frac{r_np_n^\gamma}{k_n^{\gamma/2}}\leq \frac{r_np_n^\gamma}{\left( r_np_n\right) ^{\gamma/2}}=\frac{p_n^{\gamma/2}}{r_n^{\gamma/2-1}}\leq k_n^{1/2+\delta-\delta\gamma}\to 0.$$
	Thus, 
	$$k_n^{-1/2}\sum_{i=1}^{r_n}\left(Z_{n,i}^*-\E\left[Z_{n,i}^*\right]\right)=r_n^{-1/2}\sum_{i=1}^{r_n}\left( r_n/k_n\right) ^{1/2}\left(Z_{n,i}^*-\E\left[Z_{n,i}^*\right]\right)$$ 
	satisfies a relative CLT which finishes the proof.
\end{proof}

\begin{proof}[Proof of \Cref{thm:multi-rel-clt}]
	Write $$S_n=\frac{1}{\sqrt{k_n}}\sum_{i=1}^{k_n}\left(X_{n,i}-\E\left[X_{n,i}\right]\right)$$
	for the scaled sample average and $\Sigma_n$ for its covariance matrix.
    Let $$N_n\sim \mathcal{N}\left( 0,\Sigma_n\right).$$
	By assumption, $\Sigma_n$ is componentwise a bounded sequence, hence, 
	$N_n$ is relatively compact.
	Thus, it suffices to prove $S_n\rd N_n$.
    By \Cref{cor:rel-cramer-wold}, this is equivalent to 
    $t^TS_n\rd t^TN_n$ for all $t\in \R^d$. 
    Note that     
    $$t^TN_n\sim \mathcal{N}\left( 0,t^T\Sigma_n\right)$$
	and
    $t^TS_n$ is the scaled sample average associated to $t^TX_{n,i}$.
    Accordingly, it suffices to check the conditions of \Cref{thm:rel-univ-clt}
    for $t^TX_{n,i}$:
	The moment and mixing conditions, \ref{A2:rel-univ-clt} and \ref{A3:rel-univ-clt} of \Cref{thm:rel-univ-clt},
	follow by assumption. 
	Lastly, 
	$$k_n^{-1}\sum_{i,j=1}^{k_n}\left|\cov\left[t^TX_{n,i},t^TX_{n,j}\right]\right|\leq t^Tt \max_{l_1,l_2}k_n^{-1}\sum_{i,j=1}^{k_n}\left|\cov\left[X_{n,i}^{(l_1)},X_{n,j}^{(l_2)}\right]\right|\leq t^Tt K$$
	for all $n$.
	This proves \ref{A1:rel-univ-clt} of \Cref{thm:rel-univ-clt} and combined we derive the claim.
\end{proof}

\subsection{Proof of Theorem \ref{thm:multiplier-rel-clt}}\label{sec:multiplier-rel-clt}
We derive \Cref{thm:multiplier-rel-clt} from a more general result.
Consider the general setup of \Cref{sec:asy-tightness}:
Fix some triangular array $X_{n,1},\ldots,X_{n,k_n}$ of random variables with values in a Polish space $\mathcal{X}$.
For each $n\in \N$ let $$\mathcal{F}_n=\{f_{n,t}\colon t\in T\}$$ be a set of measurable functions from $\mathcal{X}$ to $\R$.
Assume that $\cup_{n\in \N}\mathcal{F}_n$ admits a finite envelope $F:\mathcal{X}\to \R$.
\begin{theorem}\label{thm:rel-uniform-clt}
        Assume that for some $\gamma>2$
        \begin{enumerate}
			\item \label{A1:rel-uniform-clt}$\sup_{i,n}\|F(X_{n,i})\|_{\gamma}<\infty$%
			\item \label{A2n:rel-uniform-clt} $\sup_{n\in \N} \max_{m\leq k_n} m^{\rho}\beta_n(m)<\infty$ for some $\rho>2\gamma(\gamma-1)/(\gamma-2)^2$
			\item \label{A3n:rel-uniform-clt}$\int_{0}^{\delta_n}\sqrt{\ln N_{[]}(\epsilon,\mathcal{F}_n,\|\cdot\|_{\gamma,n})}d\epsilon \to 0$ for all $\delta_n\downarrow 0$ and are finite for all $n$. 
        \end{enumerate}
        Denote by $$d_n\left( s,t\right) =\|f_{n,s}-f_{n,t}\|_{\gamma,n}$$
        for $s,t\in T$.
        Assume that there exists a semi-metric $d$ on $T$ such that 
        $$\lim_{n\to \infty}\sup_{d\left( s,t\right) <\delta_n}d_{n}\left( s,t\right) =0$$ for all $\delta_n\downarrow 0$ and 
        $\left( T,d\right) $ is totally bounded.
        Then, the empirical process $\mathbb{G}_n$ 
        defined by $$\G_n\left( t\right) =\frac{1}{\sqrt{k_n}}\sum_{i=1}^{k_n}f_{n,t}\left( X_{n,i}\right) -\E\left[f_{n,t}\left( X_{n,i}\right) \right]$$
        satisfies a relative CLT in $\ell^\infty\left( T\right) $.
 \end{theorem}

 \begin{proof}
    We apply \Cref{thm:asy-tightness-sup-br} to derive relative compactness of 
    $\mathbb{G}_n$ and the existence of an asymptotically tight sequence of tight Borel measurable GPs $N_{\mathbb{G}_n}$ corresponding to $\mathbb{G}_n$. 
    Note that each
    $N_{\mathbb{G}_n}(s)$ is measurable for all $s\in T$. 
    Accordingly, $N_{\mathbb{G}_n}$ is asymptotically measurable, hence, relatively compact
    by \Cref{lem:equ-rel-comp}. 

    According to \Cref{lem:rclt-marginal-rclt}, it remains to 
    prove relative CLTs of the marginals
    $\left( \mathbb{G}_n\left( t_1\right) ,\ldots,\mathbb{G}_n\left( t_d\right) \right) $ for all $d\in \N$, $t_1,\ldots,t_d\in T$. 
    We apply \Cref{thm:multi-rel-clt}
    to the triangular array $Y_{n,1},\ldots,Y_{n,k_n}$ with $Y_{n,k}=\left( f_{n,t_1}\left( X_k\right) ,\ldots,f_{n,t_d}\left( X_k\right) \right).$

	\ref{A2:multi-rel-clt} of \Cref{thm:multi-rel-clt} follows by $\sup_{i,n}\|F(X_{n,i})\|_{\gamma}<\infty$.
	Next, pick $$ \rho^{-1}\frac{\gamma}{\gamma-2}< \alpha<\frac{\gamma-2}{2(\gamma-1)}.$$
	Such $\alpha$ exists since 
	$$\rho^{-1}\frac{\gamma}{\gamma-2}<\frac{\gamma-2}{2(\gamma-1)}.$$
	Then,
	$k_n\beta_{n}(k_n^{\alpha})^{\frac{\gamma-2}{\gamma}}\lesssim k_n^{1-\rho\alpha\frac{\gamma-2}{\gamma}}\to 0$
	since $\frac{\gamma}{\gamma-2}<\rho\alpha$.
	Lastly, the summability condition on the covariances follows by the summability condition on the 
	$\beta$-mixing coefficients (\Cref{rem:sum-implies-cov-bounded-by-moment}). 
    Combined, we obtain the claim.
\end{proof}

\begin{proof}[Proof of \Cref{thm:multiplier-rel-clt}]
	We apply \Cref{thm:rel-uniform-clt}.
		Define the random variables $Y_{n,i}=(X_{n,i},i)\in \mathcal{X}\times \N$.
		Note that $\mathcal{X}\times \N$ is Polish since $\mathcal{X}$ and  $\N$ are.
		Define $T=S\times \Fcal$, $\Fcal_n=\{h_{n,t}\colon t\in T\}$ with
		$$h_{n,(s,f)}:\mathcal{X}\times \N\to \R,(x,k)\mapsto w_{n,k}(s)f(x)$$
		for every $(s,f)\in T$ with $w_{n,k}=0$ for $k>k_n$.
		Note that each $h_{n,(s,f)}$ is measurable.
		Then, the empirical process associated to $Y_{n,i}$ and $T$ is given by $\G_n$. 
	
		Set $K=\max\{\sup_{n,i,x}|w_{n,i}(x)|,\|F\|_{\gamma,\infty}\}$.
		Note that $$F^\prime:\mathcal{X}\times \N\to \R,(x,k)\mapsto KF(x)$$ is an envelope of 
		$\cup_{n\in \N}\Fcal_n$ satisfying condition \ref{A1:rel-uniform-clt} of \Cref{thm:rel-uniform-clt}. 
		Since we have 
		$\sigma(X_{n,i})=\sigma((X_{n,i},i))$, the $\beta$-mixing coefficients w.r.t. $Y_{n,i}$ are
		equal to the $\beta$-mixing coefficients w.r.t. $X_{n,i}$.
		Thus, \ref{A2n:rel-uniform-clt} of \Cref{thm:rel-uniform-clt} are satisfied by assumption. 

		Define the semi-metric $d$ on $T$ by 
		$$d\bigl((s_1,f_1),(s_2,f_2)\bigr)=d^w(s_1,s_2)+\|f_1-f_2\|_{\gamma,\infty}.$$
		By the entropy condition \ref{uniformcltA4} and \ref{A3:weights}
		we derive that $(T,d)$ is totally bounded. 
		By Minkowski's inequality, we get
		\begin{align*}
			d_n\bigl((s_1,f_1),(s_2,f_2)\bigr)&=\|h_{n,(s_1,f_1)}-h_{n,(s_2,f_2)}\|_{\gamma,n}
			\\
			&=\left(\frac{1}{k_n}\sum_{i=1}^{k_n}\|w_{n,i}(s_1)f_1(X_{n,i})-w_{n,i}(s_2)f_2(X_{n,i})\|_\gamma^\gamma\right)^{1/\gamma}
			\\
			&\leq \left(\frac{1}{k_n}\sum_{i=1}^{k_n}\bigl[\|F(X_{n,i})\|_\gamma|w_{n,i}(s_1)-w_{n,i}(s_2)|
			+ \|w_{n,i}\|_\infty\|(f_1-f_2)(X_{n,i})\|_\gamma\bigr]^\gamma\right)^{1/\gamma}
			\\
			&\leq K\left(\frac{1}{k_n}\sum_{i=1}^{k_n}|w_{n,i}(s_1)-w_{n,i}(s_2)|^\gamma\right)^{1/\gamma}
			+K\left(\frac{1}{k_n}\sum_{i=1}^{k_n}\|(f_1-f_2)(X_{n,i})\|_\gamma^\gamma\right)^{1/\gamma}
			\\
			&\leq Kd_n^w(s_1,s_2)+K\|f_1-f_2\|_{\gamma,\infty}
		\end{align*}
		which implies 
		$$\lim_{n\to \infty}\sup_{d(s,t)<\delta_n}d_{n}(s,t)=0$$ for all $\delta_n\downarrow 0$.

		Define $g_{n,s}(i)=w_{n,i}(s)$ and $\Wcal_n=\{g_{n,s}:\N\to \R\colon s\in S \}$.
		Given $f\in \Fcal,$ $s\in S$ and $\eps$-brackets $\underline g \leq g_{n,s} \leq \overline g$ and 
		$\underline f \leq f \leq \overline f$, set 
		the centers $f_c = (\overline f + \underline f)/2$ and 
		$g_c = (\overline g + \underline g)/2$. 
		Then, 
		\begin{align*}
			|fg_{n,s}-f_cg_c|
			&\leq |fg_{n,s}-fg_{c}|+|fg_{c}-f_cg_c|
			\\
			&\leq F\frac{\overline g - \underline g}{2} + K\frac{\overline f - \underline f}{2}.
		\end{align*}
		Thus, we obtain a bracket 
		\begin{align*}
			f_cg_c-\left(F\frac{\overline g - \underline g}{2} + K\frac{\overline f - \underline f}{2}\right)
			\quad
			\leq 
			\quad
			fg_{n,s}
			\quad
			\leq
			\quad
			f_cg_c+\left(F\frac{\overline g - \underline g}{2} + K\frac{\overline f - \underline f}{2}\right)
		\end{align*}
		with 
		\begin{align*}
			\|F(\overline g - \underline g) + K(\overline f - \underline f)\|_{\gamma,n}
			&\leq \|F\|_{\gamma,\infty}\|\overline g - \underline g\|_{\gamma,n} + K\|\overline f - \underline f\|_{\gamma,n}
			\\
			&\leq 2K\eps
		\end{align*}
		hence, an $2K\eps$-bracket for $fg_{n,s}$. 
		This implies 
		\begin{align*}
			N_{[]}\left(2K\epsilon,\mathcal{F}_n,\|\cdot\|_{\gamma,n}\right)\leq N_{[]}\left(\epsilon,\mathcal{F},\|\cdot\|_{\gamma,\infty}\right)N_{[]}\left(\epsilon,S,\|\cdot\|_{\gamma,n}\right).
		\end{align*}
		Since  
		\begin{align*}
			\int_{0}^{\delta_n}\sqrt{\ln N_{\left[\right]}\left( \epsilon,\mathcal{F},\|\cdot\|_{\gamma,\infty}\right)}d\epsilon
			,\quad \int_{0}^{\delta_n}\sqrt{\ln N_{[]}\left( \epsilon,\Wcal_n,d\right)}d\epsilon\to 0
		\end{align*}
		for all $\delta_n\downarrow 0$ by assumption, this implies the entropy condition \ref{A3n:rel-uniform-clt} of \Cref{thm:rel-uniform-clt} and, 
		combined, we derive the claim.
\end{proof}

\begin{proof}[Proof of \Cref{thm:rel-sequ-clt}]\label{prf:rel-sequ-clt}
	In \Cref{thm:multiplier-rel-clt} set 
    $w_{n,i}:[0,1]\to \{0,1\},s\mapsto \ind\{i\leq \lfloor sn\rfloor\}$. 
    Now note that $|w_{n,i}(s)-w_{n,i}(t)|\leq 1$ can only be non-zero 
    for $|\lfloor sk_n \rfloor-\lfloor tk_n \rfloor|\leq k_n|s-t|+1$ many $i$'s.
    Thus, 
    \begin{align*}
        d_n^w(s,t)&=\left(\frac{1}{k_n}\sum_{i=1}^{k_n}|w_{n,i}(s)-w_{n,i}(t)|^{\gamma}\right)^{1/\gamma}
        \\
        &\leq (|s-t|+1/k_n)^{1/\gamma}
    \end{align*}
    and \ref{A2:weights} and \ref{A3:weights} are satisfied for $d^w(s,t)=|s-t|^{1/\gamma}$.
	In combination with $w_{n,i}(s)\leq w_{n,i}(t)$ for all $s\leq t$,
    $$N_{\left[\right]}\left( \epsilon,\Wcal_n,\|\cdot\|_{\gamma,n}\right)\leq \lfloor2/\eps\rfloor^{1/\gamma}$$
	for all $1/k_n\leq \eps\leq 1$.
	Note that $\Wcal$ contains at most $k_n+1$ many distinct functions. 
	For $\eps<1/k_n$, this implies 
	$$N_{\left[\right]}\left( \epsilon,\Wcal_n,\|\cdot\|_{\gamma,n}\right)\leq k_n+1 \leq 2/\eps.$$
	Combined, we obtain 
	$$\ln N_{\left[\right]}\left( \epsilon,\Wcal_n,\|\cdot\|_{\gamma,n}\right)\lesssim \ln(2/\eps)$$
	for all $0<\eps\leq 1$
    which implies \ref{A1:weights}. 
    \Cref{thm:multiplier-rel-clt} gives the claim.
\end{proof}

\subsection{Asymptotic tightness of the multiplier empirical process}
Consider the setup of \Cref{sec:weighted-unif-rclt}.
Let $V_{n,1},\ldots,V_{n,k_n}$ be an additional triangular array of identically distributed random variables.
Define the multiplier empirical process $\G_n^{V} \in \ell^\infty(S\times \Fcal)$ by
\begin{align*}
    \G_n^{V}(s,f) & =\frac{1}{\sqrt{k_n}}\sum_{i=1}^{k_n}V_{n,i}w_{n,i}(s)\bigl(f(X_{n,i})-\E[f(X_{n,i})]\bigr).
\end{align*}
We will derive asymptotic tightness of the $\G_n^{V}$ in terms of bracketing entropy 
conditions with respect to $\Fcal$.

Similar to the proofs of \Cref{thm:asy-tightness-sup-br} and \Cref{thm:multiplier-rel-clt},
the idea is to derive asymptotic equicontinuity by \Cref{thm:chaining}
applied to some class of functions $$(V_{n,i},X_{n,i},i)\mapsto w_{n,i}(s)V_{n,i}f(X_{n,i}).$$
In principle, a direct application of \Cref{thm:rel-uniform-clt}  
establishes asymptotic tightness of $\G_n^V$ in terms of entropy and mixing assumptions on $(V_{n,i},X_{n,i})$. 
The polynomial decay assumption on the mixing coefficients w.r.t. $V_{n,i}$, however, 
is unnecessarily strict. In particular, the summability of covariances 
would be implied (\Cref{rem:sum-implies-cov-bounded-by-moment}) 
rendering the result inapplicable, e.g., for proving multiplier bootstrap consistence 
(see \Cref{sec:bstrap} specifically \Cref{thm:bstrap-univ}).
A relaxation of the mixing decay rate of $V_{n,i}$ is possible by 
applying \Cref{thm:chaining} to a coupled empirical process instead of $\G_n$.

To avoid clutter 
we assume $w_{n,i}=1$, but the argument is similar for the general case.
We use $m_n$, $U_{n,i}$ resp. $U_{n,i}^*$ from the coupling paragraph (\Cref{sec:coupling}) but with $X_{n,i}$ replaced by $(V_{n,i},X_{n,i})$. 
Set $f_V:\R\times \Xcal\to \R, (v,x)\mapsto vf(x)$.
For $f\in \Fcal$ define 
$$f_{V,+,n}:(\R\times \Xcal)^{m_n}\to \R,(v,x)\mapsto \frac{1}{\sqrt{m_n}}\sum_{i=1}^{m_n}f_V(v_i,x_i).$$
Define $\Fcal_{V,+,n}=\{f_{V,+,n}\colon f\in \Fcal\}$, $r_n=k_n/(2m_n)$ and $\G_{n,1},\G_{n,2}\in \ell^{\infty}(\Fcal)$ by 
\begin{align*}
	\G_{n,1}(f)&=\frac{1}{\sqrt{r_n}}\sum_{i=1}^{r_n}f_{V,+,n}(U_{n,2i-1})-\E[f_{+,n}(U_{n,2i-1})],\\
	\G_{n,2}(f)&=\frac{1}{\sqrt{r_n}}\sum_{i=1}^{r_n}f_{V,+,n}(U_{n,2i})-\E[f_{+,n}(U_{n,2i})].
\end{align*}
and $\G_{n,j}^*$ as $\G_{n,j}$ but with $U_{n,i}$ replaced by $U_{n,i}^*$.
Note $\sqrt{2}\G_n^V=\G_{n,1}+\G_{n,2}$.

\begin{lemma}\label{lem:ind-blocks}
	Denote by $\beta_n^{(V,X)}$ the $\beta$-coefficients associated with the triangular array $(V_{n,i},X_{n,i}).$
	If $$\frac{k_n}{m_n}\beta_n^{(V,X)}\left( m_n\right) \to 0,$$ and
	$\G_{n,1}^*,\G_{n,2}^*$ are asymptotically tight
 	then, $\mathbb{G}^V_n$ is asymptotically tight.
\end{lemma}

\begin{proof}
    It holds
	\begin{align*}
		\Pr^*\left(\|\G_{n,1}-\G_{n,1}^*\|_{\Fcal}\neq 0\right)
		&\leq r_n\Pr\left( \exists i\colon U_i^*\neq U_i\right)\\
		&\leq k_n/m_n\beta_n^{(V,X)}(m_n) \to 0
	\end{align*}
	and similarly for $G_{n,2}$. 
	Thus, $\G_{n,1}-\G_{n,1}^*\overset{\Pr^*}{\to}0$ and if $\G_{n,1}^*,\G_{n,2}^*$ are asymptotically tight, so are 
	$\G_{n,1},\G_{n,2}$. 
	Since finite sums of asymptotically tight sequences remain asymptotically tight 
	we obtain that $\G_n^V=2^{-1/2}(\G_{n,1}+\G_{n,2})$ is asymptotically tight.
\end{proof}

\begin{theorem}\label{thm:multiplier-asy-tightness}
	For some $\gamma>2$ and $\alpha<(\gamma-2)/2(\gamma-1)$, assume
    \begin{enumerate}
        \item \label{A1:multiplier-asy-tightness}$\|F\|_{\gamma,\infty}<\infty$.
        \item $\sup_n \|V_{n,1}\|_\gamma<\infty$
        \item $k_n^{1-\alpha}\beta_n^V(k_n^{\alpha}) + k_n^{1-\alpha}\beta_n^X(k_n^{\alpha})\to 0$ where $\beta_n^V$ denotes the $\beta$-coefficients of the $V_{n,i}$.
        \item $\sup_n\sum_{i=1}^{k_n}\beta_n^X(i)^{\frac{\gamma-2}{\gamma}}<\infty$.
        \item \label{A5:multiplier-asy-tightness}$\int_{0}^{\infty}\sqrt{\ln N_{[]}(\epsilon,\mathcal{F},\|\cdot\|_{\gamma,\infty})}d\epsilon <\infty$.
    \end{enumerate}
    Then, $\mathbb{G}_n^V$ is asymptotically tight.
\end{theorem}

\begin{proof}
	First, we may without loss of generality assume $\E[f(X_{n,i})]=0$. 
	To see this define the function class $\Gcal_n=\{g_f\colon f\in\Fcal\}$
	with $$g_f:\Xcal\times \{1,\ldots,k_n\}\to \R,(x,i)\mapsto f(x)-\E[f(X_{n,i})].$$ 
	For $Y_{n,i}=(X_{n,i},i)$ it holds $\E[g_f(Y_{n,i})]=0$. 
	Note that $\Xcal\times \{1,\ldots,k_n\}$ remains Polish and the $\beta$-mixing coefficients 
	with respect to $Y_{n,i}$ and $X_{n,i}$ are equal. 
	Further, given an $\eps$-bracket $[\underline{f},\overline{f}]$ with respect to $\|\cdot\|_{\gamma,n}$,
	define $$u(x,i)=\overline{f}(x)-\E[\underline{f}(X_{n,i})],\quad l(x,i)=\underline{f}-\E[\overline{f}(X_{n,i})].$$
	Clearly, $f \in [\underline{f},\overline{f}]$ implies $g_f\in [l,u]$. Further,
	\begin{align*}
		|u-l|&\leq |\underline{f}-\overline{f}|+\E[|\underline{f}-\overline{f}|]
		\\
		&\leq |\underline{f}-\overline{f}|+\|\underline{f}-\overline{f}\|_{\gamma,n}
		\\
		&\leq |\underline{f}-\overline{f}|+\eps.
	\end{align*}
	Thus, $\|u-l\|_{\gamma,n}\leq 2\eps$ and $N_{[]}(2\eps,\Gcal_n,\|\cdot\|_{\gamma,n})\leq N_{[]}(\eps,\Fcal,\|\cdot\|_{\gamma,n})$.
	Replacing $\Fcal$ by $\Gcal_n$ and $X_{n,i}$ by $Y_{n,i}$ we may assume $\E[f(X_{n,i})]=0$. 

	Next, define the function class $\Fcal_V=\{f_V\colon f\in \Fcal\}$ with $$f_V:\R\times \Xcal\to \R,(v,x)\mapsto vf(x)$$
	and set $Y_{n,i}=(V_{n,i},X_{n,i})$. 
	Then, the empirical process $\G_n^V\in \ell^\infty(\Fcal)$ is the empirical process associated to $\Fcal_V$ and $Y_{n,i}$.
	Again, $\R\times \Xcal$ is Polish and the $\beta$-coefficients $\beta_n^Y$ associated to $Y_{n,i}$ satisfy
	$$\beta_n^Y(m)\leq \beta_n^X(m)+\beta_n^V(m)$$ by Theorem 5.1 (c) of \cite{bradley2005basicproperties}. Set $m_n=k_n^{\alpha}$. 
	Then, $$\frac{k_n}{m_n}\beta_n^Y(m_n)=k_n^{1-\alpha}\beta_n^Y(k_n^{\alpha})\to 0.$$
	We apply \Cref{lem:ind-blocks} to $Y_{n,i}$ and $\Fcal_V$. 
	Thus, it suffices to show that $\G_{n,1}^*$ and $\G_{n,1}^*$ are asymptotically tight. 
	We will prove the statement for $\G_{n,1}^*$. The arguments for $\G_{n,2}^*$ are the same.
	
	Similar to the proof of 
	\Cref{thm:asy-tightness-sup-br}, it suffices to show 
	\begin{align*}
		\lim_{\delta \to 0}\limsup_{n\to \infty}\E^*\|\G_{n,1}^*\|_{\Fcal_{V,\delta}}=0,
	\end{align*}
	where
	$$\Fcal_{V,\delta}=\{f_V-g_V \colon f,g\in \Fcal, \|f-g \|_{\gamma,\infty}<\delta\}.$$ 
	Recall $$\G_{n,1}^*(f)=\frac{1}{\sqrt{r_n}}\sum_{i=1}^{r_n}f_{V,+,n}(U_{n,2i-1}^*)$$
	since $\E[f(X_{n,i})]=0$.
	Note that for fixed $n$, $\G_{n,1}^*$ can be identified with an empirical process 
	$\G_{n,1}^*\in \ell^\infty(\Fcal_{V,+,n})$ associated to the independent 
	random variables $U_{n,2i-1}^*$ and the function class $\Fcal_{V,+,n}=\{f_{V,+,n}\colon f\in \Fcal\}$. 

	Denote by $\|\cdot\|_{2,n,+}$ the $\|\cdot\|_{2,n}$-seminorm on $\Fcal_{V,+,n}$ induced 
	by $U^*_{n,2i-1}$.
	Let $f\in \Fcal$ and an $\eps$-bracket $f\in [\underline{f},\overline{f}]$
	with respect to $\|\cdot\|_{\gamma,n}$ be given. 
	Clearly $f_{V,+,n}\in [\underline{f}_{V,+,n},\overline{f}_{V,+,n}]$.
	It holds
	\begin{align*}
		\|f_{V,+,n}\|_{2,n,+}^2
		&= \frac{1}{r_n}\sum_{i=1}^{r_n}\E[f_{V,+,n}(U_{n,2i-1}^*)^2]
		\\
		&=\frac{1}{r_n}\sum_{i=1}^{r_n}\var[f_{V,+,n}(U_{n,2i-1}^*)]
		&\E[f(X_{n,i})]=0
		\\
		&\lesssim \frac{2}{k_n}\sum_{i,j=1}^{k_n}|\cov[f(X_{n,i}),f(X_{n,j})]
		\\
		&\lesssim \|f\|_{\gamma,n}^2,
	\end{align*}
	by \Cref{rem:sum-implies-cov-bounded-by-moment}.
	This yields 
	$$N_{[]}(C\eps,\Fcal_{V,+,n},\|\cdot\|_{2,n,+})\leq N_{[]}(\eps,\Fcal,\|\cdot\|_{\gamma,n}),$$
	for some constant $C$ independent of $n$ and
	$$\|f_{V,+,n}-g_{V,+,n}\|_{2,n,+}\le C\|f-g\|_{\gamma,n}\le C\delta$$ 
	for all $f_{V}-g_{V} \in \Fcal_{V,\delta}.$
	Without loss of generality, $C=1$.
	Lastly, note that $F_{V,+,n}$ are envelopes for $\Fcal_{V,+,n}$. 

	Putting everything together, and in combination with \Cref{cor:chaining-independent}, we obtain
	\begin{align*}
		\E\|\G_{n,1}\|_{\Fcal_{V,\delta}} &\lesssim \int_0^{2\delta} \sqrt{\ln_+ N_{\left[\right]}(\epsilon) } d \epsilon 
		\\&\quad + \frac{B  \ln_+ N_{\left[\right]}(\delta)  }{\sqrt{r_n}} + \sqrt{r_n}\|F_{V,+,n}\ind\{F_{V,+,n}>B\}\|_{1,n} + \sqrt{r_n} N_{[]}^{-1}(e^{r_n})
	\end{align*}
	with $N_{\left[\right]}(\epsilon)=N_{\left[\right]}(\epsilon,\Fcal,\|\cdot\|_{\gamma,\infty})$.
	Let $B=a_n\sqrt{r_n }$, with $a_n \to 0$ arbitrarily slowly. We obtain 
	$$\frac{B  \ln_+ N_{\left[\right]}(\delta) }{\sqrt{r_n}} = a_n \ln_+ N_{\left[\right]}(\delta) \to 0,$$
	for every fixed $\delta > 0$ due to condition \ref{A5:multiplier-asy-tightness}.
	Further, 
	\begin{align*}
		\sqrt{r_n}\|F_{V,+,n}\ind\{F_{V,+,n}>a_n\sqrt{r_n}\}\|_{1,n}
		&\leq \sqrt{r_n m_n}\|F_V\ind\{F_V>a_n\sqrt{r_n/m_n}\}\|_{1,n}\\
		&= \sqrt{r_n m_n(r_n/m_n)^{1-\gamma}}\|F_V\|_{\gamma,n}^\gamma a_n^{1 - \gamma} \\
		&=\sqrt{r_n m_n(r_n/m_n)^{1-\gamma}}\|V_{n,1}\|_\gamma^\gamma\|F\|_{\gamma,n}^\gamma a_n^{1 - \gamma} \\
		&\lesssim \sqrt{\frac{m_n^{\gamma}}{r_n^{\gamma-2}}} \|F\|_{\gamma,n}^\gamma a_n^{1 - \gamma}\\
		&\lesssim \sqrt{\frac{m_n^{\gamma+\gamma-2}}{k_n^{\gamma-2}}} \|F\|_{\gamma,n}^\gamma a_n^{1 - \gamma}\\
		&=\sqrt{k_n^{2\alpha(\gamma-1)-(\gamma-2)}}\|F\|_{\gamma,n}^\gamma a_n^{1 - \gamma} \\
		&\to 0, 
	\end{align*}
	for $a_n \to 0$ sufficiently slowly, where we used that
	$V_{n,i}$ are identically distributed with $\sup_n\|V_{n,i}\|_\gamma<\infty$ 
	in the third and fourth step, and our condition on $\alpha$ 
	and \ref{A1:multiplier-asy-tightness} in the last.
	Combined, we obtain 
	\begin{align*}
		\limsup_n \E\|\G_{n,1}\|_{\Fcal_{V,\delta}} \lesssim \int_0^\delta \sqrt{\ln_+ N_{\left[\right]}(\epsilon) } d \epsilon 
	\end{align*}
	for all $\delta>0$. Thus
	\begin{align*}
		\lim_{\delta\to 0}\limsup_n \E\|\G_{n,1}\|_{\Fcal_{V,\delta_n}} &= 	\lim_{\delta\to 0}\limsup_n \int_0^{\delta_n} \sqrt{\ln_+ N_{\left[\right]}(\epsilon, \Fcal_n, \| \cdot \|_{\gamma, \infty}) } d \epsilon 
		\\&= 0,
	\end{align*}
	by condition \ref{A5:multiplier-asy-tightness}, completing the proof. 
\end{proof}

\begin{remark}
	A simple calculation reveals that the mixing assumptions
	of \Cref{thm:asy-tightness-sup-br} imply that of 
	\Cref{thm:multiplier-asy-tightness} on $X_{n,i}$.
	On the other hand, the summability condition on $\beta_n^X$ in the latter theorem
	suggests a polynomial decay of the order given in \Cref{thm:asy-tightness-sup-br}. 
	In other words, the two theorems essentially impose the same mixing assumptions on $X_{n,i}$.
	The mixing assumptions on $V_{n,i}$, on the other hand, are substantially weaker. 
\end{remark}

\section{Proofs for the bootstrap}\label{ap:bootstrap}
We will use the notation introduced in \Cref{sec:bstrap} without further mentioning.

\subsection{Proof of Proposition \ref{thm:bstrap-rwc} and a corollary}

\begin{proof}[Proof of \Cref{thm:bstrap-rwc}]
    Since $\G_n$ is relatively compact, so is $\G_n^{\otimes 3}$ \citep[Example 1.4.6]{van2023weak}.
    Both statements 
    can be checked at the level of subsequences and, because $\G_n$ and $\G_n^{\otimes 3}$ are relatively compact, we may assume that 
    $\G_n\dc \G$ converges weakly to some tight Borel law. 
    By independence, we obtain $\G_n^{\otimes 3}\dc \G^{\otimes 3}$.
    Then \ref{A2:bstrap-rwc} is equivalent to 
    $$\bigl(\G_n,\G_n^{(1)},\G_n^{(2)}\bigr)\dc \G^{\otimes3}$$
    in $\ell^{\infty}(\Fcal)^3$.
    Thus, \ref{A2:bstrap-rwc} is equivalent to 
    (a) of Lemma 3.1 in \cite{bucher2019note} and we obtain the claim. 
\end{proof}

\begin{corollary}\label{cor:rclt-bstrap}
    Assume that $\G_n$ satisfies a relative CLT and $\G_n^{(i)}$ 
	are relatively compact. 
    Then, $\G_n^{(1)}$ is a consistent bootstrap scheme if 
    \begin{enumerate}
        \item \label{A2:rclt-bstrap} all marginals of $\bigl(\G_n,\G_n^{(1)},\G_n^{(2)}\bigr)$ satisfy a relative CLT, 
        \item \label{A3:rclt-bstrap} for $n\to \infty$,
        \begin{align*}
            \cov\left[\G_n^{(i)}(s),\G_n^{(i)}(t)\right]-\cov\left[\G_n(s),\G_n(t)\right]
            &\to 0,\\
            \cov\left[\G_n^{(i)}(s),\G_n^{(j)}(t)\right]&\to 0.
        \end{align*}
        for all $i,j=0,1,2$, $i\neq j$, and $\G_n^{(0)}:=\G_n$.
    \end{enumerate}
\end{corollary}

\begin{proof}
	We conclude by  \Cref{thm:bstrap-rwc}, i.e., we prove 
	$$\left(\G_n,\G_n^{(1)},\G_n^{(2)}\right)\rd \G_n^{\otimes 3}.$$
	Since $\G_n$ satisfies a relative CLTs resp. $\G_n^{(i)}$ is relatively compact 
	and satisfies marginal relative CLTs,
	any subsequence of $n$ contains a further subsequence such that 
	both, $\G_n$ and $\G_n^{(i)}$, converge weakly to some tight and measurable GP.
	By \Cref{prop:chara-rc-rwc},
	we may assume that 
	$\G_n$ and $\G_n^{(i)}$ converge weakly to some tight and measurable GP.
    By \ref{A3:rclt-bstrap}, such limiting GPs are equal in distribution, i.e., $\G_n,\G_n^{(i)}\dc N$ with $N$ some tight and measurable GP. 
    Denote by $N^{(i)}$ iid copies of $N$.
	Since $\G_n^{\otimes 3}\dc N^{\otimes 3}$, it suffices to prove
    $$\left(\G_n,\G_n^{(1)},\G_n^{(2)}\right)\dc (N,N^{(1)},N^{(2)}).$$
    By \ref{A2:rclt-bstrap} and \ref{A3:rclt-bstrap},
    \begin{align*}
        \bigl(\G_n(t_1),\ldots,\G_n(t_k),\G_n^{(1)}(t_{k+1}),\ldots,\G_n^{(1)}(t_{k+m}),\G_n^{(2)}(t_{k+m+1}),\ldots,\G_n^{(2)}(t_{m+k+l})\bigr)
    \end{align*}
    converges weakly to 
    $$\bigl(N(t_1),\ldots,N(t_k),N^{(1)}(t_{k+1}),\ldots,N^{(1)}(t_{k+m}),N^{(2)}(t_{k+m+1}),\ldots,N^{(2)}(t_{m+k+l})\bigr).$$
    Thus, 
    $$\left(\G_n,\G_n^{(1)},\G_n^{(2)}\right)\dc N^{\otimes 3}$$
    \citep[Section 1.5 Problem 3.]{van2023weak}
    which finishes the proof.
\end{proof}

\subsection{Proof of Proposition \ref{thm:bstrap-univ}}

\begin{lemma}\label{lem:eps-of-n}
    For every $\epsilon>0$ define $\nu_n(\eps)$ as the maximal 
    natural number such that $$\max_{|i-j|\leq \nu_n(\eps)}\left|\cov[V_{n,i},V_{n,j}]-1\right|\le\epsilon.$$ 
    Assume that for some $\gamma>2$
    \begin{enumerate}
        \item $\sup_{n,i}\E[|F(X_{n,i})|^\gamma]<\infty$,
        \item $k_n\beta_n^X\bigl(\nu_n(\eps)\bigr)^{\frac{\gamma-2}{\gamma}}\to 0$ and 
        \item $k_n^{-1}\sum_{i,j=1}^{k_n}|\cov[f(X_{n,i}),g(X_{n,j})]|\leq K$ for all $f,g\in \Fcal$.
    \end{enumerate}
    Then,
    $$\cov\bigl[\G_n^{(j)}(s,f),\G_n^{(j)}(t,g)\bigr]-\cov\bigl[\G_n(s,f),\G_n(t,g)\bigr]\to 0$$
	for all $(s,f),(t,g)\in S\times \Fcal$.
\end{lemma}
\begin{proof}
    Recall 
    $$|\cov[f(X_{n,i}),g(X_{n,j})]|\lesssim \sup_{n,k}\|F(X_{n,k})\|_\gamma^2\beta_n^X(|i-j|)^{\frac{\gamma-2}{\gamma}}$$
    by Theorem 3 of \cite{doukhan2012mixing}.

    Under the assumptions
    \begin{align*}
       \sum_{\substack{i,j\leq k_n\\|i-j|\leq \nu_n(\eps)}} \bigl|\left[\cov[V_{n,i},V_{n,j}]-1\right]\cov[f(X_{n,i}),g(X_{n,j})]\bigr|
        &\lesssim
        \epsilon 
        \sum_{\substack{i,j\leq k_n\\|i-j|\leq \nu_n(\eps)}}|\cov[f(X_{n,i}),g(X_{n,j})]|
        \\
        &\lesssim  
        k_n\epsilon
        \\
        \\
        k_n^{-1}\sum_{\substack{i,j\leq k_n\\|i-j|>\nu_n(\eps)}}\bigl|\left[\cov[V_{n,i},V_{n,j}]-1\right]\cov[f(X_{n,i}),g(X_{n,j})]\bigr|
        &\leq 
        k_n^{-1}\sum_{\substack{i,j\leq k_n\\|i-j|> \nu_n(\eps)}}|\cov[f(X_{n,i}),g(X_{n,j})]|
        \\
        & \lesssim
        k_n\beta_n^X\bigl(\nu_n(\eps)\bigr)^{\frac{\gamma-2}{\gamma}} \to 0
    \end{align*}
    Thus, 
    \begin{align*}
        \limsup_n k_n^{-1}\sum_{i,j=1}^{k_n}\bigl|\left[\cov[V_{n,i},V_{n,j}]-1\right]\cov[f(X_{n,i}),g(X_{n,j})] \bigr| \lesssim \epsilon,
    \end{align*}
    with constant independent of $\epsilon$.
    Hence, 
    \begin{align*}
        &\limsup_n \left|\cov\bigl[\G_n^{(j)}(s,f),\G_n^{(j)}(t,g)\bigr]-\cov\bigl[\G_n(s,f),\G_n(t,g)\bigr]\right|
        \\
        &\leq
        \sup_{n,i,s}|w_{n,i}(s)|^2\left[\limsup_n\frac{1}{k_n}\sum_{i,j=1}^{k_n}\bigl|\left[\cov[V_{n,i},V_{n,j}]-1\right]\cov[f(X_{n,i}),g(X_{n,j})]\bigr| \right]
        \\
        &\lesssim \epsilon.
    \end{align*}
    Since this is true for all $\epsilon>0$, taking $\epsilon \to 0$ yields the claim.
\end{proof}

\begin{proof}[Proof of \Cref{thm:bstrap-univ}]\label{proof:bstrap-univ}
	We check the conditions of \Cref{cor:rclt-bstrap}.
	By \Cref{thm:multiplier-rel-clt} $\G_n$ satisfies a relative CLT. 
	As in the proof of \Cref{thm:rel-uniform-clt}, there exists some $\alpha^\prime<\frac{\gamma-2}{2(\gamma-1)}$ such that 
	$k_n\beta^X_{n}(k_n^{\alpha^\prime})^{\frac{\gamma-2}{\gamma}}\to 0$.
	Taking the maximum of $\alpha$ and $\alpha^\prime$ we may assume $\alpha^\prime=\alpha$. 
	By \Cref{thm:multiplier-asy-tightness} $\G_n^{(j)}$ is asymptotically tight hence relatively compact. 

	To prove marginal relative CLTs, assume without loss of generality $\E[f(X_{n,i})]=0$ and
	note that the $\beta$-coefficients associated to the triangular arrays
	$$(w_{n,i}(s_1)f_1(X_{n,i}),\ldots, V_{n,i}^{(2)}w_{n,i}(s_k)f_k(X_{n,i}))$$ are bounded above by
    $\beta_n^X+\beta_n^V$ by Theorem 5.1 (c) of \cite{bradley2005basicproperties}.
	In particular, the assumptions on the $\beta$-coefficients'
	decay rate given in
	\Cref{thm:multi-rel-clt} is satisfied for the marginals.
	Next, 
	\begin{align*}
		\E[|V_{n,i}^{(k)}w_{n,i}(s)f(X_{n,i})|^\gamma]&=\E[|V_{n,i}^{(k)}|^\gamma]\E[|w_{n,i}(s)f(X_{n,i})|^\gamma]
	\end{align*}
	by independence of $\mathds{X}_n$ and $\mathds{V}_n^{(k)}$.
	In particular, 
	$$\sup_{n,i}\E[|V_{n,i}^{(k)}w_{n,i}(s)f(X_{n,i})|^\gamma]<\infty$$
	for all $(s,f)\in S\times \Fcal$ by assumption. 
	Note that by the law of total covariances and $\E[f(X_{n,i})]=0$
	\begin{align*}
		\left|\cov\bigl[V^{(k)}_{n,i}f(X_{n,i}),V^{(k)}_{n,j}g(X_{n,j})\bigr]\right|
		&=\left|\E\left[V^{(k)}_{n,i}V^{(k)}_{n,j}\right]\cov\bigl[f(X_{n,i}),g(X_{n,j})\bigr]\right|
		\\
		&\lesssim\bigl|\cov\bigl[f(X_{n,i}),g(X_{n,j})\bigr]\bigr|.
	\end{align*}
	Thus, \ref{A1:multi-rel-clt} of \Cref{thm:multi-rel-clt} follows by 
	the summability of the $\beta$-coefficients of $X_{n,i}$ (\Cref{rem:sum-implies-cov-bounded-by-moment}).
	\Cref{thm:multi-rel-clt} implies marginal relative CLTs of $(\G_n,\G_n^{(1)},\G_n^{(2)})$.

    By the law of total covariances, independence of $\mathds{V}_n^{(k)},\mathds{X}_n$ and
    $\E[f(X_{n,i})],\E[V_{n,i}^{(k)}]=0$ 
    we derive
    \begin{align*}
        \cov\bigl[w_{n,i}(s)f(X_{n,i}),V_{n,j}^{(k)}w_{n,j}(t)g(X_{n,j})\bigr]
        &=0
        \\
        \cov\bigl[V_{n,i}^{(k)}w_{n,i}(s)f(X_{n,i}),V_{n,j}^{(l)}w_{n,j}(t)g(X_{n,j})\bigr]
        &=0
    \end{align*}
    for all $k\neq l$ and $(s,f),(t,g)\in S\times \Fcal$.
    By the above computation we obtain 
    \begin{align*}
        \cov[\G_n^{(k)}(s,f),\G_n^{(l)}(t,g)]=\cov[\G_n(s,f),\G_n^{(k)}(t,g)]=0
    \end{align*}
    for $k\neq l$.
    Lastly, 
    \begin{align*}
        \cov[\G_n^{(1)}(s,f),\G_n^{(1)}(t,g)]-\cov[\G_n(s,f),\G_n(t,g)]\to 0 
    \end{align*}
    by \Cref{lem:eps-of-n}.
    Then, \Cref{cor:rclt-bstrap} provides the claim.
\end{proof}

\subsection{Proof of Proposition \ref{prop:bstr-mean-estim-MSE-rate}}

    Set $\G_n^*$ as $\G_n^{(1)}$ but with $V_{n,i}^{(1)}$ replaced by $V_{n,i}$. 
    Observe that for every $\epsilon > 0$,
    \begin{align*}
        &\quad \var[\wh \G_n^*(s, f) - \G_n^*(s, f)] \\
        &= \var\left[\frac{1}{\sqrt{n}} \sumin V_{n,i} w_{n,i}(s) ( \mu_n(i, f) - \wh \mu_n(i, f))\right]  \\
        &= \frac{1}{n} \sum_{i = 1}^n \sum_{j = 1}^n \cov[V_{n,i}, V_{n, j}] w_{n,i}(s)  w_{n,j}(s) \E\left[(\mu_n(i, f) - \wh \mu_n(i, f))(\mu_n(j, f) - \wh \mu_n(j, f))\right]  \\
        &\lesssim  \nu_n(\epsilon) \sup_{i} \E[(\mu_n(i, f) - \wh \mu_n(i, f))^2] + \epsilon,
    \end{align*}
    using the same arguments as in the proof of \cref{thm:bstrap-univ}. Taking $\epsilon \to 0$ and our assumption on the mean squared error give $ \var[\wh \G_n^*(s, f) - \G_n^*(s, f)] \to 0$.
    Since also $$\E[\wh \G_n^*(s, f) - \G_n^*(s, f)]=\E[\E[\wh \G_n^*(s, f) - \G_n^*(s, f)|\mathds{X}_n]]=0$$
    because of $\E[V_{n,i}]=0$, Markov's inequality yields $\wh \G_n^*(s, f) - \G_n^*(s, f) \overset{\Pr}{\to} 0$ for every $s, f$.
	By \Cref{thm:bstrap-rwc}, $\G_n^*$ is a consistent bootstrap and relatively compact. 
    Thus, $\wh \G_n^* - \G_n^*$ is relatively compact, which implies $\|\wh \G_n^* - \G_n^*\|_{S \times \Fcal}  \overset{\Pr^*}{\to} 0$. 
	Since $\G_n^*$ is a consistent bootstrap for $\G_n$, we conclude by \Cref{thm:bstrap-rwc}. \qed

\subsection{Proof of Proposition \ref{prop:boot-conservative-1}}
Define $\bar \Z_n^*$ as the Gaussian process corresponding to $\bar \G_n^*$ and
    $\bar q_{n, \alpha}^*$ as the $\alpha$-quantile of $\|\bar \Z_n^*\|_{S \times \Fcal} $.
    For every $\epsilon > 0$, \eqref{ass:multiplied-consistent-estimator-of-E} yields
    \begin{align*}
         & \quad \limsup_{n \to \infty} \Pr(\|\wh \G_n^*\|_{S \times \Fcal} \le \bar q_{n, \alpha}^* - \epsilon)                                                                             \\
         & \le  \limsup_{n \to \infty}\Pr(\|\bar \G_n^*\|_{S \times \Fcal} \le \bar q_{n, \alpha}^* ) +   \limsup_{n \to \infty}\Pr(\|\wh \G_n^* -\bar \G_n^*\|_{S \times \Fcal} \ge \epsilon) \\
         & \le \limsup_{n \to \infty} \Pr(\|\bar \G_n^*\|_{S \times \Fcal} \le \bar q_{n, \alpha}^* ),
    \end{align*}
    Taking $\epsilon \to 0$ implies
    \begin{align*}
         & \limsup_{n \to \infty} \Pr(\|\wh \G_n^*\|_{S \times \Fcal} < \bar q_{n, \alpha}^* )  \le \limsup_{n \to \infty} \Pr(\|\bar \G_n^*\|_{S \times \Fcal} \le  \bar q_{n, \alpha}^* ).
    \end{align*}
    Next, observe that \citet[Proposition 1]{giessing2023anticoncentrationsupremagaussianprocesses} and our assumption give
    \begin{align*}
        \liminf_{n \to \infty} \var[\|\bar \Z_n^*\|_{S \times \Fcal }] \gtrsim  \liminf_{n \to \infty} \inf_{(s, f) \in S \times \Fcal }\var[\bar \G_n^*(s, f)] > 0 .
    \end{align*}
    Let $n_{k}$ be any subsequence of such that $\bar \Z_{n_k}^*$ converges weakly to a Gaussian process $\bar \Z^*$ with $\alpha$-quantile $q_{\alpha}$.  Then the above implies that $\bar q_{n_k, \alpha}^* \to q_{\alpha} > 0$, and  $\Pr(\|\bar \Z^*\|_{S \times \Fcal} = q_{\alpha}) = 0$. Thus,  \cref{prop:portmanteau} gives
    \begin{align*}
        \limsup_{n \to \infty} \Pr(\|\bar \G_n^*\|_{S \times \Fcal} \le  \bar q_{n, \alpha}^* )  \le \limsup_{n \to \infty} \Pr(\|\bar \Z_n^*\|_{S \times \Fcal} \le \bar q_{n, \alpha}^*)  = 1 - \alpha.
    \end{align*}
    We have shown that
    \begin{align*}
        \limsup_{n \to \infty} \Pr(\|\wh \G_n^*\|_{S \times \Fcal} <  \bar q_{n, \alpha}^* ) \le 1 - \alpha,
    \end{align*}
    which implies $\wh q_{n, \alpha}^* \ge \bar q_{n, \alpha}^*$ with probability tending to 1.
    This further implies
    \begin{align*}
        \liminf_{n \to \infty} \Pr(\|\G_n\|_{S \times \Fcal} \le \wh q_{n, \alpha}^*)
         & \ge
        \liminf_{n \to \infty} \Pr(\|\G_n\|_{S \times \Fcal} \le \bar q_{n, \alpha}^*) \ge \liminf_{n \to \infty} \Pr(\|\Z_n\|_{S \times \Fcal} \le \bar q_{n, \alpha}^*),
    \end{align*}
    using $\G_n \rd \Z_n$ and a similar continuity argument as above.
    Decompose
    \begin{align*}
        \bar \G_n^*(s, f) - \bar \G_n^*(t, g) = \G_n(s, f) - \G_n(t, g) +  [\bar \G_n^*(s, f) - \G_n(s, f) - \bar \G_n^*(t, g) +  \G_n(t, g)],
    \end{align*}    
    and observe that
    $\G_n(s, f) - \G_n(t, g)$ and $[\bar \G_n^*(s, f) - \G_n(s, f) - \bar \G_n^*(t, g) + \G_n(t, g)]$ are uncorrelated for every $(s, f), (t, g) \in S \times \Fcal$.
    Thus, 
    \begin{align*}
        &\quad \, \var[\Z_n(s, f) - \Z_n(t, g)] \\
        &= \var[\G_n(s, f) - \G_n(t, g)]  \\
        &\le  \var[\G_n(s, f) - \G_n(t, g)]  + \var[\G_n^*(s, f) - \G_n^*(t, g) - \G_n(s, f) +  \G_n(t, g)] \\
        & =  \var[\bar \G_n^*(s, f) - \bar \G_n^*(t, g)] \\
                                      & = \var[\bar \Z_n^*(s, f) - \bar \Z_n^*(t, g)].
    \end{align*}
    Since both $\Z_n(s, f)$ and $\bar \Z_n^*(s, f)$ are asymptotically tight, they are separable for large enough $n$. 
    The version of Fernique's inequality given by \citet[eq.~3.11 and following paragraph]{ledoux1991probability} implies
    \begin{align*}
        \liminf_{n \to \infty} \Pr(\|\Z_n\|_{S \times \Fcal} \le \bar q_{n, \alpha}^*)  \ge   \liminf_{n \to \infty}  \Pr(\|\bar \Z_n^*\|_{S \times \Fcal} \le \bar q_{n, \alpha}^*)=  1-  \alpha.
    \end{align*}
    Altogether, we have shown that
    \begin{align*}
        \liminf_{n \to \infty} \Pr(\| \G_n\|_{S \times \Fcal} \le \wh q_{n, \alpha}^*)   \ge 1 - \alpha,
    \end{align*}
    as claimed. \qed

\subsection{Proof of Proposition \ref{prop:boot-conservative-2}}

Conditions \eqref{ass:multiplied-consistent-estimator-of-E} and \eqref{eq:boot-inconsistent}  imply that $\wh q_{n, \alpha}^* \ge t_n$ with probability tending to 1.
Then every subsequence of the sets $S_{n} = \{z \colon |z| \le t_{n}\}$ converges to $\R$, whose boundary has probability zero under every tight Gaussian law. \cref{prop:portmanteau} and \cref{thm:multiplier-rel-clt} give
 \begin{align*}
       \liminf_{n \to \infty} \Pr(\|\G_n\|_{S \times \Fcal}  \le \wh q_{n, \alpha}^*) \ge \liminf_{n \to \infty} \Pr(\|\G_n\|_{S \times \Fcal} \le t_n)\ge \liminf_{n \to \infty} \Pr(\|\Z_n\|_{S \times \Fcal} \le t_n) = 1,
  \end{align*}
as claimed. \qed

\subsection{A useful lemma}

\begin{lemma} \label{lem:vn-exp-zero}
	Let $V_{n, 1}, \ldots, V_{n, n}$ be a sequence of $m_n$-dependent random variables with $m_n = o(n^{1/2})$, $\E[V_{n, i}] = 0$, $\var[V_{n, i}] = 1$, and $\sup_{i,n} \E[|V_{n, i}|^a] < \infty$ for any $a \in \N$.
	Let $\Fcal_{n}$ be classes of functions satisfying conditions \ref{A1:asy-tightness-sup-br} and \ref{A3n:asy-tightness-sup-br} of \cref{thm:asy-tightness-sup-br} with $\gamma = 1$. Then
	\begin{align*}
			\sup_{t \in T} \left| \frac{1}{n} \sum_{i=1}^{n} V_{n, i} \E[f_{n, t}(X_i)]  \right| & \overset{\Pr^*}{\to} 0,
	\end{align*}
\end{lemma}
\newcommand{\V}{\mathds{V}}
\begin{proof}
	Let $N(\eps)$ be the number of $\epsilon$-brackets  of $\Fcal_n$ with respect to the $\|\cdot\|_{1, n}$-norm.
	As in the proof of \cref{thm:multiplier-rel-clt}, we can construct a $C\eps$-bracketing with respect to the $\|\cdot\|_{1, n}$-norm (induced by $V_{n,i}$) for the class 
	\begin{align*}
			\Gcal_n = \left\{g_{n, t}(v, i) = v \E[f_{n, t}(X_i)] \colon t \in T\right\},
	\end{align*}
	for some $C < \infty$ and size $N(\eps)$.
	Let $\Gcal_{n, k} = [\underline g_{n}^{(k)}, \overline g_{n}^{(k)}]$ be the $k$-th $C\epsilon$-bracket. Recall that
	\begin{align*}
			\P_n g_{n, t} = \frac 1 n \sum_{i=1}^{n} g_{n, t}(V_{n,i}, i) = \frac 1 n \sum_{i=1}^{n} V_{n,i} \E[f_{n, t}(X_i)],
	\end{align*}
	and 
	\begin{align*}
			\P g_{n, t} = \frac 1 n \sum_{i=1}^{n} \E[g_{n, t}(V_{n,i}, i)] = \frac 1 n \sum_{i=1}^{n} \E[V_{n,i}] \E[f_{n, t}(X_i)] = 0.
	\end{align*}
	We thus get
	\begin{align*}
			&\quad \sup_{t \in T}| \P_n g_{n, t}| \\
			&= \sup_{t \in T}| (\P_n - P) g_{n, t}| \\
			& \le \max_{1 \le k \le N(\eps)}  |(\P_n -P)\underline g_{n}^{(k)}| +  \sup_{g \in \Gcal_{n, k} }| (\P_n - P)g - (\P_n - P)\underline g_{n}^{(k)}| \\
			& \le \max_{1 \le k \le N(\eps)}  |(\P_n -P)\underline g_{n}^{(k)}| +  \sup_{g \in \Gcal_{n, k} }| \P_n (g - \underline g_{n}^{(k)})| +  \sup_{g \in \Gcal_{n, k} }| \Pr(g - \underline g_{n}^{(k)})| \\
			& \le \max_{1 \le k \le N(\eps)}  |(\P_n -P)\underline g_{n}^{(k)}| +  | \P_n (\overline  g_{n}^{(k)} - \underline  g_{n}^{(k)})| +  \sup_{g \in \Gcal_{n, k} }| \Pr(g - \underline g_{n}^{(k)})| \\
			& \le \max_{1 \le k \le N(\eps)}  |(\P_n -P)\underline g_{n}^{(k)}| +  | (\P_n - P) (\overline  g_{n}^{(k)} - \underline  g_{n}^{(k)})| +  2\sup_{g \in \Gcal_{n, k} }| \Pr(g - \underline g_{n}^{(k)})| \\
			&\le  3\max_{1 \le k \le N(\eps)}  |(\P_n -P)\underline g_{n}^{(k)}|  +  2C\eps,
	\end{align*}
	where in the last step, we assumed without loss of generality that any upper bound of the brackets also appears as a lower bound.
	By the moment condition on $V_{n, i}$, we have $\max_{1\le i \le n}|V_{n, i} | \le a_n$ with probability tending to 1 for any $a_n \to \infty$ arbitrarily slowly.
	On this event, Bernstein's inequality \cref{lem:bernstein} with $m = m_n + 1$ gives 
	\begin{align*}
		 \E\left[\max_{1 \le k \le N(\eps)}  |(\P_n -P)\underline g_{n}^{(k)}|\right] & \lesssim  \sqrt{\frac{m_n \ln N(\eps)}{n}} + \frac{a_n m_n \ln N(\eps)}{n},
	\end{align*}
	where we used that
	\begin{align*}
				\frac{1}{n} \sum_{i=1}^{n} \sum_{j=1}^{n} \cov[V_{n, i}, V_{n, j}] \E[f_{n, t}(X_i)] \E[f_{n, t}(X_j)] 
			 &\lesssim \frac{1}{n} \sum_{i=1}^{n} \sum_{j=1}^{n} |\cov[V_{n, i}, V_{n, j}]| 
			 \lesssim m_n,
	\end{align*}
	where the constant depends on the envelope of $\Fcal_n$.
	Since $\ln N(\eps) \le C \eps^{-2}$, we can choose $\eps = n^{-1/5}$ to get
	\begin{align*}
			\E\left[\max_{1 \le k \le N(\eps)}  |(\P_n -P)\underline g_{n}^{(k)}|\right] + \eps \to 0. \tag*{\qedhere}
	\end{align*}
\end{proof}

\section{Proofs for the applications}\label{app:applications}

\subsection{Proof of Corollary \ref{prop:kernel1}}

	We apply \Cref{thm:multiplier-rel-clt} with class of weights
	\begin{align*}
			\Wcal_n = \left\{j \mapsto \frac{1}{b}K\left(\frac{j -  sn }{nb}\right) \colon s \in [0, 1]\right\}.
	\end{align*}
	$\Fcal=\{id\}$ and $\gamma=4$. 
	Conditions \ref{uniformcltA1}--\ref{uniformcltA4} on $X_{n,i}$ follow immediately (e.g., \Cref{ex:gamma-rho}).
	Since the kernel is $L$-Lipschitz, we have 
	\begin{align*}
			\sup_j \left| \frac 1 b K\left(\frac{j -  sn  }{nb}\right) - \frac 1 b K\left(\frac{j -  s'n  }{nb}\right)\right| \le \frac{L|s- s'|}{b^2},
	\end{align*}
	which also implies 
	\begin{align*}
		d_n^w(s,t)\leq \frac{L|s- s'|}{b^2}.
	\end{align*}
	Thus, \ref{A2:weights} and \ref{A3:weights} hold.
	Let $s_k = k\epsilon b^2 / L$ for $k = 1, \ldots, N(\eps)$ with $N(\eps) = \lceil (\epsilon b^2 /L)^{-1} \rceil$. 
	Then the functions 
	\begin{align*}
			\underline K_k(j) =\frac 1{b} K\left(\frac{j -  s_kn  }{nb}\right)  - \epsilon/2, \qquad \overline K_k(j) = \frac 1{b} K\left(\frac{j -  s_kn  }{nb}\right) + \epsilon/2,
	\end{align*}
	form an $\epsilon$-bracketing of $\Wcal_n$ with respect to the $\| \cdot \|_{n, \gamma}$-norm. Thus,
	\begin{align*}
	 \int_{0}^{\delta_n}\sqrt{\ln N_{[]}(\epsilon,\Wcal_n,\|\cdot\|_{n, \gamma})}d\epsilon \le  \int_{0}^{\delta_n}\sqrt{-\ln (\epsilon b^2 /L)}d\epsilon \lesssim  \delta_n \sqrt{\log \delta_n^{-1}},
	\end{align*}
	which implies \ref{A3:weights} and we conclude by \Cref{thm:multiplier-rel-clt}.\qed

\subsection{Proof of Corollary \ref{prop:kernel2}}
	We apply \Cref{prop:boot-conservative-2}.
	Pick the class of weights $\Wcal_n$ and functions $\Fcal$ as in \Cref{prop:kernel1}.
	Set 	    
	\begin{align*}
        \bar \mu_b^*(sn) & = \frac{1}{nb} \sum_{i=1}^{n} V_{n, i} K\left(\frac{i - sn}{nb} \right) (X_i - \mu_{b}(i)).
    \end{align*}
	With the notation of \Cref{sec:bstrap} we have $\wh \mu_n(i)=\wh \mu_b(i)$,
	\begin{align*}
		\G_n(s)&=\sqrt{n}\bigl(\wh \mu_b(sn)-\mu_n(sn)\bigr),\\
		\bar \G_n^*(s)&=\sqrt{n}\bar \mu_b^*(sn),\\
		\wh \G_n^*(s)&=\sqrt{n}\wh \mu_b^*(sn).
	\end{align*}
	We already verified the conditions of \Cref{thm:multiplier-rel-clt}. 
	\Cref{cor:block-bstrap} implies the conditions of \Cref{thm:bstrap-univ}.
	In order to verify \eqref{ass:multiplied-consistent-estimator-of-E}, note 
	\begin{align*}
		\bar \G_n^*(s)-\wh \G_n^*(s)
		&=\sqrt{n}\left(\bar \mu_b^*(sn)-\wh \mu_b^*(sn)\right)
		\\
		&=\frac{1}{\sqrt{n}} \sum_{i=1}^{n} V_{n, i} b^{-1}K\left(\frac{i - sn}{nb} \right) (\wh \mu_{b}(i) - \mu_{b}(i)).
	\end{align*}
	Set $Y_{n,i}=\wh \mu_{b}(i) - \mu_{b}(i)\in \R$ and recall 
	$$\max_{i\leq n}\E[Y_{n,i}^2]=\max_{i\leq n}\var[Y_{n,i}]=\mathcal{O}(n^{-1})$$
	by the mixing assumption on $X_i$ and uniform boundedness of the weights. 
	Since the multipliers $V_{n,i}$ are $m_n$-dependent with uniformly bounded variance, we obtain 
	$$\frac{1}{n}\sum_{i,j=1}^{n}|\cov[V_i,V_j]|=\mathcal{O}(m_n).$$
	Thus, 
	\begin{align*}
		\limsup_n \sum_{i=1}^{n} \var[V_{n,i}]\E[Y_{n,i}^2]&\lesssim \limsup_n n \max_{i\leq n}\E[Y_{n,i}^2]<\infty
		\\ 
		\frac{1}{n}\sum_{i,j=1}^{n}|\cov[V_{n,i},V_{n,j}]\E[Y_{n,i}Y_{n,j}]|&\leq \max_{i\leq n}\E[Y_{n,i}^2] \frac{1}{n}\sum_{i,j=1}^{n}|\cov[V_{n,i},V_{n,j}]| 
		\\
		& \lesssim \frac{m_n}{n}\to 0.
	\end{align*}
	\Cref{prop:multiplier-asy-equivalent} yields 
	$$\bar \G_n^*-\wh \G_n^*\overset{\Pr^*}{\to} 0.$$
	Next, by the law of total covariances and $\E[V_{n,i}]=0$ we get
	\begin{align*}
		\var[\bar \G_n^*(s)]
		&=\var[\G_n^*(s)]+\bar \sigma_n^*(s).
	\end{align*}
	By assumption, there exists some $s\in [0,1]$ with
	$\var[\bar \G_n^*(s)]\to \infty$. 
	Lastly, note $\E[\G_n^*(s)]=0$ since $\E[V_{n,i}]=0$. 
	\Cref{thm:multi-rel-clt} gives $\bar \G_n^*(s)/\var[\bar \G_n^*(s)]^{1/2}\dc \mathcal{N}(0,1)$.
	We obtain 
	$$\Pr(\|\bar \G_n^*\|_{[0,1]}>t_n)\to 1$$
	for some $t_n\to \infty$. 
	Applying \Cref{prop:boot-conservative-2} yields the claim. \qed

\subsection{Proof of Corollary \ref{appl:test}}

Observe that
    \begin{align*}
        T_n = \sup_{f \in \Fcal, s \in S} \left|   \frac{1}{\sqrt{n}}\sum_{i=1}^{n} w_{n, i}(s) (f(Z_i) - \E[f(Z_i)])\right|.
    \end{align*}
    The relative CLT (\cref{thm:multiplier-rel-clt}) gives
    \begin{align*}
        \frac{1}{\sqrt{n}}\sum_{i=1}^{n} w_{n, i}(s) (f(Z_i) - \E[f(Z_i)]) & \leftrightarrow_{d} \Z_n(s, f) \quad \text{in } \ell^{\infty}(S \times \Fcal),
    \end{align*}
    where $\{\Z_n(s ,f)\colon f \in \Fcal\}$ is a relatively compact, mean-zero Gaussian process. Under the null hypothesis, $\E[f(Z_i)] = 0$ for all $f \in \Fcal$, and the relative bootstrap CLT (\cref{thm:multiplier-rel-clt} and \cref{cor:block-bstrap}) gives
    \begin{align*}
        \frac{1}{\sqrt{n}}\sum_{i=1}^{n} V_{n, i} w_{n, i}(s) f(Z_i) & \leftrightarrow_{d} \Z_n(s, f) \quad \text{in } \ell^{\infty}(S \times \Fcal),
    \end{align*}
    The relative continuous mapping theorem now implies
    \begin{align*}
        T_n & \leftrightarrow_{d} \sup_{f \in \Fcal, s \in S} |\Z_n (s, f)|  \qquad \text{and} \qquad  T_n^{*}  \leftrightarrow_{d} \sup_{f \in \Fcal, s \in S} |\Z_n (s, f)|,
    \end{align*}
    which proves that $\Pr(T_n > c_n^*(\alpha)) \to \alpha$ under $H_0$.

    Under the alternative, we have
    \begin{align*}
        T_n & = \sup_{f \in \Fcal, s \in S} \left|   \frac{1}{\sqrt{n}}\sum_{i=1}^{n} w_{n, i}(s) (f(Z_i) - \E[f(Z_i)]) +   \frac{1}{\sqrt{n}}\sum_{i=1}^{n} w_{n, i}(s) \E[f(Z_i)]
        \right|                                                      \\
            & \ge \sup_{f \in \Fcal, s \in S}  \left| \frac{1}{\sqrt{n}}\sum_{i=1}^{n} w_{n, i}(s) \E[f(Z_i)] \right| -  \sup_{f \in \Fcal, s \in S} \left|   \frac{1}{\sqrt{n}}\sum_{i=1}^{n} w_{n, i}(s) (f(Z_i) - \E[f(Z_i)]) \right|,
    \end{align*}
    which implies
    \begin{align*}
        T_n / \sqrt{n} \ge \sup_{f \in \Fcal} \left| \frac{1}{n}\sumin w_{n, i}(s) \E[f(Z_i)]\right| \ge \delta > 0,
    \end{align*}
    with probability tending to 1.
    For the bootstrap statistic,
    \begin{align*}
        \frac{ T_n^* }{\sqrt{n}} \le \sup_{f \in \Fcal, s \in S}  \left| \frac{1}{n}\sum_{i=1}^{n} V_{n, i} w_{n, i}(s)(f(Z_i) -  \E[f(Z_i)])  \right| + \sup_{f \in \Fcal, s \in S}  \left| \frac{1}{n}\sum_{i=1}^{n} V_{n, i} w_{n, i}(s) \E[f(Z_i)]  \right|.
    \end{align*}
    The first term on the right converges to 0 in probability by \cref{cor:block-bstrap}, the second by \cref{lem:vn-exp-zero}.
    Thus, $  T_n^* / \sqrt{n} \overset{\Pr^*}{\to} 0$, which implies $c_n^*(\alpha) / \sqrt{n} \to 0$ and the claim follows.
		\qed

\section{Auxiliary results}
\begin{proof}[Proof of \Cref{lem:bounded-sample-paths}]\label{proof:bounded-sample-paths}
    For any $M>0$ and $t_1,\ldots,t_n\in T$ it holds 
	$$\Pr(\max_{1\leq i\leq n}|\G(t_i)|<M)=\prod_{i=1}^{n}\Pr(|\G(t_i)|<M)=[\phi(M)-\phi(-M)]^n$$
	where $\phi$ denotes the standard normal CDF. 
	For $n\to \infty$, i.e., if $T$ is infinite, the right side converges to zero. 
	Thus, $\Pr^*(\|\G\|_T<M)=0$ for all $M$
	In other words, $\G$ does not have bounded samples paths. 
\end{proof}

\subsection{Bracketing numbers under non-stationarity}\label{sec:comp-brackets}

Fix some triangular array $X_{n,1},\ldots,X_{n,k_n}$ of random variables with values in a Polish space $\mathcal{X}$.
Denote by $\mathcal{F}$ a set of measurable functions $f:\mathcal{X}\to \R$.

\begin{lemma}\label{lem:bounded-radon-nik-der}
    Assume that there exists some probability measure $Q$ on $\mathcal{X}$ and a constant $K\in \R$ such
    that $P_{X_{n,i}}\left( A\right) \leq KQ\left( A\right) $ for all $i$ and measurable sets $A$. 
    Then, $$\|f\|_{\gamma,n}\leq \sup_{n\in \N,i\leq k_n}\|f(X_{n,i})\|_{\gamma}\leq K^{1/\gamma}\|f\|_{L_{\gamma}\left( Q\right) }$$ for all $\gamma>0$. Hence, 
    \begin{align*}
        N_{\left[\right]}\left( \epsilon,\mathcal{F},\|\cdot\|_{\gamma,n}\right) \leq N_{\left[\right]}\left( K^{1/\gamma}\epsilon,\mathcal{F},\|\cdot\|_{L_{\gamma}\left( Q\right) }\right) .
    \end{align*}
\end{lemma}

\begin{proof}
    The condition $$P_{X_{n,i}}\left( A\right) \leq KQ\left( A\right) $$ for all measurable sets $A$ 
    is equivalent to $P_{X_{n,i}}$ being absolutely continuous with respect to $Q$ and all Radon-Nikodyn derivatives are bounded by $K$, i.e., 
    $$\frac{\partial P_{X_{n,i}}}{\partial Q}\leq K$$ (up to zero sets of $Q$)  for all $i$.
    For all $i$ it holds 
    \begin{align*}
        \|f(X_{n,i})\|_{\gamma}^\gamma
        &=\int |f(X_{n,i})|^\gamma dP\\
        &=\int |f|^\gamma \frac{\partial P_{X_{n,i}}}{\partial Q} dQ\\
        &\leq \int |f|^\gamma K dQ\\ 
        &=K\|f\|_{L_{\gamma}\left( Q\right) }^\gamma.
    \end{align*}
    This proves the claim.
\end{proof}

\subsection{Multiplier asymptotic equivalence}

\begin{proposition}\label{prop:multiplier-asy-equivalent}
	Suppose $w_{n,i}:S\to \R$ is a uniformly bounded sequence of weights satisfying \ref{A2:weights}--\ref{A3:weights} for some $\gamma\geq 2$.
	Let $V_{n,i}\in \R$ be centered random variables with finite second moment and $Y_{n,1},\ldots, Y_{n,k_n}\in \R$ a triangular array 
	of random variables independent of $V_{n,i}$ satisfying 
	$$\limsup_n \sum_{i=1}^{k_n} \var[V_{n,i}]\E[Y_{n,i}^2]<\infty, \qquad \frac{1}{k_n}\sum_{i,j=1}^{k_n}|\cov[V_{n,i},V_{n,j}]\E[Y_{n,i}Y_{n,j}]|\to 0.$$
	Then, 
	$$\frac{1}{\sqrt{k_n}}\sum_{i=1}^{k_n}w_{n,i}V_{n,i}Y_{n,i}\overset{\Pr^*}{\to}0$$
	in $\ell^\infty(S)$. 
\end{proposition}

\begin{proof}
	For $s,t\in S$
	\begin{align*}
		\frac{1}{\sqrt{k_n}}\sum_{i=1}^{k_n}(w_{n,i}(s)-w_{n,i}(t))V_{n,i}Y_{n,i}
		& \leq \left(\frac{1}{k_n}\sum_{i=1}^{k_n}(w_{n,i}(s)-w_{n,i}(t))^2\right)^{1/2}\left(\sum_{i}^{k_n}V_{n,i}^2Y_{n,i}^2\right)^{1/2}
		\\
		&\leq d_n^w(s,t)\left(\sum_{i}^{k_n}V_{n,i}^2Y_{n,i}^2\right)^{1/2}
	\end{align*}
	by the Cauchy-Schwarz inequality. 
	Accordingly, 
	\begin{align*}
		\E \sup_{d_n^w(s,t)<\eps }\left|\frac{1}{\sqrt{k_n}}\sum_{i=1}^{k_n}(w_{n,i}(s)-w_{n,i}(t))V_{n,i}Y_{n,i}\right|
		& \leq \eps \left(\sum_{i=1}^{k_n}\var\left[V_{n,i}\right]\E\left[Y_{n,i}^2\right]\right)^{1/2}
		\\
		&\lesssim \eps
	\end{align*}
	by Hölder's inequality, independence of $V_{n,i}$ and $Y_{n,i}$ and assumption.
	Conclude that $k_n^{-1/2}\sum_{i=1}^{k_n}w_{n,i}V_{n,i}Y_{n,i}$ is asymptotically 
	uniformly $d^w$-equicontinuous in probability by Markov's inequality and \ref{A2:weights}.

	Next, for any $s\in S$
	\begin{align*}
		\var\left[\frac{1}{\sqrt{k_n}}\sum_{i=1}^{k_n}w_{n,i}(s)V_{n,i}Y_{n,i}\right]
		&=\frac{1}{k_n}\sum_{i,j=1}^{k_n}w_{n,i}(s)w_{n,j}(s)\cov[V_{n,i},V_{n,j}]\E[Y_{n,i}Y_{n,j}]
		\\
		&\lesssim \frac{1}{k_n}\sum_{i,j=1}^{k_n}|\cov[V_{n,i},V_{n,j}]\E[Y_{n,i}Y_{n,j}]|
		\\
		&\to 0
	\end{align*}
	by the law of total covariances and the assumptions. 
	Thus, $k_n^{-1/2}\sum_{i=1}^{k_n}w_{n,i}V_{n,i}Y_{n,i}\in \ell^\infty(S)$ is asymptotically tight
	by Theorem 1.5.7 of \cite{van2023weak} and converges marginally to zero in probability. 
	Conclude that $$\frac{1}{\sqrt{k_n}}\sum_{i=1}^{k_n}w_{n,i}V_{n,i}Y_{n,i}\overset{\Pr^*}{\to}0$$
	in $\ell^\infty(S)$ by (iii) of Lemma 1.10.2 of \cite{van2023weak}.
\end{proof}

\end{document}